\documentclass[10pt]{article}
\usepackage[francais,english]{babel}
\usepackage{amsmath}
\usepackage{amsfonts}
\usepackage{amssymb}
\usepackage{amsthm}
\usepackage{graphics}
\usepackage{amscd}
\usepackage{fullpage}
\usepackage{epsfig}
\usepackage{color}
\usepackage[final]{showkeys}
\usepackage{hyperref}

\newcommand{\Z}{\mathbb{Z}}
\newcommand{\R}{\mathbb{R}}
\newcommand{\N}{\mathbb{N}}

\newcommand{\C}{\mathbb{C}}

\newcommand{\E}{\mathbb{E}}

\newcommand{\mc}{\mathcal}
\newcommand{\mb}{\mathbb}
\newcommand{\mf}{\mathfrak}
\newcommand{\eps}{\varepsilon}
\newcommand{\ind}{{\bf 1}}
\renewcommand{\P}{\mathbb{P}}

\renewcommand{\H}{\mathbb{H}}

\newcommand{\Lap}{\Delta\!}

\DeclareMathOperator{\Vir}{Vir}

\DeclareMathOperator{\End}{End}
\DeclareMathOperator{\Aut}{Aut}
\DeclareMathOperator{\Ker}{Ker}

\DeclareMathOperator{\Tr}{Tr}
\DeclareMathOperator{\Id}{Id}

\DeclareMathOperator{\SLE}{SLE}

\DeclareMathOperator{\Diff}{Diff}

\DeclareMathOperator{\Det}{Det}

\newcommand{\thetaf}[2]{\vartheta\left [\!\!\begin{array}{c}#1\\#2\end{array}\!\!\right]}

\title{SLE and Virasoro representations: localization}
\author{Julien Dub\'edat\footnote{Partially supported by NSF grant DMS-1005749 and the Alfred P. Sloan Foundation.}}
\newtheorem{thm}{Theorem}

\newtheorem{Thm}[thm]{Theorem}

\newtheorem{Prop}[thm]{Proposition}
\newtheorem{Lem}[thm]{Lemma}

\begin{document}
\maketitle
\begin{abstract}
We consider some probabilistic and analytic realizations of Virasoro highest-weight representations. Specifically, we consider measures on paths connecting points marked on the boundary of a (bordered) Riemann surface. These Schramm-Loewner Evolution (SLE)- type measures are constructed by the method of localization in path space. Their partition function (total mass) is the highest-weight vector of a Virasoro representation, and the action is given by Virasoro uniformization.

We review the formalism of Virasoro uniformization, which allows to define a canonical action of Virasoro generators on functions (or sections) on a - suitably extended - Teichm\"uller space. Then we describe the construction of families of measures on paths indexed by marked bordered Riemann surfaces. Finally we relate these two notions by showing that the partition functions of the latter generate a highest-weight representation - the quotient of a reducible Verma module - for the former.

\end{abstract}

\tableofcontents

\section{Introduction}

The Virasoro algebra is the infinite dimensional Lie algebra with generators $(L_n)_{n\in\Z}$, ${\bf c}$ and bracket given by
$$[L_m,L_n]=(m-n)L_{m+n}+\frac{m(m^2-1)}{12}\delta_{n,-m}{\bf c}$$
and $[L_n,{\bf c}]=0$ for all $m,n\in\Z$. A $(c,h)$-highest-weight representation is one generated by a vector $\phi$ such that $L_0\phi=h\phi$, ${\bf c}\phi=c\phi$ and $L_n\phi=0$ if $n>0$. A classical question is to determine the structure of these representations; for a given {\em central charge} $c$ and generic {\em weight} $h$, there is (up to isomorphism) a single such representation, which is irreducible. Exceptions occur for certain values of the weight, $h=h_{r,s}(c)$, $r,s\in\N$, given by the {\em Kac determinant formula} (e.g. \cite{Kac_Bombay,IohKog_Vir}).

It has long been understood - in particular from the work of Cardy \cite{Cardy_BCFT,Cardy_fusion} - that these representations play a crucial role in (Boundary) Conformal Field Theory (BCFT). Conformal Field Theory describes the scaling limit of conformally invariant critical systems in two dimensions, such as percolation or the Ising model. In particular, one can consider the scaling limit of correlators of certain observables (fields). In BCFT, these observables may be located on the boundary of a domain (of course, BCFT also accommodates ``bulk" observables located in the interior of the domain). Associated to such a boundary field is a critical exponent (scaling dimension). To each model corresponds a central charge $c$, and one generally expects that some of the fields of interest have scaling dimension given by the weights coming from the Kac determinant formula. 

\paragraph{Partition functions.} 
Arguably the best understood example is that of the Ising model. Versions of the following discussion appear e.g. in \cite{FriKal,Kont_arbeit,BBK,Dub_Euler}. Consider a sequence of graphs $\Gamma_\eps\subset\eps\Z^2$ which approximate a planar domain $D$. An Ising configuration $\sigma\in\{\pm 1\}^{\Gamma_\eps}$ is an assignment of a $\pm 1$ ``spin" to each interior vertex of $\Gamma_\eps$; we will prescribe the boundary spins below. The energy of a configuration is 
$$E(\sigma)=-J\sum_{(xy)\in E_{\Gamma_\eps}}\sigma(x)\sigma(y)$$
and the partition function is
\begin{equation}\label{eq:boltzmann}
Z=\sum_{\sigma}e^{-\beta E(\sigma)}
\end{equation}
for a well-chosen (critical) inverse temperature $\beta$. Of particular interest to us will be the boundary condition change (bcc) operators, which we now discuss.

Consider boundary points (or rather, boundary edges) $x_1,\dots,x_{2n}=x_0$ at macroscopic distance of each other and in cyclic order. We set the boundary spins to $+1$ (resp. $-1$) between $x_{2k}$ and $x_{2k+1}$ (resp. $x_{2k+1}$ and $x_{2k+2}$). Then the partition function for the Ising model with this boundary condition is denoted symbolically by
$$\langle \psi(x_1)\dots\psi(x_{2n})\rangle$$
(Equivalently, one can consider a fixed $+$ boundary condition and bring disorder variables to the boundary; a disorder variable is the endpoint of a disorder string, across which the Ising couplings are negated \cite{KC}). We consider the small mesh ($\eps\searrow 0$) limit of these correlators. There is an exponential (in the volume) divergence given by the free energy per site of the model; an exponential (in the perimeter) divergence
from the boundary; and a power law divergence coming from the scaling dimension of the bcc operators. The first two are lattice-dependent; the last one is {\em universal} (i.e. independent of the lattice on which the model is defined, provided the couplings are critical). 

One can expect that the boundary contribution can be explicitly compensated for as $\eps\searrow 0$ only in very specific situations, viz. polygonal domains (in the style of \cite{DupDav,Ken_Lap}). In more general situations, there is {\em a priori} no canonical way to regulate a single correlator as $\eps\searrow 0$. However, considering {\em families} of correlators yields non-trivial conditions on their regulated limit, as we now discuss (again in the Ising context, for the sake of concreteness).

Consider four planar domains $A_1,A_2,B_1,B_2$ s.t. $C=A_i\cap B_j$ does not depend on $i,j$; then set $D_{ij}=A_i\cup B_j$ for $i,j\in\{1,2\}$. The $D_{ij}$'s constitute a {\em neutral collection} of domains in the sense of \cite{KontSuh}. Correspondingly, $A^\eps_i$ (resp. $B^\eps_j$) is a mesh $\eps$ approximation of $A_i$ (resp. $B_j$), and $D^\eps_{ij}=A^\eps_i\cup B^\eps_j$. Let $\alpha_i$ (resp. $\beta_j$) be a product of local operators (e.g. spin variables and bcc operators) in $A_i\setminus \bar C$ (resp. $B_j\setminus\bar C$). Then one may consider the ratio
$$\frac{\langle\alpha_1\beta_1\rangle_{D_{11}^\eps}\langle\alpha_2\beta_2\rangle_{D_{22}^\eps}}{\langle\alpha_1\beta_2\rangle_{D_{12}^\eps}\langle\alpha_2\beta_1\rangle_{D_{21}^\eps}}$$
where each term is a sum over configurations of Boltzmann weights (as in \eqref{eq:boltzmann}) modified by the insertions $\alpha,\beta$. On general grounds, one expects such ratios to converge to a finite, universal limit (independent of the choice of discretization, provided that the weights are critical). 

A well-posed problem is to find a collection of ``continuous correlators" $(\langle \gamma\rangle_D)_{\gamma,D}$ indexed by (say, Jordan $C^1$) domains $D$ and products of local operators $\gamma$ such that 
\begin{equation}\label{eq:partfunconsist}
\frac{\langle\alpha_1\beta_1\rangle_{D_{11}^\eps}\langle\alpha_2\beta_2\rangle_{D_{22}^\eps}}{\langle\alpha_1\beta_2\rangle_{D_{12}^\eps}\langle\alpha_2\beta_1\rangle_{D_{21}^\eps}}
\xrightarrow{\eps\rightarrow 0}\frac{\langle\alpha_1\beta_1\rangle_{D_{11}}\langle\alpha_2\beta_2\rangle_{D_{22}}}{\langle\alpha_1\beta_2\rangle_{D_{12}}\langle\alpha_2\beta_1\rangle_{D_{21}}}
\end{equation}
whenever applicable. This is satisfied in the ``ideal" situation where (say, for polygonal domains $D^\eps$ with rational slopes)
$$\left\langle\prod_i\gamma_i\right\rangle_D=\lim_{\eps\searrow 0}\eps^{-\sum_i\Delta_i}\exp\left(-\sum_{x\in D^\eps}f-\sum_{x\in\partial D^\eps}g(\theta_x)-\sum_{x{\rm\ corner}}h(\nu_x)\log(\eps)\right)\left\langle\prod_i \gamma_i\right\rangle_{D^\eps}$$
is well-defined and positive, 
where $\Delta_i$ is the {\em scaling dimension} of the local operators $\gamma_i$'s, $f$ is the free energy per site, $g(\theta)$ the free energy per boundary site with slope $\theta$, $h(\nu_x)$ a contribution per corner $x$ with angle $\nu_x$.

Manifestly, such a collection of continuous partition functions $(\langle\gamma\rangle_D)$ is not uniquely specified by the condition \eqref{eq:partfunconsist}, since e.g.
$$\langle \gamma\rangle'_D\stackrel{def}{=}a_0a_1^{{\rm Vol}(D)}a_2^{{\rm length}(\partial D)}a_3^{{\rm\ number\ of\ insertions}}\langle \gamma\rangle_D$$
is then also a solution of \eqref{eq:partfunconsist} for arbitrary constants $a_0,\dots,a_3$. Yet, \eqref{eq:partfunconsist} is a highly non-trivial system of consistency conditions (even without any insertion). For ``very" solvable models (discrete Gaussian Free Field, Loop-Erased Random Walks/Uniform Spanning Trees and Ising model), large families of partition functions solving \eqref{eq:partfunconsist} can be constructed.

For a (boundary) CFT, if $\phi:D\rightarrow \phi(D)=D'$ is a conformal equivalence, and $\gamma_1,\dots,\gamma_n$ are (non-chiral) local fields, 
$$\frac{\langle\gamma_1(\phi(z_1))\dots\gamma_n(\phi(z_n))\rangle_{\phi(D)}}{\langle 1\rangle_{\phi(D)}}=\prod_i|\phi'(z_i)|^{-\Delta_i}\frac{\langle\gamma_1(z_1)\dots\gamma_n(z_n)\rangle_{D}}{\langle 1\rangle_{D}}.$$

\paragraph{Interfaces and Schramm-Loewner Evolutions.}

Following the introduction of Schramm-Loewner Evolutions (SLEs) by Schramm in \cite{S0}, a more geometric approach to these scaling limits has been investigated. 
Chordal SLE (\cite{S0,RS01,Law_AMS,W1}) is a one-parameter family of probability measures on paths connecting two marked boundary points in a simply-connected domain, indexed by a ``roughness" parameter $\kappa>0$.

In the Ising model, one can consider interfaces separating $+$ clusters from $-$ clusters, shifting the point of view from the correlators to the collection of interfaces. With the boundary conditions described above (viz. with bcc ``operators" at $x_1,\dots,x_{2n}$), one creates $n$ macroscopic interfaces pairing these $2n$ marked boundary points. In that case (at criticality), the interfaces are known to converge to (systems of) SLEs \cite{Sm1,Smi_ICM}. These can be seen as probability measures on ($n$-tuples) of paths, or - better - as positive measures with total mass (or {\em partition function}) $\langle\psi(x_1)\dots\psi(x_{2n})\rangle$.

In order to specify such $\SLE$ partition functions, one can introduce consistency conditions of the type discussed in \ref{eq:partfunconsist}. For example, let $\gamma$ be interface started from a boundary point $x_1$ in a domain $D_\eps$ (an $\eps$-grid approximation of a domain $D$), and aiming at $x_2$. Let $\gamma^\tau$ denote an initial slit of $\gamma$, for instance up to first exit of a small (but fixed) ball centered at $x_1$. Then the probability that $\gamma^\tau$ is a specific lattice path $\gamma_0$ from $x_1$ to $x'_1$ is, essentially by definition,
$$\P_{D_\eps}\{\gamma^\tau=\gamma_0\}=f(\gamma_0)\frac{\langle \psi(x'_1)\psi(x_2)\rangle_{D_\eps\setminus\gamma_0}}{\langle \psi(x_1)\psi(x_2)\rangle_{D_\eps}}$$
where $f(\gamma_0)$ is a product of local factors (depending only on the couplings of the edges crossed by $\gamma_0$).  Let $D'$ be another domain which agrees with $D$ in a neighborhood of $x_1$ (with a marked boundary point $x'_2$ away from $x_1$), and $D'_\eps$ a grid approximation of $D'$. Then
$$\frac{\P_{D'_\eps}\{\gamma^\tau=\gamma_0\}}{\P_{D_\eps}\{\gamma^\tau=\gamma_0\}}=\frac{\langle \psi(x'_1)\psi(x'_2)\rangle_{D'_\eps\setminus\gamma_0}\langle \psi(x_1)\psi(x_2)\rangle_{D_\eps}}
{\langle \psi(x'_1)\psi(x_2)\rangle_{D_\eps\setminus\gamma_0}\langle \psi(x_1)\psi(x'_2)\rangle_{D'_\eps}}$$
By convergence to $\SLE_3$, the lefthand side converges as $\eps\searrow 0$ to a function of $\gamma_0,D,D',x_1,x_2,x'_2$ (at least weakly), which is explicitly determined by the {\em restriction property} of $\SLE$ \cite{LSW3}. 

It is then rather natural (and, by now, quite standard \cite{Dub_Comm,BBK,Dub_dual,Kyt_Virmod,Law_multdef} etc.) to call partition function of $\SLE$ an assignment $(D,x,y)\mapsto{\mc Z}(D,x,y)$ so that (in the chordal case)
\begin{equation}\label{eq:partfunconsistSLE}
\E_{(D',x,y')}(\phi(\gamma^\tau))=\E_{(D,x,y)}\left(\phi(\gamma^\tau)\frac{{\mc Z}(D'\setminus\gamma^\tau,\gamma_\tau,y'){\mc Z}(D,x,y)}{{\mc Z}(D\setminus\gamma^\tau,\gamma_\tau,y'){\mc Z}(D',x,y')}\right)
\end{equation}
where $\gamma$ is the $\SLE$ trace, $\gamma^\tau$ an initial slit in a neighborhood of $x$ common to $D$ and $D'$, $\gamma_\tau$ its tip, and $\phi$ a generic test function). Similarly to \eqref{eq:partfunconsist}, the 4-tuple $(D,D',D\setminus\gamma^\tau,D'\setminus\gamma^\tau)$ constitutes a neutral collection. (Here the situation is somewhat complicated by the fact that the boundary is rough near the tip, which can be remediated by keeping track of a 1-jet of local coordinate at the marked points). Again, the {\em collection} of probability measures $(D,x,y)\mapsto\P_{(D,x,y)}$ gives a non-trivial condition \eqref{eq:partfunconsistSLE} on the {\em collection} of partition functions $(D,x,y)\mapsto{\mc Z}(D,x,y)$. This formalism is particularly useful in the presence of multiple $\SLE$ paths or in non simply-connected topologies. 

\paragraph{Virasoro uniformization.}

In the present article we will not be concerned with scaling limits (i.e. the analysis of discrete correlators or interfaces as the mesh of the underlying lattice goes to zero) but rather will be working directly in the continuum to relate SLE-type measures (on paths or systems of paths) with Virasoro representations. For this purpose we will review Virasoro uniformization (\cite{Kont_Vir,FriKal,Kont_arbeit}) and the method of ``localization in path space" for SLE (\cite{KontSuh,Law_multdef}).

As is well-known, the moduli space of compact surfaces of a given genus has itself a smooth - even complex - structure (keeping track of a Teichm\"uller marking avoids orbifold singularities resulting from surfaces with exceptional symmetries, such as the tori $\C/(\Z+i\Z)$ and $\C/(\Z+e^{\frac{i\pi}3}\Z)$). There are different ways to represent the tangent space to the Teichm\"uller space, corresponding to different ways to think of a (first-order, infinitesimal) deformation of the complex structure. For instance one can deform a compatible Riemannian metric; or deform the $\bar\partial$ operator (Beltrami equation). In the Kodaira-Spencer approach, one starts from a surface $\Sigma$, which by definition is covered by charts with analytic transition maps; the deformation consists in keeping the charts fixed and deforming the transition maps.

In particular one can consider deformations near a marked point $X$; somewhat informally, one can think of cutting out a small disk around that point and gluing it back with a different gluing data (transition map). This deformation is given by a vector field defined in a pointed neighborhood of $X$. In particular, if $z$ is a local coordinate at $X$, one can consider the vector field $-z^{n+1}\partial_z$, $n\in\Z$. This defines a tangent vector to the relevant Teichm\"uller space (of marked surfaces). However this tangent vector depends on the choice of coordinate $z$. Again informally, we can think of this construction as defining a vector field $\ell_n$ on the space of surfaces with a marked point and a marked local coordinate: $(\Sigma,X,z)$. 

It turns out that it is sufficient (and technically easier) to consider a formal local coordinate $\tilde z$ (viz. an element of the completed local ring at $X$ with a first-order zero) rather than a genuine local coordinate $z$. The augmented Teichm\"uller space (the space of marked surfaces of type $(\Sigma,X,\tilde z)$) is the projective limit of a tower of (smooth, finite dimensional) Teichm\"uller spaces. There is a natural notion of smooth functions on this space and the $\ell_n$'s are well-defined as derivations on these smooth functions and represent the Virasoro algebra with $c=0$. This action is geometric, canonical and local (it is defined in terms of a local chart around $X$  independently of other markings, the global geometry of the surface, etc.). A highest-weight vector is a function that has a tensor dependence on the (formal) coordinate $\tilde z$ (i.e. replacing $\tilde z$ with $\tilde z'$ results in multiplying the h.w.-vector by $(\frac{d\tilde z}{d\tilde z'}(X))^{h}$, where $h$ is the weight).

In order to obtain a Virasoro representation with general central charge $c$, one needs to consider sections of a {\em determinant line bundle} (rather than smooth functions) over the augmented Teichm\"uller space. One way to think of these sections is as functionals of a Riemannian metric satisfying a Polyakov anomaly formula (parameterized by $c$). Then one defines - again in a geometric, local fashion - first-order differential operators $L_n$'s that operate on smooth sections of that bundle. This gives a representation of the Virasoro algebra algebra with central charge $c$. This discussion can be carried out for deformations of bordered Riemann surface at a marked boundary point, which is the natural set-up for BCFT.

The next task is to identify sections which are ``interesting" highest-weight vectors; and have a natural probabilistic interpretation, viz. as the partition function (total mass) of a measure on paths (connecting two marked boundary points). In particular we need a collection of measures indexed by the underlying bordered surface.

\paragraph{Localization.}

For our purposes, a crucial result is the {\em restriction property} of $\SLE$ \cite{LSW3}, which quantifies how this measure behaves under a deformation of the (simply-connected) domain (away from the endpoints). This is also the first occurrence of the central charge in $\SLE$ theory. It also enables to define rather easily SLE-type measures and systems of SLEs in more complex geometries.

Specifically, for a bordered surface $\Sigma$ with two marked boundary points $X,Y$, one can consider the path space ${\mc P}(\Sigma,X,Y)$ of simple paths connecting $X$ to $Y$. If $D\subset\Sigma$ is a simply-connected domain which agrees with $\Sigma$ near $X,Y$, ${\mc P}(D,X,Y)$ is a (relatively) open subset of ${\mc P}(\Sigma,X,Y)$; as $D$ varies, one gets a cover of the path space. In order to define a measure $\mu_\Sigma$ on the path space ${\mc P}(\Sigma,X,Y)$, it is enough to define consistent restrictions $\ind_{\gamma\subset D}d\mu_\Sigma(\gamma)$ to the ${\mc P}(D,X,Y)$'s (localization). The advantage is that ${\mc P}(D,X,Y)$ has a natural reference measure: chordal $\SLE$ in $D$.

The problem is thus to define a collection of Radon-Nikodym derivatives $\phi_D^\Sigma$ (derivative of the restriction of $\mu_\Sigma$ to ${\mc P}(D,X,Y)$ w.r.t. chordal $\SLE$ in $D$). In order to verify that these local measures patch up correctly, it is enough to consider the case $D'\subset D$ and compare with the restriction property for simply-connected domains; this gives a (solvable) condition on the densities $(\phi_\Sigma^D)_D$.

This construction assigns a measure $\mu_{\Sigma,X,Y}$ on the path space ${\mc P}(\Sigma,X,Y)$ to a marked bordered surface $(\Sigma,X,Y)$, with a tensor dependence at the endpoints, with weight $h=h_{2,1}$ (one of the special weights appearing in the Kac determinant formula). The partition functions of this collection of measures defines a function on the augmented Teichm\"uller space: 
$${\mc Z}:(\Sigma,X,Y,z,w)\mapsto\|\mu_{(\Sigma,X,Y,z,w)}\|$$
where $z$ (resp. $w$) is a local coordinate at $X$ (resp. $Y$). Two difficulties consist of establishing the finiteness and the smoothness of this partition function. If $c\leq 0$, one can obtain finiteness (for one SLE strand) by comparing with the corresponding measure on the universal cover. Given finiteness, smoothness follows by hypoellipticity arguments. It is rather natural to then consider a section of the determinant bundle ${\mc Z}s$, where $s$ is a reference section expressed e.g. in terms of Laplacian $\zeta$-regularized determinants.

In simply-connected domains, a fundamental property of chordal SLE is the {\em Domain Markov} property. The measures $\mu_\Sigma$ inherit path decomposition identities, which in turn translate into the following {\em null vector equation} for the partition function:
$$\Delta_{2,1}({\mc Z}s)=(L_{-1}^2-\frac 4\kappa L_{-2})({\mc Z}s)=0$$
which is expected on CFT grounds. In terms of the earlier discussion, one can think of bcc operators as corresponding to inserting germs of chordal SLE; the partition function is the correlator of these bcc operators. In algebraic terms, the highest-weight module generated by ${\mc Z}s$ is a quotient of a reducible Verma module. 

The main goal of the article is to define - in what we hope is a concrete and precise manner - the terms of this equation, and then check it (Theorem \ref{Thm:part}); as well as lay the ground for further work, in particular on fusion \cite{Dub_Virfus} and bosonic representations of partition functions. Many of the important ingredients appear in some form in the literature, in particular in \cite{FriKal,Fri_CFTSLE,Kont_arbeit,KontSuh,Dub_SLEGFF,Law_multdef}; earlier realizations of Virasoro representations as ``infinite-dimensional" differential operators in a CFT context appear in \cite{Huang_VOA,BauBer_martVir,BB_CFTSLE,BB_part}. 
The point of view adopted here is an attempt of a middle ground between the more physical/algebraic and the more analytic/probabilistic of these references, with an emphasis on bridging the gap between, in particular, representation-theoretic and probabilistic concepts, an integrating SLE notions within the existing CFT framework. 

In Section 2, we review basic material on Riemann surfaces and discuss Virasoro uniformization (at $c=0$). In Section 3, we discuss loop measures, $\zeta$-regularized determinants,  and anomalies. This is combined in Section 4 to construct the Virasoro action on sections of the determinant bundle (with some technical aspects relegated to appendices). Section 5 describes SLE-type measures obtained by localization in path space and concludes with the null vector equation.

\paragraph{Discussion.}

The construction of $\SLE$-type measures by localization in path space presented here is largely a formal consequence of the restriction property for chordal $\SLE$ \cite{LSW3} and the loop measure \cite{LW} - this is also where the central charge first appears in SLE theory; almost simultaneously, the role of the central charge in SLE in relation with CFT was considered in \cite{BB_CFTSLE} and \cite{FW}. Many (most) of the arguments appear in some form in the literature, for which we now provide a short (and non-exhaustive) guide.

The question of the definition of SLE in more complex geometries has been addressed in many places, from quite a few points of view. In \cite{Bef_the}, the restriction property is used to show in particular that chordal $\SLE_{8/3}$ conditioned on avoiding a hole defines a conformally invariant, Domain Markov process. From the chordal (with marked points) and radial cases, it is fairly natural to use uniformization to try and define $\SLE$ in multiply-connected domains. This has been used in particular in \cite{Dub_ann,Zhan_doubly,Zhan_annprop,BauFri_multiply}. In this line, the main issue is to identify the ``physically relevant" SLEs and define them for all times. As pointed out by Makarov, in more complex topologies one has to preclude new potential pathologies, such as limit cycles.

Makarov and Zhan use the change of coordinate rules for $\SLE$ (e.g. \cite{Dub_Comm,SchWil}) - which follows fairly directly from \cite{LSW3} - to define $\SLE$s in general geometries using local charts. In a slight rephrasing (see e.g. Section 9 in \cite{Dub_Comm}), given $Z$ a suitable partition function, one can construct a Domain Markov (at least for short times) $\SLE$; the drift of the driving process of the said $\SLE$ (when written in a local coordinate $z$) is $-\kappa(\ell_{-1}Z)/Z$. In \cite{Dub_Comm}, in an effort to identify the physically relevant $\SLE$s, a necessary (under smoothness assumptions) condition for reversibility is stated and shown to lead to a differential equation for the ``partition function" $Z$; this differential equation (or rather pair of equations, one per endpoint) is nothing but the null-vector equation $\Delta_{2,1}({\mc Z}s)=0$.

In \cite{FriKal,Fri_CFTSLE,Kont_arbeit}, Friedrich, Kalkkinen and Kontsevich introduce the Virasoro uniformization to SLE and posit the existence of partition functions (that may be thought of as continuous limits of statistical mechanical ones) satisfying the null vector equations; the hypoelliptic nature of the null-vector equation $\Delta_{2,1}({\mc Z}s)$ is also pointed out there. Virasoro uniformization is also closely related to the sewing formalism introduced by Segal \cite{Segal} and elaborated on in particular by Huang in \cite{Huang_VOA} and subsequent work (see Section \ref{sssec:sewing}). Remark however that we chose to follow rather closely the approach to the determinant bundle of \cite{KontSuh,Fri_CFTSLE} rather than the Grassmannian construction of \cite{Segal}. It is also unclear (at least to us) how to accommodate several key features of this article, such as $\SLE$ measures and null-vector equations, in the framework of \cite{Segal,Huang_VOA}.

In \cite{BauBer_martVir,BB_CFTSLE,BB_part}, Bauer and Bernard consider (from a more physical point of view) connections of SLE and CFT, and Virasoro representations involving germs of analytic functions at infinity and highlight the role of partition functions. This was subsequently expanded on by Kyt\"ol\"a, see in particular  \cite{Kyt_Virmod}. In simply-connected domains, the resulting differential Virasoro representation and its action on $\SLE$ (local) martingales have natural interpretations in the framework discussed in this manuscript (see Sections \ref{ssec:BB}, \ref{sssec:locmart}), when written  with a suitable choice of (explicit) coordinates.

A crucial commonality between the earlier approaches of Huang \cite{Huang_VOA} and Bauer-Bernard \cite{BauBer_martVir,BB_CFTSLE,BB_part} and the present work, as sketched in Sections \ref{sssec:sewing} and \ref{ssec:BB}, is the representation of Virasoro generators as infinite-dimensional differential operators.

In \cite{FW}, building on the restriction property, Friedrich and Werner construct Virasoro representations operating on hierarchies of boundary correlators (involving an increasing number of marked points), in relation with the Ward identities.

In \cite{KontSuh}, Kontsevich and Suhov employ localization in path space to define $\SLE$ measures, with a different formalism but - as discussed in Section \ref{ss:comp} - in a manner essentially equivalent to the one presented here. For planar domains, the construction is explained in details in \cite{Law_multdef} (see also \cite{KozLaw_conf}). (In planar domains, points come equipped with a reference local coordinate, given by the embedding in the plane - this is the only nuance between the measures of \cite{Law_multdef} and those discussed here).

In \cite{DoyRivCar}, Doyon, Riva and Cardy consider a representation of the (bulk) stress-energy tensor in central charge 0 based on the SLE restriction property. Conformal Loop Ensemble formulations of the stress-energy tensor are also examined by Doyon in 
\cite{Doy_CLEstress}, in relation with a Virasoro action defined on functionals on spaces of conformal maps \cite{Doy_calc,Doy_hypo,Doy_higher}; see also the survey \cite{Doy_loopCFT}.

The finiteness of the partition function (for $c\leq 0$ and for general $c$ in annuli) is obtained in \cite{Law_multdef}. For annuli, the smoothness of the partition function is obtained by a Feynman-Kac representation in a parabolic set-up in \cite{Zhan_revwhole,Law_multdef}.

\paragraph{Acknowledgments.} It is my pleasure to thank anonymous referees for their insightful and helpful comments.

\section{Riemann surfaces}

In this section, we gather material on Riemann surfaces that will be used later, for the reader's convenience. We will be concerned mostly with bordered Riemann surfaces. The doubling procedure associates a (closed, compact) surface to a bordered surface. Constructions on the doubled surface are very useful in the study of the bordered surface. Hence we begin with elements of the classical theory of compact Riemann surfaces (see e.g. \cite{FarKra, GriHar, ABMNV}).

\subsection{Compact Riemann surfaces}

Let $\Sigma$ be a compact Riemann surface, that is, a smooth compact connected surface equipped with a complex structure. The complex structure is given by an analytic atlas (i.e. a covering by open sets identified with disks, in such a way that the transition maps are analytic). Alternatively, a complex structure is a Riemannian metric modulo the action of smooth functions by Weyl scaling. The complex structure induces an almost complex structure, i.e. a section $J$ of $\End(T\Sigma)$ with $J^2=-\Id$. In complex dimension one, every almost complex structure is integrable (i.e. corresponds to a complex structure), and the two data are equivalent. The almost complex structure gives an orientation.

The homology group $H_1(\Sigma,\Z)$ is a free abelian group of rank $2g$, where $g$ is the genus of $\Sigma$. A canonical basis of $H_1(\Sigma,\Z)$ (not uniquely defined) consists in (classes of) cycles $(A_1,\dots,A_g,B_1,\dots B_g)$ such that the only intersections are between $A_i$ and $B_i$, $i=1\dots g$, with direct orientation for these crossings. Given such a basis, one can identify $H_1(\Sigma,\Z)\simeq\Z^{2g}$. A Teichm\"uller surface is a Riemann surface marked with a canonical basis of $H_1(\Sigma,\Z)$. Equivalently, it is a Riemann surface equipped with a diffeomorphism to a reference smooth surface $\Sigma^s$, given up to isotopy.

A holomorphic bundle over $\Sigma$ is a bundle of complex vector spaces over $\Sigma$ with analytic transition functions. To such a bundle is associated the (invertible) sheaf of its holomorphic sections; we shall not distinguish between the two notions. The structure sheaf of analytic functions is denoted by ${\mc O}$.

The canonical sheaf $K$ is the sheaf of holomorphic $1$-forms (unless mention of the contrary, all forms will be 1-forms). In a local coordinate $z$ defined in an open set $U$, a holomorphic form is written as $\omega=f(z)dz$ where $f$ is holomorphic in $U$. The global sections of $K$ constitute the $g$-dimensional complex vector space $H^0(\Sigma,K)$ of differential forms of the first kind (DFK), or abelian differentials. Given a canonical homology basis, one can find a basis $(v_1,\dots,v_g)$ of $H^0(\Sigma,K)$ dual to the $A$-cycles, i.e. $\int_{A_i}v_j=\delta_{ij}$. The $g\times g$ period matrix $\Pi$ is then defined as $\Pi=(\int_{B_j}v_i)_{1\leq i,j\leq n}$. It is a symmetric matrix with positive definite imaginary part (as follows from Riemann's bilinear relations). The period matrix characterizes a Teichm\"uller surface (Torelli's theorem). If $g\geq 4$, not all $g\times g$ symmetric matrices with positive definite imaginary part are period matrices (the Schottky problem consists in identifying those which are actually period matrices). 

An abelian differential with vanishing $A$- (or $B$-) periods is zero. An abelian differential with imaginary $A$ and $B$ periods is zero.
  
We shall also consider meromorphic forms. The residue at a point of a form is defined invariantly (independently of a choice of local coordinate). The sum of residues of a form is zero. A meromorphic forms with zero residues is a differential form of the second kind (DSK); other meromorphic forms are deemed to be of the third kind. Given any two points $X,Y\in\Sigma$, there exists a meromorphic form with first order poles at $X,Y$ (residues $1,-1$) and holomorphic elsewhere. It is uniquely defined if one requires its $A$-periods to vanish (this requires to fix the $A$-cycles, due to residues). It is also uniquely defined if one requires all periods to be real. By taking a limit $Y\rightarrow X$, given any point $X\in\Sigma$, one can find a DSK with a second order pole at $X$ and regular elsewhere.

The holomorphic tangent bundle is denoted by $T\Sigma$ and can be identified with $K^{-1}$. Its sections can be written as $f(z)\frac{\partial}{\partial z}$ in a local coordinate $z$.

A divisor $D$ is formal finite linear combination with integer coefficients of points of $\Sigma$: $D=\sum n_iP_i$. The sheaf ${\mc O}(-D)$ can be defined as follows: its (holomorphic) sections are meromorphic functions with poles of order at most $-n_i$ at $P_i$ (if $n_i\leq 0$), vanishing at order at least $n_i$ at $P_i$ if $n_i>0$ and regular outside of the support of the divisor. 

\subsection{Theta functions and prime forms}

In this subsection, we collect results on theta functions that we shall use later on (including in planned subsequent work). These will be useful in particular to control the smoothness of various quantities under deformation of the complex structure; and give explicit examples of partition functions. For a complete account, see e.g. \cite{Mum_Tata1, Mum_Tata2, Fay_Theta}. Conventions are
as in \cite{FarKra}.

Let $\Pi$ be a fixed symmetric $g\times g$ complex matrix with positive definite imaginary part (as is the case for period matrices of genus $g$ Riemann surfaces). Such matrices constitute the Siegel half-space ${\mf S}_g$. The Riemann theta function is defined as:
$$\vartheta(z|\Pi)=\sum_{N\in\Z^g}\exp \left(2i\pi(\frac 12 \vphantom{N}^tN\Pi N+ \vphantom{N}^tNz)\right)$$
for $z\in\C^g$. The following transformation property is immediate:
$$\vartheta(z+\Pi N+M|\Pi)=\exp 2i\pi(-\frac 12\vphantom{N}^tN\Pi N-\vphantom{N}^tNz)\vartheta(z|\Pi)$$
for all $M,n\in\Z^g$, and consequently $\vartheta(.|\Pi)$ can be seen as a multivalued function on the complex torus $\C^g/(\Z^g+\Pi\Z^g)$ (it is in particular $\Z^g$-periodic). It is also even.

One can extend the definition to theta functions with characteristic. Let $\eps,\eps'$ be in $\R^g$; one identifies $\C^g$ with $(\R^g)^2$ via $(\eps,\eps')\mapsto \eps'+\Pi\eps$. Then define:
$$\thetaf{2\eps}{2\eps'}(z)=\sum_{N\in\Z^{2g}}\exp
2i\pi\left(\frac 12 \vphantom{N}^t(N+\eps)\Pi (N+\eps)+\vphantom{N}^t(N+\eps)(z+\eps')
\right)$$
keeping now the dependence on $\Pi$ implicit. 
When $2\eps,2\eps'$ have integer coordinates, this function is the first order theta function with integer characteristic $[2\eps\ 2\eps']$. Up to sign, there are $2^{2g}$ such functions, corresponding to the 2-torsion of $\C^g/2(\Z^g+\Pi\Z^g)$. Of these, $2^{g-1}(2^g+1)$ are even (in $z$) and the remaining $2^{g-1}(2^g-1)$ are odd; this depends on the parity of $\vphantom{\eps}^t\eps\eps'$.
If $[\eps\ \eps']$ is an integer characteristic, one has the following transformation properties:
\begin{align*}
\thetaf{\eps}{\eps'}(z+e_k)&=\exp (i\pi\eps_k)\thetaf{\eps}{\eps'}(z)\\
\thetaf{\eps}{\eps'}(z+\Pi e_k)&=\exp i\pi(-2z_k-\Pi_{kk}-\eps'_k)\thetaf{\eps}{\eps'}(z)\\
\thetaf{\eps+2\nu}{\eps'+2\nu'}(z)&=\exp (i\pi \vphantom{\eps}^t\eps\nu')\thetaf{\eps}{\eps'}(z)\\
\end{align*}
where $(e_k)$ is the standard basis of $\Z^g$ and $\nu,\nu'$ are integer valued vectors.

An immediate property is the heat equation:
\begin{align*}
4i\pi\partial_{\Pi_{ii}}\vartheta&=\partial_{z_iz_i}\vartheta\\
2i\pi(\partial_{\Pi_{ij}}+\partial_{\Pi_{ji}})\vartheta&=\partial_{z_iz_j}\vartheta
\end{align*}
relating variations w.r.t. the $z_.$ and $\Pi_.$ variables of a theta function with characteristics $\vartheta$.

Let $\Sigma$ be a Teichm\"uller surface of genus $g$, $(A_1,\dots,A_g,B_1,\dots,B_g)$ the homology basis, $v=(v_1,\dots,v_g)$ the basis of abelian forms dual to the $A$-cycles, $\Pi$ the period matrix. Then there exists an odd integer 
characteristic $[\delta\ \delta']$ which is non singular in the sense that the gradient of the associated theta function at 0 does not vanish (see chapter II in \cite{Fay_Theta}, \cite{Mum_Tata2} from p207). Let us fix such a non singular theta characteristic and denote by $\vartheta$ the associated theta function $\thetaf{\delta}{\delta'}$. Consider the abelian form:
$$\zeta=\sum_{i=1}^g\partial_{z_i}\vartheta(0)v_i.$$
It turns out that the zeroes of this form have even order (in a local coordinate $t$, $\zeta=a(t)dt$ with $a$ holomorphic and with even order zeroes). Thus one can consider $\sqrt\zeta$ as a global section of a holomorphic bundle $L$ such that $L^{\otimes 2}\simeq K$.

At this point one can define the {\em prime form} $E$ on $\Sigma\times\Sigma$:
$$E(x,y)=\frac{\vartheta\left(\int_x^y v\right)}{\sqrt{\zeta(x)}\sqrt{\zeta(y)}}.$$
This depends on the path of integration: changing the path of integration introduces additional periods, hence involves transformation properties of $\vartheta$. One can make it single valued by lifting to the universal cover (i.e. $x,y\in\tilde\Sigma$ the universal cover of $\Sigma$). We reproduce the following properties from \cite{Mum_Tata2}:
\begin{enumerate}
\item $E(x,y)=0$ iff $x$ and $y$ project to the same point in $\Sigma$
\item $E$ vanishes to first order along the diagonal of $\tilde\Sigma\times \tilde\Sigma$
\item $E(x,y)=-E(y,x)$
\item Let $t$ be a local coordinate about $x\in\Sigma$ (i.e. $t(x)=0$) such that $\zeta=dt$; then
$$E(x,y)=\frac{t(x)-t(y)}{\sqrt{dt(x)}\sqrt{dt(y)}}(1+O((t(x)-t(y))^2)).$$
\item $E(x,y)$ is unchanged if $x$ or $y$ is moved along an $A$-period. If $x$ is moved by a $B$-period $\Sigma n_iB_i$ to $x'$,
$$E(x',y)=\pm E(x,y)\exp(-i\pi \vphantom{n}^tn\Pi n+2i\pi\vphantom{n}^tn\int_x^yv).$$
If $y$ is moved to $y'$ along the same $B$ period:
$$E(x,y)=\pm E(x,y)\exp(-i\pi \vphantom{n}^tn\Pi n-2i\pi\vphantom{n}^tn\int_x^yv).$$
\end{enumerate}
(The $\pm$ sign is kept undetermined in order to circumvent a discussion of half-order differentials, and is unimportant for our purposes.)

The prime form does not depend on the choice of non singular odd characteristic (as is easily seen from its vanishing properties), and has a simple dependence on the choice of homology basis. 

Various meromorphic sections can be reconstructed from the prime form $E$. In particular, as noted earlier, for $a,b\in\Sigma$, there is a unique meromorphic 1-form $\omega_{a-b}$ on $\Sigma$ which is regular except at $a,b$, where it has simple poles with residues 1,-1 respectively, and has vanishing $A$-periods. This form can be written as:
$$\omega_{a-b}(x)=d_x\log\frac{E(x,a)}{E(x,b)}.$$
Note that although $E$ is not single-valued, $\omega_{a-b}$ is well defined.

Similarly, the ``fundamental 2-form" on $\Sigma\times\Sigma$, expressed as:
$$\omega(x,y)=d_xd_y\log E(x,y)$$
is well defined and symmetric in $x,y$. In a local coordinate $t$, it has the expansion:
$$\omega(x,y)=\left(\frac 1{(t(x)-t(y))^2}+(reg)\right)dt(x)dt(y)$$
near the diagonal, where $reg$ is biholomorphic. Moreover, integrating the variable $x$ along an $A$-period, the resulting 1-form (in $y$) is zero: $\int_{A_i}\omega(.,y)=0$. Along a $B$-period:
$$\int_{B_j}\omega(.,y)=2i\pi v_j(y).$$
Also, $\int_a^b\omega(.,y)=\omega_{b-a}(y)$, the integral being taken on a path that does not intersect cycles of the homology basis. Hence $\int_{B_j}\omega_{b-a}=2i\pi\int_a^b v_j$. The differential form:
$$\Omega_{b-a}=\omega_{b-a}-2i\pi \vphantom{v}^tv(\Im\Pi)^{-1} \Im\int_a^b v$$
has residues $-1,1$ at $a,b$ and pure imaginary ($A$ and $B$) periods (these being defined modulo $2i\pi\Z$). It is uniquely defined by these properties.

For $a\in\Sigma$, there is a unique (up to multiplicative constant) meromorphic form $\eta_a$ with double pole at $a$, regular elsewhere and with vanishing $A$-periods. It can be expressed as: $\omega(x,a)/dt(a)$, where $t$ is a local coordinate at $a$.

This can be used to express variations of different quantities under a variation of the surface $\Sigma$ in terms of $\vartheta$ and its derivatives, in particular in conjunction with the heat equation. 

The Bergman connection $B_p$ evaluated at $p\in\Sigma$ w.r.t. local coordinate $z$ is defined from the following expansions for $x,y$ near $p$ in $\Sigma$:
\begin{align*}
E(x,y)\sqrt{dz(x)dz(y)}&=(y-x)(1-\frac{B_p}{12}(z(y)-z(x))^2+O((z(y)-z(x))^3))\\
\omega(x,y)&=\left(\frac 1{(z(y)-z(x))^2}+\frac{B_p}6+O(z(y)-z(x))\right)dz(x)dz(y)\\
\eta_p(x)&=\left(\frac 1{(z(x)-z(p))^2}+\frac{B_p}6+O(z(x)-z(p))\right)dz(x)
\end{align*}
From this it is immediate that $B_p$ depends on the local coordinate $z$ as a Schwarzian connection:
\begin{align*}
\left(\frac 1{z^2}+\frac{B(z)}6+\cdots\right)dz
&=\frac{dz'}{dz}\left(\frac 1{z'^2}+\frac{B(z')}6+\cdots\right)dz'\\
&=\left(\frac 1{z^2}\left(1+\frac{z}2\frac{\partial_z^2z'}{\partial_zz'}+\frac{z^2}6\frac{\partial_z^3z'}{\partial_zz'}
\right)^{-2}+\frac{B(z')}6(\partial_zz')^2+\cdots\right)\left(1+z\frac{\partial_z^2z'}{\partial_zz'}+\frac{z^2}2\frac{\partial_z^3z'}{\partial_zz'}\right)dz
\end{align*}
where $z,z'$ are local coordinates at $p$ ($z(p)=z'(p)=0$), which implies:
$$B(z)=B(z')\left(\frac{dz'}{dz}\right)^2+\{z';z\}$$
where 
$$\{z';z\}=\partial_z^3z'/\partial_zz'-3/2(\partial_z^2z'/\partial_zz')^2$$
is the Schwarzian derivative. This is because the difference of the sides in the equation is an abelian form with vanishing $A$-periods, hence 0. The difference of two Schwarzian connections is a quadratic differential.

Let us illustrate these various concepts in the genus 1 case: let $\Sigma=\C/(\Z+\tau\Z)$ be an elliptic curve ($\Im\tau>0$). There is a unique odd integer characteristic theta function $\thetaf{1}{1}$, which we denote by $\theta$ (notations as in \cite{Cha}). Then $\zeta=\theta'(0)dz$, $E(x,y)=\theta(y-x)/\theta'(0)\sqrt{dxdy}$. The fundamental 2-form can be expressed in terms of the elliptic function $\wp$ (Weierstrass $\wp$ function):
$$\omega(x,y)=(\wp(y-x)+2\eta_1)dz(x)dz(y)$$
with $2\eta_1=-\int_A\wp(z)dz$, $2\eta_2=-\int_B\wp(z)dz$, where $A,B$ is the usual homology basis. Note that $\int_A\omega(x,.)=0$ and $\int_B\omega(x,.)=(2\eta_1\tau-2\eta_2)dz(x)=2i\pi dz(x)$, in agreement with the Legendre relation. The Bergman connection (in the flat coordinate $z$) is $B_x=-2\frac{\theta'''}{\theta'}(0)$ at any $x\in\Sigma$. The dependence on the modulus $\tau$ can be made more explicit, e.g. in terms of Dedekind's $\eta$ function: 
$$\theta'(0|\tau)=2\pi\eta^3(\tau)$$
and consequently
$$\frac{\theta'''}{\theta'}(0|\tau)=4i\pi\partial_\tau\log \theta'(0|\tau)=12i\pi\partial_\tau\log\eta(\tau)$$
using the heat equation.

\subsection{Bordered surfaces}

It will be convenient later on to use bordered Riemann surfaces. These are classically studied by considering their (Schottky) double, which are compact Riemann surfaces (see e.g. chapter VI in \cite{Fay_Theta}).

A bordered Riemann surface is modelled locally either on the unit disk (for interior points) or on the semidisk $D^+=\{z:\Im z\geq 0, |z|<1\}$ for boundary points; transition maps are analytic. 

We shall consider only surfaces with a finite number of boundary components (diffeomorphic to circles); these boundary curves are positively oriented, that is, the surface lies to their lefthand side.

The Schwarz reflection principle implies that a continuous analytic function $f$ on, say, the semidisk $D^+$ which is real on the boundary extends to an analytic function on the semidisk via $f(\bar z)=\overline{f(z)}$.

Let $\Sigma$ be a bordered Riemann surface of genus $\rho$ with $n$ boundary components. The double $\hat\Sigma$ is the compact Riemann surface obtained by gluing $\Sigma$ and a conjugate copy of $\Sigma$ along their boundaries. It carries an antiholomorphic involution $\iota$ whose fixed points are the boundary points of $\Sigma$. For instance, if $z$ is a local coordinate at $x\in\Sigma\subset\hat\Sigma$, $\overline{z\circ\iota}$ is a local coordinate at $\iota x$. The involution operates similarly on all natural holomorphic structures (differential forms, \dots). If $\omega$ is a tensor, we denote simply $\iota\omega$ for $\iota^*\omega$ (without ambiguity as $\iota$ is an involution). The double has genus $g=2\rho+n-1$.

One can choose a homology basis on $\hat\Sigma$ adapted to the double structure as follows. Let $A_1,B_1,\dots,A_\rho,B_\rho$ be pairs of cycles around the ``handles'' of $\Sigma$; the oriented boundary components of $\Sigma$ are 
$\Gamma_0,\dots,\Gamma_\rho$. Let $A'_i=\iota A_i$, $B'_i=-\iota B_i$ ($\iota$ is antiholomorphic, so the minus sign is needed to preserve the direct orientation of the crossing). One can take a canonical basis of $H_1(\hat\Sigma,\Z)$ in which the $g=(2\rho+n-1)$  $A$-cycles are 
$$A_1,\dots,A_\rho,A_{\rho+1},\dots,A_{\rho+n-1},A'_1,\dots,A'_\rho$$
where $A_{\rho+i}=\Gamma_i$, $i=1,\dots,n-1$. The cycle $B_{\rho+i}$ starts at a point on $\Gamma_i$, travels on $\Sigma$ to a point on $\Gamma_0$, and comes back symmetrically on $\iota\Sigma$, without intersecting other cycles. Note that the set of $A$-cycles is preserved by $\iota$. It follows that if $(v_1,\dots,v_{\rho+1},\dots,v'_1,\dots)$ is the dual basis of $H^0(\hat\Sigma,K)$, one gets: $\overline{\iota v_i}=v'_i$, $\overline{\iota v_{\rho+i}}=v_{\rho+i}$. Similarly, the period matrix $\Pi$ has symmetries, in particular $(\Pi_{\rho+i,\rho+j})_{1\leq i,j\leq n-1}$ is pure imaginary. In what follows, the same basis will be denoted $(v_1,\dots,v_g)$.

In a similar way, one gets: $\omega_{\iota a-\iota b}=\overline{\iota\omega_{a-b}}$, $\eta_{\iota a}=\overline{\iota\eta_a}$, $\omega=\overline{\iota\omega}$ (fundamental 2-form), and $E(\iota x,\iota y)=\overline{E(x,y)}$ (see \cite{Fay_Theta}, Cor. 6.12).
 
\paragraph{Potential Theory.}  
Harmonic invariants on the bordered surface $\Sigma$ can be expressed in terms of holomorphic invariants on $\hat\Sigma$ (and ultimately in terms of the prime form or theta functions).

The two basic potential theory problems on $\Sigma$ are the (a) the Poisson  problem: For a given $(1,1)$-form written locally as $hdz\wedge d\bar z$, find $f$ vanishing on the boundary s.t.
$$\partial\bar\partial f=f_{z\bar z}dz\wedge d\bar z=hdz\wedge d\bar z$$
and (b) the Dirichlet boundary value problem: given a continuous function $f_0$ on the boundary $\partial\Sigma$, find a continuous extension $f$ to $\Sigma$ which is harmonic there.

The Green's function $G(x,y)$ on $\Sigma$ (with Dirichlet boundary conditions on $\partial\Sigma$) depends only on the conformal structure. It can be expressed as:
\begin{align*}
2\pi G(x,y)&=\frac 12\int_{\iota x}^x\Omega_{\iota y-y}\\
&=-\frac 12\int_x^{\iota x}\int_y^{\iota y}\omega+\pi\sum_{j,k=1}^g(\Im\Pi)^{-1}_{jk}\Im(\int_x^{\iota x}v_j)\Im(\int_y^{\iota y}v_k)
\end{align*}
where $\Omega_{b-a}$ is the meromorphic form with residues $-1,1$ at $a,b$ and purely imaginary periods, and $\omega(x,y)=d_xd_y\log(E(x,y))$ is the fundamental 2-form. 

For a fixed $y\in\Sigma$, $G(.,y)$ vanishes on the boundary, is harmonic on $\Sigma\setminus\{y\}$, and has an expansion near $y$:
$$G(x,y)=\frac 1{2\pi}\log|x-y|+(reg)$$
where $(reg)$ is continuous (the leading part does not depend on the choice of coordinates). These properties characterize uniquely $G$. Setting
$$f(x)=i\int_\Sigma G(x,y)h(y)dy\wedge d\bar y$$
solves the Poisson problem.

Let $p\in\partial\Sigma$. Considering $\tilde\eta_p$ the differential form of the second kind with the same divisor and meromorphic part at $p$ as $\eta_p$, but with the condition that all its periods are real. It is defined w.r.t a local coordinate $z$ which is chosen real and increasing at $p$ along the (oriented) boundary. Let $h_p(q)=\pi^{-1}\Im(\int_{p_0}^q\tilde\eta_p)$ for $q\in\Sigma$, where $p_0$ is a base point on $\partial\Sigma$. Then $h_p$ is a harmonic function vanishing on $\partial\Sigma\setminus\{p\}$; it is the unique such function with expansion $h_p(q)=\pi^{-1}\Im(1/(z(q)-z(p)))+O(1)$ near $p$. It is thus identified as the Poisson kernel $P_\Sigma(.,p)=h_p(.)$ (w.r.t. the length element $dz$ at $p$) in the following sense: if $f_0$ is continuous on $\partial\Sigma$, setting 
$$f(q)=\int_{\partial\Sigma}f_0(p)P_\Sigma(q,p)dz(p)$$
solves the Dirichlet boundary value problem. 
 
Another classical operator is the Dirichlet-to-Neumann operator, which maps $f_0\in C^\infty(\partial\Sigma)$ to the normal derivative on $\partial\Sigma$ of its harmonic extension to $\Sigma$. This defines a singular integral operator on $\partial\Sigma$ with kernel (also called the Poisson excursion kernel) given by 
\begin{equation}\label{eq:Poissonexcdef}
(q,p)\mapsto H(p,q)=\pi^{-1}\Re(\tilde\eta_p(q))
\end{equation}
which is naturally defined as a $1$-form in $p$ and $q$ and is easily seen to be symmetric (the Poisson kernel can be realized by taking the normal derivative w.r.t. one argument of the Green kernel on the boundary; and the Poisson excursion kernel by taking the normal derivative w.r.t. both variables on the boundary).
 
Define $S_p$ w.r.t. the coordinate $z$ by:
\begin{equation}\label{Sconn}
\tilde\eta_p(x)=\left(\frac 1{(z(x)-z(p))^2}+\frac{S_p}6+O(z(x)-z(p))\right)dz(x)
\end{equation}
Since $\int_{A_j}\eta_p=0$, $\int_{B_j}\eta_p=2i\pi (v_j/dz)(p)$, one gets:
$$\tilde\eta_p=\eta_p-2\pi\sum_{jk} v_j(\Im\Pi)^{-1}_{jk}\Re(v_k/dz)(p)$$
from which follows (evaluated w.r.t. $z$):
$$S_p=B_p-12\pi\vphantom{w}^tw(\Im\Pi)^{-1}w$$
where $w=(\Re(v_k/dz)(p))_k$, using the fact that $S_p(z)$ is real ($z$ is real along the boundary). If $z'$ is another such coordinate, one has the Schwarzian connection identity:
\begin{equation}\label{eq:Sconndef}
S(z)=S(z')\left(\frac{dz'}{dz}\right)^2+\{z';z\}
\end{equation}
This also has a simple expression in terms of the {\em bubble measure} of \cite{LW} and the Dirichlet-to-Neumann operator.

\subsection{Teichm\"uller space, deformations, and Virasoro uniformization}

\subsubsection{Teichm\"uller space}

Let $\Sigma^s$ be a smooth compact oriented surface of genus $g$. Any Riemann surface of genus $g$ is diffeomorphic to $\Sigma$. A Teichm\"uller surface is a Riemann surface equipped with a standard homology basis. Two Teichm\"uller surfaces are equivalent if there is a conformal isomorphism between them compatible with the marking. The Teichm\"uller space ${\mc T}_g$ is the space of equivalence classes of equivalence of genus $g$ Teichm\"uller curves.

There is a natural complex structure on ${\mc T}_g$ given by the following prescription: if $\pi:\Xi\rightarrow B$ is an analytic submersion ($B$ a small polydisk) where the fibers $\pi^{-1}(b)$, $b\in B$, are genus $g$ surfaces (with continuous Teichm\"ulcer marking), then the map $b\mapsto [\pi^{-1}(b)]\in{\mc T}_g$ is analytic.
This can be done by considering instead classes of equivalence of quasi-conformal maps and using the Ahlfors-Bers result on solving the Beltrami equation analytically in the Beltrami differential (see e.g. \cite{Gardiner} and references therein). More precisely, there exists a universal Teichm\"uller curve ${\mc C}_g$, that is a holomorphic family of Teichm\"uller curves parameterized by the Teichm\"uller space: ${\mc C}_g\rightarrow{\mc T}_g$.

A complex structure on $\Sigma^s$ is given by an analytic atlas. With a partition  of unity, one can construct a Riemannian metric $g$ on $\Sigma^s$ compatible with this complex structure; that is, if $z=x+iy$ is a local analytic coordinate,
$$g=e^{2\sigma}(dx\otimes dx+ dy\otimes dy)$$
locally. Conversely, given $g$, one can find locally ``isothermal" coordinates $x,y$ such that this holds, and hence recover a complex structure. The metrics $g$ and $e^{2\sigma} g$ yield the same complex structure (``Weyl scaling"). Also,
if $\phi$ is an orientation preserving diffeomorphism of $\Sigma^s$, $(\Sigma^s,\phi^*g)$ and $(\Sigma^s,g)$ are equivalent as Riemannian surfaces, {\em a fortiori} as Riemann surfaces. Moreover, if $\phi\in \Diff_0(\Sigma^s)$ the connected component of the identity in the group of diffeomorphisms of $\Sigma
^s$ (viz. $\phi$ is isotopic to the identity), $(\Sigma^s,\phi^*g)$ and $(\Sigma^s,g)$ are equivalent as Teichm\"uller surfaces. So if ${\rm Met}$ is the (convex) set of smooth Riemannian metrics on $\Sigma^s$, and $C^\infty(\Sigma^s)$ operates on ${\rm Met}$ by $\sigma.g=e^{2\sigma} g$, then the Teichm\"uller space is the space of orbits:
$${\rm Met}/(C^\infty(\Sigma^s)\rtimes \Diff_0(\Sigma^s))$$

\subsubsection{Kodaira-Spencer deformation}

\paragraph{Compact surfaces.} 
We proceed with the Kodaira-Spencer description of the tangent space to the Teichm\"uller space at a surface $\Sigma$ (see \cite{Kodaira}). Let $(\Sigma_t)_{-\eps<t<\eps}$ be a smooth family of compact Riemann surfaces (i.e. without boundary), the lift of a smooth path on the Teichm\"uller space with $\Sigma_0=\Sigma$. For instance, one can fix an underlying differentiable manifold $\Sigma^s$ and consider a smooth family of complex structures on $\Sigma^s$ (e.g. given by a smooth family of Riemannian metrics, or by a smooth family of almost complex structures). 

Let ${\mf U}=\{U_i\}$ be a locally finite covering of $\Sigma$ by analytic disks. Then $\Sigma_t$ is covered by $\{U_i^t\}$ ($U_i^t$ is the open set $U_i$ equipped with the complex structure inherited from that of $\Sigma_t$). Let $\theta_i(t):U_i\rightarrow U_i^t$ be a conformal equivalence (w.r.t. the complex structures on $\Sigma$, $\Sigma_t$ respectively), smooth in $t$, $\theta_i(0)=\Id_{U_i}$. Then on $U_{ij}=U_i\cap U_j$, $\theta_i(t)^{-1}\circ\theta_j(t)$ is analytic (for any $z\in U_{ij}$, it is defined for $t$ small enough and a small enough neighborhood of $z$); it is the identity at $t=0$. Let $\alpha_{ij}$ be the time derivative at $t=0$ of $(\theta_i(t)^{-1}\circ\theta_j(t))$. It is naturally seen as a holomorphic vector field on $U_{ij}$. On $U_{ijk}=U_i\cap U_j\cap U_k$, one has the identity:
$$\alpha_{ij}+\alpha_{jk}+\alpha_{ki}=0$$
Hence the collection $(\alpha_{ij})$ defines a \v Cech 1-cocycle with values in the holomorphic tangent sheaf $T\Sigma\simeq K^{-1}$. Its class in $H^1({\mf U},K^{-1})$ does not depend on choices. Indeed, if $\theta_i$ is replaced with $\theta_i\circ\beta_i$, where $(\beta_i^t)$ is a family of conformal maps $U_i\rightarrow U_i$, it is simple to check that $\alpha_{ij}$ is replaced with $\alpha_{ij}+\beta_i-\beta_j$, which consists in adding the coboundary $d\beta$. Hence the infinitesimal deformation defines an element in $H^1({\mf U},K^{-1})$. Since the covering ${\mf U}$ was chosen acyclic ($K^{-1}$ is trivialized on the analytic disks $U_i$), $H^1({\mf U},K^{-1})\simeq H^1(\Sigma,K^{-1})$ (sheaf cohomology). One can also check that this does not depend on the choice of covering.

Conversely, any element in $H^1(\Sigma,K^{-1})$ corresponds to an infinitesimal deformation of $\Sigma$ (i.e. there is a smooth one parameter family of Riemann surfaces which induces a prescribed element of $H^1(\Sigma,K^{-1})$ via the construction described above); this follows from the existence theorem of Kodaira-Spencer (\cite{Kodaira}, building on the Newlander-Nirenberg theorem), given that $H^2(\Sigma,K^{-1})=0$ for dimensional reasons.

This gives the identification:
\begin{equation}\label{eq:KSisom}
(T{\mc T_g})_\Sigma\simeq H^1(\Sigma,K^{-1})
\end{equation}
and by Serre duality: $(T^*{\mc T}_g)_\Sigma\simeq H^0(\Sigma,K^2)$. The complex structure of $T{\mc T}_g$ (recall that the Teichm\"uller space is complex analytic) corresponds to the natural complex structure of $H^1(\Sigma,K^{-1})$. An application of the Riemann-Roch theorem shows that $h^0(\Sigma,K^{-1})=\dim H^0(\Sigma,K^{-1})=3g-3$ if $g\geq 2$ (0 if $g=0$, 1 if $g=1$), i.e. ${\mc T}_g$ is $3g-3$ dimensional (as a complex manifold).

\paragraph{Punctures.} 
We shall also consider Teichm\"uller surfaces with marked points (often referred to as {\em punctures}), or more generally marked $k$-jets. A configuration consists now of a Teichm\"uller surface $\Sigma$ with marked points $X_i$, and the $k_i$-jet of a parameter at $X_i$, $i=1,\dots,r$. A $k$-jet at $X$ is an element of $\C[z]/(z^{k+1}\C[z])$ with a first-order zero, where $z$ is a given analytic local coordinate at $X$: $z(X)=0$; more intrinsically, it is an element of ${\mc O}_X/{\mf m}_X^{k+1}$ where ${\mc O}_X$ is the local ring at $X$ and ${\mf m}_X$ its maximal ideal. 

Two configurations are equivalent if there is a conformal equivalence sending marked points (and jets) to marked points (and jets). The Teichm\"uller space ${\mc T}={\mc T}_{g,1^{k_1},\dots,1^{k_r}}$ (notation as in \cite{Kont_Vir}) is the set of equivalence classes of such configurations. As before, it is a complex analytic space with tangent space identified as:
$$(T{\mc T}_{g,1^{k_1},\dots,1^{k_r}})_{\Sigma,\dots}\simeq H^1(\Sigma,K^{-1}\otimes{\mc O}(-\sum_{i=1}^r(k_i+1)X_i)).$$
The sheaf in the RHS is that of holomorphic vector fields with a zero of order $\geq k_i+1$ at $X_i$, $i=1,\dots,r$. Indeed, a vector field preserving a $k$-jet at $X$ vanishes at order $k+1$ there. %
Moreover, there is a natural projection - consisting in forgetting the jets - ${\mc T}_{g,1^{k_1},\dots,1^{k_r}}\rightarrow{\mc T}_{g,1,\dots,1}$, making the first a $G$-principal bundle over the second, where $G=\prod\Aut({\mc O}_{X_i}/{\mf m}_{X_i}^{k_i+1})$. %

Results on the marked Teichm\"uller space may be recovered easily from the classical (unmarked) set-up. For instance, ${\mc T}_{g,1,1}$ is naturally identified with an open subset of the fibered product ${\mc C}_g\times_{{\mc T}_g}{\mc C}_g$, where ${\mc C}_g\rightarrow{\mc T}_g$ is the universal Teichm\"uller curve. 

\paragraph{Bordered surfaces.} 
We now turn to bordered Teichm\"uller surfaces. Let $\Sigma^s$ be a smooth bordered oriented surface with genus $\rho$ and $n$ boundary components, used for reference. A Teichm\"uller surface (of this topological type) is a Riemann surface equipped with a diffeomorphism to $\Sigma_s$, defined up to isotopy. As usual, to a bordered surface $\Sigma$, one associates its compact double $\hat\Sigma$, here of genus $g=2\rho+n-1$. The Teichm\"uller space is the set of Teichm\"uller surfaces (with a marked homology basis) of this type up to equivalence. It can be constructed directly (see e.g. \cite{Jost} for a detailed discussion) or seen as a real analytic subspace of the complex analytic space ${\mc T}_g$. The Kodaira-Spencer construction goes through if one defines the tangent sheaf $T\Sigma$ as the sheaf of analytic vector fields that flow along the boundary (this is a sheaf in real vector spaces). Then
$$(T{\mc T}_{\rho,n})_\Sigma\simeq H^1(\Sigma,T\Sigma)$$
as real vector spaces, and one can see easily by a reflection argument that $H^1(\Sigma,T\Sigma)\otimes_\R\C\simeq H^1(\hat\Sigma,T\hat\Sigma)$. Indeed, if $\omega$ is a (local) section of $T\hat\Sigma$, then $\omega=\frac{\omega+\overline{\iota\omega}}2+\frac{\omega-\overline{\iota\omega}}2$, and $\frac{\omega+\overline{\iota\omega}}2$ and $i\frac{\omega-\overline{\iota\omega}}2$ are sections of $T\Sigma$. Hence ${\mc T}_{\rho,n}$ has $3g-3$ real dimensions.
Similarly, one can mark points and jets on the boundary or in the bulk (the interior of $\Sigma$); jets on the boundary are assumed to be real along the boundary and compatible with the orientation of the boundary. If a $k_i$-jet is marked at $X_i$, one gets:
$$(T{\mc T}_{\rho,n,1^{k_1},\dots,1^{k_r}})_\Sigma\simeq H^1(\Sigma,(T\Sigma)\otimes{\mc O}(-\sum(k_i+1)X_i)).$$

\subsubsection{Virasoro uniformization}

We proceed by describing Virasoro uniformization of these Teichm\"uller spaces, as introduced by Kontsevich and Beilinson-Schechtman (\cite{Kont_Vir,BeiSch}). Let $\Sigma$ be a reference Teichm\"uller curve, ${\mc T}$ the Teichm\"uller space of curves which are diffeomorphic to it. Then $(T{\mc T})_\Sigma\simeq H^1(\Sigma,T\Sigma)$ by the Kodaira-Spencer isomorphism \eqref{eq:KSisom}. Choose a point $X\in\Sigma$ and a local analytic coordinate $z$ at $X$ ($z(X)=0$). Let $D=z^{-1}(D(0,\eta))$ a disk neighborhood of $X$ in $\Sigma$, for $\eta$ small enough. The covering ${\mf U}=\{D,\Sigma\setminus\{X\}\}$ yields an injective map $H^1({\mf U},T\Sigma)\rightarrow H^1(\Sigma,T\Sigma)$. Actually, ${\mf U}$ is a Leray covering: indeed, $H^1(D,T\Sigma)=0$ by Dolbeault's lemma, and $H^1(\Sigma^\times,T\Sigma)=0$ since $H^1(\Sigma^\times,{\mc O})=0$ and the Mittag-Leffler problem is solvable on the non-compact Riemann surface $\Sigma^\times=\Sigma\setminus\{X\}$ (see e.g. \cite{Forster} Section 26). Hence:
\begin{equation}\label{eq:Virunif0}
H^1(\Sigma,T\Sigma)\simeq H^1({\mf U},T\Sigma)\simeq (T\Sigma)(D^\times)/\left((T\Sigma)(D)+(T\Sigma)(\Sigma^\times)\right)
\end{equation}
where ${\mc F}(U)$ denotes the sections of the sheaf ${\mc F}$ on the open set $U$ and $D^\times$ is the punctured disk $D\setminus\{X\}$ (for instance, $(T\Sigma)(D^\times)$ is the vector space of holomorphic vector fields on the punctured disk $D^\times$). In words, the tangent space to ${\mc T}$ at $\Sigma$ can be represented by the holomorphic vector fields in the punctured analytic disk $D^\times$  modulo holomorphic vector fields defined in $D$ or in $\Sigma^\times$. Similarly,
$$H^1(\Sigma,(T\Sigma)\otimes{\mc O}(-X)^k)\simeq H^1({\mf U},(T\Sigma)\otimes{\mc O}(-X)^k)).$$
Jets at other points may also be marked in a similar way.

Writing vector fields on $D^\times$ in the local coordinate $z$, a Laurent vector field $v\in \C((z))\partial_z$ (here $\C((z))=\C[[z]][z^{-1}]$) yields an element of $H^1({\mf U},T\Sigma(-kX))$, hence of $H^1(\Sigma,T\Sigma(-kX))$ (there is no convergence issue since we quotient by vector fields in $D$ vanishing at order $k+1$ at $X$). A Laurent vector field converging in some annulus $r_1D\setminus r_2D$ around $X$ also yields such an element (by considering the covering $\{r_1D,\Sigma\setminus r_2D\}$). Thus we get a map
\begin{equation}\label{eq:Virunifcomp}
\C[z,z^{-1}]\partial_z/(z^{k+1}\C[z]\partial_z)\simeq\C((z))\partial_z/(z^{k+1}\C[[z]]\partial_z)\longrightarrow H^1(\Sigma,(T\Sigma)\otimes{\mc O}(-X)^k)\simeq (T{\mc T}_{1^k})_\Sigma
\end{equation}
It is worth pointing out that any class in the RHS can be represented by an element of $\C[z,z^{-1}]\partial_z$ (rather than a general holomorphic vector field in $D^\times$ as in \eqref{eq:Virunif0}). Using e.g. Riemann-Roch and the prime form, one can construct a replicating kernel $S_L(z,w)$ s.t. $S_L(.,w)$ is a section of $L=T\Sigma(-kX)$; $S_L(.,w)$ is a section of $K\otimes L^{-1}$; $S_L$ is biholomorphic except on the diagonal, where it has an expansion $S_L(z,w)=\frac{dw}{2i\pi(z-w)}$ in any trivialization, and at $X$ where it has a pole of bounded order. Then if $s$ is a section of $L$ on $D^\times$, $C$ a contour around the puncture in $D^\times$, then $(S_Ls)(z)=\oint_C s(w)S_L(z,w)dw$ is a meromorphic section of $L$ with a pole at $X$ and a jump across $C$ given by $s$. Consequently $S_Ls$ is meromorphic in $D$ and represents the same class of $H^1(\Sigma,L)$ as $s$ does.

This argument shows that the map \eqref{eq:Virunifcomp} is surjective (from \eqref{eq:Virunif0}); and has kernel given by restrictions of holomorphic vector fields on $\Sigma^\times$ to $D^\times$ (here $\partial_z=\frac\partial{\partial z}$).

In the case of bordered surfaces, we will consider deformations at boundary points. Then we have the corresponding statement for the map
$$\R((z))\partial_z/(z^{k+1}\R[[z]]\partial_z)\longrightarrow H^1(\Sigma,(T\Sigma)\otimes{\mc O}(-X)^k)\simeq(T{\mc T}_{1^k})_\Sigma$$
where $(\Sigma,X,\dots)$ is a bordered surface with a $k$-jet marked at the boundary point $X$; ${\mc T}_{1^k}$ is the corresponding Teichm\"uller space; $T\Sigma$ is the sheaf of analytic vector fields flowing along the boundary of $\Sigma$; and $z$ is a local coordinate mapping a neighborhood of $X$ in $\Sigma$ to a neighborhood of $0$ in $\H$. As explained earlier, it follows from \eqref{eq:Virunifcomp} by doubling arguments.

\subsubsection{Witt algebra representation}\label{Sec:Wittrep}

We want to realize first the Witt algebra and later the Virasoro algebra and their universal enveloping algebras as differential operators on a space of smooth functions. We start with a discussion of the Witt algebra. 

The (real) Witt algebra is the Lie algebra with basis $(\ell_n)_{n\in\Z}$ and bracket given by
\begin{equation}\label{eq:Wittcomm}
[\ell_m,\ell_n]=(m-n)\ell_{m+n}
\end{equation}
for $m,n\in\Z$. It may be realized as $\R[z,z^{-1}]\partial_z$, with $\ell_n=-z^{n+1}\partial_z$.

\paragraph{Formal coordinates.}

Let us fix $\Sigma^s$ a bordered smooth surface; let ${\mc T}_{1^k}$ be the Teichm\"uller space of surfaces diffeomorphic to $\Sigma^s$ with marked points $X,X_1,\dots,X_n$ ($X$ a boundary point) and a marked $k$-jet at $X$; the deformation occurs at $X$ and one keeps track of the ``spectator" points $X_1,\dots,X_n$ (we omit this additional marking from the subscript of ${\mc T}_{1^k}$). We have natural smooth covering maps ${\mc T}_{1^{k+1}}\rightarrow{\mc T}_{1^k}$ and we may define the projective limit
$${\mc T}_{1^\infty}=\lim_{\longleftarrow}{\mc T}_{1^k}$$
As a topological space, it may be equipped with the initial topology, viz. the coarsest topology making the canonical projections $\pi_k:{\mc T}_{1^\infty}\rightarrow{\mc T}_{1^k}$ continuous. Concretely, a basis of the topology of ${\mc T}_{1^\infty}$ is given by $(\pi_k^{-1}(U_{k,\alpha}))_{k,\alpha}$, $(U_{k,\alpha})_{\alpha\in A}$ a basis of the topology of ${\mc T}_{1^k}$. 

A point in ${\mc T}_{1^\infty}$ is represented by a Teichm\"uller surface with a marked point $X$ and a formal local coordinate at that point, viz. an invertible (for composition) element of the completed local ring $\hat{\mc O}_X=\varprojlim{\mc O}_X/{\mf m}_X^k$. If $z$ is a (genuine) local coordinate at $X$, a formal local coordinate is a formal power series $\sum_{n\geq 1}a_nz^n$, $a_1\neq 0$ (e.g. $\sum_{n\geq 1} n!z^n$). We will consider in particular marked points on the boundary, in which case we have $a_n\in\R$, $a_1>0$ (if $z$ maps a neighborhood of $X$ to a neighborhood of $0$ in $\H$).

The projection $\pi_k$ corresponds to the truncation of formal local coordinate 
$$\sum_{n\geq 1}a_nz^n\longmapsto \sum_{n\geq 1}a_nz^n\mod z^{k+1}\R[[z]]$$
Remark that the map itself is independent of the choice of local coordinate $z$.

Next we want define a notation of smooth functions (and more generally sections) on ${\mc T}_{1^\infty}$. One possible route is to realize ${\mc T}_{1^\infty}$ as a Fr\'echet manifold (\cite{Ham_Frechet}). Instead (but essentially equivalently) we follow an elementary approach suited to the situation. We can simply define
$$C^\infty({\mc T}_{1^\infty})=\varinjlim C^\infty({\mc T}_{1^k})$$
where the direct limit is taken as sheaves. This means that if $U\subset{\mc T}_{1^\infty}$ is open, $f:U\rightarrow\R$ is smooth iff $f$ can be written locally as the pullback of a smooth function on one of the ${\mc T}_{1^k}$, i.e. iff there is a collection of open sets $U_\alpha$ of ${\mc T}_{1^{k(\alpha)}}$ and smooth (in the usual sense) functions $f_\alpha:U_\alpha\rightarrow\R$ such that $f=f_\alpha\circ\pi_{k(\alpha)}$ on $\pi_{k(\alpha)}^{-1}(U_\alpha)$, and $U$ is covered by the $\pi_{k(\alpha)}^{-1}(U_\alpha)$'s.

\paragraph{Construction of the $\ell_n$'s.}

We will now define local operators: $\ell_n:C^\infty(U)\rightarrow C^\infty(U)$ that represent the Witt algebra. From the definition of $C^\infty(U)$, it is enough to evaluate $\ell_n(f\circ\pi_k)$ when $f\in C^\infty(U_k)$, $U_k$ an open set of ${\mc T}_{1^k}$, $k\in\N$ fixed.

Let $(\Sigma,X,\tilde z,\dots)$ a point in $U_{k'}$, i.e. a Teichm\"uller surface with a marked boundary point $X$ and $\tilde z$ a $k'$-jet of local coordinate at $X$ ($k'\geq k$); here $U_{k'}=\pi_{k',k}^{-1}(U_{k})$. Let us extend $\tilde z$ to an actual local coordinate $z$ (i.e. $\tilde z=z\mod z^{k'+1}{\mc O}_X$); $z$ identifies a neighborhood of $X$ in $\Sigma$ with a neighborhood of $0$ in the upper half-plane ${\mb H}$. Let $r>0$ be small enough so that $z^{-1}$ is defined and analytic on $D(0,2r)\cap {\mb H}$; set $A=\{z\in{\mb H}: \frac 34r<|z|<\frac 54r\}$. We may represent $\Sigma$ as $z^{-1}(D(0,\frac 32r))$ and $\Sigma\setminus z^{-1}(D(0,r/2))$ glued along their intersection.

In ${\mb H}$, consider the flow of analytic maps defined by $h_0(z)=z$, $\dot h_t(z)=-h_t^{n+1}(z)$, i.e. the flow generated by the vector field $-z^{n+1}\frac{\partial}{\partial z}$. For any fixed semi-annulus around $0$, $h_t$ is analytic on this semi-annulus for $t$ small enough. Then $z^{-1}\circ h_t$ maps $A$ to a semi-annulus in $\Sigma\setminus z^{-1}(h_t(D(0,r/2))$. Let $\Sigma_t$ be the surface obtained by identifying these two open sets via $h_t\circ z$. Then $\Sigma_t$ has the same smooth type and markings as $\Sigma=\Sigma_0$ and has also a distinguished local coordinate at $X$. The surface $\Sigma_t$ and the germ of the local coordinate (a fortiori its $k$-jet) do not depend on the choice of annulus (for $r$ small enough). Thus for $t$ small we have a path $t\mapsto (\Sigma_t,z\mod z^{k+1}{\mc O}_X,X,\dots)$ in ${\mc T}_{1^k}$. This path is smooth (as in the general Kodaira-Spencer construction); at $t=0$, $\frac d{dt}\Sigma_t$ is the tangent vector given by the class of $-z^{n+1}\partial_z$ in $H^{1}(\Sigma,(T\Sigma)\bigotimes{\mc O}(-X)^k)\simeq (T{\mc T}_{1^k})(\Sigma,\dots)$. This class does not depend merely on the $k$-jet $\tilde z$, but rather on a $k'$-jet for a large enough $k'$. Indeed, given two local coordinates $z_1,z_2$ at $X$, observe that if $z_2=z_1\mod z_1^{k'+1}{\mc O}_X$, then
$$-z_1^{n+1}\partial_{z_1}=-z_2^{n+1}\partial_{z_2}\mod z_2^{k+1}{\mc O}_X\partial_{z_2}$$
provided that $k'\geq k-n$. Consequently, we can define a function $\ell_n f$ on $U_{k'}$ by:
\begin{equation}\label{eq:lndef}
(\ell_n f)(\Sigma,X,\tilde z,\dots)={\frac{d}{dt}}_{|t=0}f(\Sigma_t,X,\tilde z,\dots)
\end{equation}
provided that $k'\geq k+n^-$, where $n^-=\max(-n,0)$.

Remark that, for $n<0$, $\ell_n$ does not define a vector field (or derivation) on any of the ``classical" (finite-dimensional) Teichm\"uller spaces ${\mc T}_{1^k}$, which motivates the introduction of the projective limit ${\mc T}_{1^\infty}$.

\paragraph{Smoothness.}

Then we need to verify that $\ell_n f$ is smooth, i.e. is in $C^\infty(U_{k'})$. For notational simplicity we will check that $\ell_nf$ is smooth in the case where $X$ is a bulk (viz. interior), rather than boundary, point. The proof in the case of a boundary point is very similar, using reflection/doubling arguments.

Let us describe a neighborhood of $(\Sigma,X,\tilde z)$ in ${\mc T}_{1^{k'}}$. As before, we fix a local coordinate $z$ at $X$ with $k'$-jet $\tilde z$ and describe $\Sigma$ as the gluing of a semidisk $D(0,r)$ and $\Sigma\setminus z^{-1}(D(0,r/2))$, identified via $z$. Let $g_{\underline t}$ be a smooth family of analytic maps defined in the semi-annulus $A$ with $g_{0,\dots,0}(z)=z$ (here $d$ is the dimension of the Teichm\"uller space ${\mc T}_{1^{k'}}$ and $\underline t=(t_1,\dots,t_d)$) and $\Sigma_{\underline t}$ be the surface obtained by twisting the identification along the annulus by $g_{\underline t}$. If the vector fields $\partial_{t_1}g_{\underline t}(z)\partial_z$, \dots, $\partial_{t_d}g_{\underline t}(z)\partial_z$ map to a basis of $H^1(\Sigma,(T\Sigma)\bigotimes{\mc O}(-k'X))$, then $(t_1,\dots,t_d)$ are smooth local coordinates for ${\mc T}_{1^{k'}}$ near $(\Sigma,X,\tilde z)$, and $\partial_{t_1},\dots,{\partial}_{t_d}$ are smooth vector fields. %

From \eqref{eq:lndef}, we may write
$$(\ell_n f)(\Sigma_{\underline t})={\frac{d}{d\eps}}_{|\eps=0}f(\Sigma_{\underline t,\eps},X,\tilde z,\dots)$$
where $\Sigma_{\underline t,\eps}$ is the surface obtained from $\Sigma$ by twisting the gluing by $h_\eps\circ g_{\underline t}$. Since $\underline t\mapsto \Sigma_{\underline t}$ is a complete family of deformations in the sense of \cite{Kodaira}, there is a smooth map $(\underline t,\eps)\mapsto \underline s(\underline t,\eps)$ s.t. for small $\underline t,\eps$, $\Sigma_{\underline t,\eps}$ and $\Sigma_{\underline s}$ are equivalent in ${\mc T}_{1^{k'}}$. It follows that $\ell_n f$ is smooth.

In conclusion we have defined by \eqref{eq:lndef} an operator $\ell_n: C^\infty(U_k)\rightarrow C^\infty(U_{k'})$ provided that $k'\geq k+n^-$. More precisely, if ${\rm Der}(U_k)$ is the space of derivations of functions on $U_k$, we have
$$\ell_n\in {\rm Der}(U_k)\otimes_{C^\infty(U_k)}C^\infty(U_{k'})$$
In coordinates, if $t_1,\dots,t_d$ are smooth coordinates on $U_k$, then $\ell_n f=\sum_i g_i\frac{\partial}{\partial t_i}f$ where the $g_i$'s are smooth functions on $U_{k'}$. Remark also that the construction of $\ell_n$ commutes with the natural inclusions $C^\infty(U_k)\hookrightarrow C^\infty(U_{k+1})$ (this is immediate e.g. from the representation \eqref{eq:lndef}). Consequently, we may define
$$\ell_n: C^\infty(U)\longrightarrow C^\infty(U)$$
for any open set $U$ of ${\mc T}_{1^\infty}$.

\paragraph{Bracket.}

We then want to check that $\R((z))\partial_z\rightarrow \End(C^\infty(U_\infty))$, $-z^{n+1}\partial_z\mapsto\ell_n$ is a Lie algebra morphism. Let $(g_t)_{t\geq 0}$, $(h_t)_{t\geq 0}$, $(k_t)_{\geq 0}$ be the flows of analytic maps generated (in $\H$) by the vector fields $-z^{m+1}\partial_z$, $-z^{n+1}\partial_z$, $[-z^{m+1}\partial_z,-z^{n+1}\partial_z]=-(m-n)z^{m+n+1}\partial_z$; for small $t>0$, these are defined in a neighborhood of a fixed semicircle in ${\mb H}$. As before we fix a surface $(\Sigma,X,\dots)$ with a local coordinate at $X$ and consider $\Sigma_{t,s}$ the surface obtained by twisting the gluing by $g_t\circ h_s$, and $\tilde\Sigma_{t,s}$ the surface obtained by twisting the gluing by $h_s\circ g_t$. Then for $f\in C^\infty(U_k,V_k)$
\begin{align*}
(\ell_m\ell_n f)(\Sigma,\dots)&=\frac{d}{dt}_{|t=0}\frac{d}{ds}_{|s=0}f(\Sigma_{t,s},\dots)\\
(\ell_n\ell_m f)(\Sigma,\dots)&=\frac{d}{ds}_{|s=0}\frac{d}{dt}_{|t=0}f(\tilde\Sigma_{t,s},\dots)
\end{align*}
We have: $g_t\circ h_s\circ g_t^{-1}\circ h_s^{-1}=k_{st}+o(st)$ in the sense of uniform convergence of analytic maps on compacts subsets of $\overline{\mb H}\setminus\{0\}$. Let us define $\hat\Sigma_{t,s}$ the surface obtained from $\tilde \Sigma_{t,s}$ by twisting the gluing by $k_{st}$. Then (e.g. using again the notion of complete family of \cite{Kodaira})
\begin{align*}
f(\Sigma_{t,s})-f(\hat\Sigma_{t,s})&=o(st)\\
f(\hat \Sigma_{t,s})-f(\tilde\Sigma_{t,s})&=st(m-n)\ell_{m+n}f(\tilde\Sigma_{t,s})+o(st)
\end{align*}
Note that the limit \eqref{eq:lndef} is locally uniform. Consequently 
$$\ell_m\ell_n-\ell_n\ell_m=(m-n)\ell_{m+n}\in \End(C^\infty(U_\infty,V_\infty))$$

For $n\geq -2$, we have concrete representations for the $\ell_n$'s operating on $C^\infty(U_\infty)$. Take $f\in C^\infty(U_\infty)$, then we can evaluate $f$ at the surface $(\Sigma,X,\dots,z)$ where $x$ is a local coordinate at $X$; $f$ depends on $z$ only through a $k$-jet, $k$ locally bounded. Then
\begin{align*}
(\ell_nf)(\Sigma,X,\dots,z)&=\frac{d}{d\eps}_{|\eps=0}f(\Sigma,X,\dots,z-\eps z^{n+1})&{\rm if\ }n\geq 0\\
(\ell_{-1}f)(\Sigma,X,\dots,z)&=\frac{d}{d\eps}_{|\eps=0}f(\Sigma,z^{-1}(\eps),\dots,z-\eps)\\
-(\ell_{-2}f)(\Sigma,X,\dots,z)&=\frac{d}{d\eps}_{|\eps=0^+}f(\Sigma\setminus z^{-1}([0,i\sqrt{2\eps}]),z^{-1}(i\sqrt{2\eps}),\dots,\sqrt{z^2+2\eps})\\
\end{align*}
This suggests the following alternative (and rather elementary) approach: starting from these expressions, verify that the Witt commutation relations \eqref{eq:Wittcomm} hold for $m,n\geq -2$ (though doing this cleanly seems to require an argument essentially isomorphic to the one we used). Then if we define inductively $\ell_{-n-1}=\frac{1}{1-n}[\ell_{-n},\ell_{-1}]$ for $n\geq 2$, \eqref{eq:Wittcomm} holds for $m,n\in\Z$.

\subsubsection{Sewing}\label{sssec:sewing}

Here we briefly - and rather informally - discuss the relation of this construction with other formalisms related to the sewing of surfaces (leading to modular functors and vertex operator algebras), following Segal \cite{Segal} (see also the approach of Vafa \cite{Vafa_punctured} for punctured surfaces), and subsequently developed in particular by Huang \cite{Huang_VOA}.

Consider a surface $\Sigma$ with marked points $X_1,\dots,X_n$ and corresponding marked (genuine) local coordinates $z_1,\dots,z_n$ (so that $z_i(X_i)=0$). Let $D=D(0,1)$ denote the unit disk in $\C$. If the coordinates are s.t. $z_i^{-1}:D\rightarrow\Sigma$ is well-defined and the $z_i^{-1}(D)$'s are pairwise disjoint, then $\Sigma\setminus\cup_iz_i^{-1}(D)$ is a surface with $n$ holes and analytically parameterized boundary circles (see \cite{Segal}). Conversely, in a surface with holes and paramaterized boundary circles, one can fill the holes with analytic disks and recover a surface with marked points and local coordinates.

Consider now two surfaces with marked points and coordinates $(\Sigma,X_1,\dots,X_m,z_1,\dots,z_m)$ and $(\Sigma',Y_1,\dots,Y_n,w_1,\dots,w_n)$. Provided that $z_m^{-1}:D\rightarrow\Sigma$ and $w_1^{-1}:D\rightarrow\Sigma'$ are well defined (and their ranges do not include other marked points), one can define a sewed surface
$$\Sigma''=\Sigma{\rm\ }\vphantom{\infty}_m\infty_1{\rm\ }\Sigma'$$
by excising the disk $z_m^{-1}(D)$ (resp. $w_1^{-1}(D)$) from $\Sigma$ (resp. $\Sigma'$) and gluing the excised surfaces along the unit circle (a point $W$ on $\Sigma$ and $W'$ on $\Sigma'$ are identified if $|z_m(W)|=|w_1(W)|=1$ and $z_m(W)w_1(W')=1$). Here $\vphantom{\infty}_m\infty_1$ denotes the sewing operation (relative the ``out" point $X_m$ and the ``in" point $Y_1$), and the resulting surface has a natural complex structure and inherits all other markings ($X_1,\dots$).

Sewing is then a partially defined associative operation on the collection of such marked surfaces. The Riemann sphere $(\hat\C,0,\infty)$ with the standard local coordinate $z$ (resp. $-z^{-1}$) at $0$ (resp. $\infty$) is an identity for this sewing operations.

For $n\in\Z$, $t\in\R$, one can consider perturbations of that identity element given by
\begin{equation*}
\left\{\begin{array}{llll}
\Sigma_n(t)&=(\hat\C,0,\infty,z(1-ntz^n)^{-1/n},z^{-1})&&{\rm if\ }n>0\\
\Sigma_n(t)&=(\hat\C,0,\infty,ze^{-t},z^{-1})&&{\rm if\ }n=0\\
\Sigma_n(t)&=(\hat\C,0,\infty,z,z^{-1}(1-ntz^n)^{-1/n})&&{\rm if\ }n< 0
\end{array}
\right.
\end{equation*}
One may check that the $\Sigma_n(.)$'s are partial one-parameter groups in the sense that
$$\Sigma_n(t)\vphantom{\infty}_2\infty_1\Sigma_n(t')=\Sigma_n(t+t')$$
for $t,t'$ small enough (this follows from $h_t\circ h_{t'}=h_{t+t'}$ near $0$ if we denote $h_t(z)=z(1+ntz^n)^{-1/n}$, the flow generated by the vector field $z\mapsto -z^{n+1}\partial_z$). 

Then we may think of the $\Sigma_n$'s as an exponentation of the Witt algebra with ``$\Sigma_n(t)=\exp(t\ell_n)$" \cite{Huang_VOA}, e.g. in the sense that for small $t$
$$\Sigma_m(t)\infty\Sigma_n(t)\infty\Sigma_{-m}(t)\infty\Sigma_{-n}(t)\simeq\Sigma_{m+n}((m-n)t^2)$$
(compare with \eqref{eq:Wittcomm}).

In this framework, one can interpret \eqref{eq:lndef} as follows. If $(\Sigma,X,z,\dots)$ is a surface with a marked boundary point $X$ and local coordinate $z$, one can consider the family of deformations
$$t\mapsto\Sigma_n(t)\vphantom{\infty}_2\infty_1(\Sigma,X,z,\dots)$$
and for a suitable test function $f$ set
$$(\ell_n f)(\Sigma,X,z,\dots)={\frac{d}{dt}}_{|t=0}f\left(\Sigma_n(t)\vphantom{\infty}_2\infty_1(\Sigma,X,z,\dots)\right)$$
(The sewing operation has natural compatibilities with doubling, so that it is not really problematic to consider boundary deformations). The difficulty consists in defining a suitable class of test functions on which this operation is defined and the commutation relations \eqref{eq:Wittcomm} are satisfied, which we addressed directly in the previous section.

\section{Determinants of Laplacians}

In this section, we recall the definition of $\zeta$-regularized determinants of Laplacians, and the properties we will need later, in particular the Polyakov-Ray-Singer conformal anomaly. A relation with the Brownian loop measure is also pointed out and used to (re)derive some of these properties. This relation will also be used for analytic surgery formulae in Appendix \ref{sec:surgery}.

\subsection{Loop measures}

The Brownian Loop measure was introduced and studied in \cite {LW}, motivated by Conformal Restriction measures (\cite{LSW3}). We give here a slightly more general construction (see also {\cite{LJ_loops}). 

Consider a Riemannian manifold $(M,g)$, possibly with boundary. There is a natural Brownian Motion on $M$, with generator the Laplace-Beltrami operator $\Delta$. In local coordinates,
$$\Lap=\frac 1{\sqrt{\det g}}\sum_{i,j}\frac{\partial}{\partial x_i}g^{ij}\sqrt{\det g}\frac{\partial}{\partial x_j}$$
where $g^{ij}=(g^{-1})_{ij}$. 
(We follow here the analytic, rather than %
the geometric convention, that considers a positive operator).
We restrict ourselves for now to Dirichlet boundary conditions: the process is killed upon hitting the boundary. This gives a semigroup $(P_t)_t=(e^{t\Lap})_t$ operating on $C_0^\infty(M)$ and a family $({\mb W}^x)_{x\in M}$ of subprobability measures on paths (or probability measures if, as is customary, one extends the state space $M$ by a cemetery state $\partial$). Here $M\cup\{\partial\}$ is a compactification of $M$, and ${\mb W}^x$ denotes the measure on the path space $C([0,\infty),M\cup\{\partial\})$ induced by Brownian Motion started from $x\in M$. Note that this corresponds to Brownian motion running at speed 2.

Given $f_1,f_2\in{\mb L}^2(M)$, $F$ a bounded Borel functional on $C([0,\infty),M\cup\{\partial\})$ with Skorokhod topology ($\partial$ is an isolated cemetery state), one can consider:
$$(f_1,f_2)\mapsto\left(F\mapsto\int_0^\infty\frac{dt}t\int_M f_1(x){\mb W}^x(F(X_{0\leq s\leq t})f_2(X_t))dA(x)\right)$$
where $dA$ is the volume measure associated with $g$ (in local coordinates, $dA=\sqrt{\det g} dx_1\wedge\cdots\wedge dx_n$). This defines an operator 
from ${\mb L^2(M)}\otimes{\mb L^2(M)}$ to measures on paths, or equivalently an operator on ${\mb L}^2(M)$ taking values in measures on paths. Taking the trace of this operator, one gets a measure on paths. This measure is supported $\nu_r$ on loops.

More explicitly, the measure ${\mb W}^x$ on paths starting from $x$ can be disintegrated w.r.t. $X_t$; it is well-known that the distribution of $X_t$ is absolutely continuous w.r.t. $A$. If $p_t(x,y)$ designates the heat kernel, we have
$${\mb W}^x=\int_M p_t(x,y){\mb W}_{t,x,y}dA(y)$$
where ${\mb W}_{t,x,y}$ is the bridge measure on paths from $x$ to $y$ with lifetime $t$.
Then the measure $\nu_r$ can be written as:
$$\nu_r\int_0^\infty\frac {dt}t\int_{M}p_t(x,x){\mb W}_{t,x,x}dA(x)$$

Consider the following set of (rooted) loops:
$\{\gamma\in C([0,t],M):\gamma(0)=\gamma(t)\}$. There is an equivalence relation $\sim$ given by: $\gamma_1\sim\gamma_2$ if $\gamma_1(.)=\gamma_2(t_0+.)$ for some $t_0$, in the sense of periodic continuation. Classes of equivalence are unrooted (but oriented) loops of lifetime $t$. The measure induced on unrooted loops is the Brownian loop measure, and will be denoted here by $\nu$. A coarser equivalence relation is given by: $\gamma_1\sim\gamma_2$ if $\gamma_1=\gamma_2\circ \iota$, where $\iota:\R\rightarrow\R$ is an increasing bijection with $\iota(x+t_1)=\iota(x)+t_2$ ($t_i$ is the period of $\gamma_i$, $i=1,2$). Classes of equivalence are unrooted, oriented loops up to reparameterization. One can equip this space with a natural metric (viz. uniform distance minimized over reparameterizations).

\begin{Prop}[\cite{LW}]\label{Prop:loopmeas}
\begin{enumerate}
\item  (Restriction) If $K\subset M$, the loop measure on $M$ restricted to loops contained in $M\setminus K$ is the loop measure on $M\setminus K$.
\item  (Conformal Invariance) In dimension 2, the measure on loops up to time reparameterization induced by the loop measure (on parameterized, unrooted loops) is invariant under Weyl scaling: $g\rightarrow e^{2\sigma}g$.
\end{enumerate}\end{Prop}
\begin{proof}
The restriction property is immediate (Dirichlet boundary conditions).\\
The second property is a consequence of time-change properties of Brownian motion and cyclical reindexing of Markovian loops.
If $g'=e^{2\sigma} g$, then $\Lap_{g'}=e^{-2\sigma}\Lap_g$; if $(X_t)_{t\geq 0}$ is the Brownian motion corresponding to $g$, $s=\int_0^ue^{2\sigma}(X_u)du$ is a random time change, then $X'_s=X_{t(s)}$ is the Brownian motion corresponding to $g'$.

Consider a bounded Borel functional $F$ on rooted paths, up to time change, and $g$ a Borel function on $\R^+$. Then the time change result means that:
$${\mb W}^x(F(X_{0\leq t\leq T})g(T))=({\mb W}^x)'(F(X_{0\leq t\leq T})g(S))$$
where $S=\int_0^Te^{2\sigma}(X_t)dt$. It follows that 
$$d\nu_r'(\gamma_.)=\frac TS\cdot\frac{dS}{dT}(\gamma_0)d\nu_r(\gamma_.)$$
where $\nu_r,\nu'_r$ denote the rooted loop measures relative to $g,g'$.

To proceed to unrooted loops, observe that (from the simple Markov property):
$$p_t(x,z){\mb W}_{t,x,z}=\int_M p_\tau(x,y)p_{t-\tau}(y,z) {\mb W}_{\tau,x,y}\bullet{\mb W}_{t-\tau,y,z}dA(y)$$
where $\tau\leq t$ is fixed and $\bullet$ designates concatenation of paths. It follows that 
if $\theta_\tau$ is a shift operator on loops, then $\nu_r$ is invariant under rerooting loops:
$$\int Fd\nu_r=\iint K(T,d\tau)\theta_\tau Fd\nu_r$$
where $K(T,d\tau)$ is an arbitrary collection of probability measures. We shall use this where $K(T,d\tau)=\ind_{\tau\leq T}\frac{d\tau}T$.
If $F$ is a bounded Borel functional on unrooted loops, then :
$$\int Fd\nu_r'=\int F\frac TS.\frac{dS}{dT}(\gamma_0)d\nu_r=\int\int_0^T\theta_\tau\left(F\frac TS.\frac{dS}{dT}(\gamma_0)\right)\frac{d\tau}Td\nu_r=\int F\frac 1S\int_0^T\frac{dS}{dT}(\gamma_\tau)d\tau d\nu_r=\int Fd\nu_r$$
since $\theta_\tau F=F$, $\theta_\tau S=S$, $\theta_\tau T=T$, and $\theta_\tau(\frac{dS}{dT}(\gamma_0))=\frac{dS}{dT}(\gamma_\tau)$.
\end{proof}

One can extend this definition to Brownian motions (or diffusions) with different boundary conditions, in particular Neumann conditions on some of the boundary components, or oblique conditions in dimension 2. The properties above stay valid for these loop measures with reflection (in the restriction property, it is then understood that the boundary condition on $K$ is Dirichlet, other boundary conditions being unchanged). There are also natural analogues for (discrete state space, discrete or continuous time) Markov chains (\cite{LJ_loops}).
In dimension greater than 2, conformal invariance is no longer satisfied. However (Weyl scaling), the loop measure associated with a generator $L$ is identical to the one associated with $e^{2\sigma}L$.

\subsection{$\zeta$-regularization}

For simplicity, we discuss in details only the case of dimension 2. For the Laplacian on the Riemannian manifold $(M,g)$ (see e.g. \cite{Ros_Lap}), one has the following Pleijel-Minakshisundaram expansion for the heat kernel $p$ at small times:
\begin{equation}\label{eq:pleijel}
p_t(x,x)=\frac 1{4\pi t}+\frac{K(x)}{12\pi}+O(t)
\end{equation}
where $K$ is the Gauss curvature. For $y\neq x$, we have the following large deviation estimate of Varadhan:
$$\lim_{t\searrow 0}\frac 1t\log p_t(x,y)\leq -\frac 14 d(x,y)^2$$
in terms of the geodesic distance. These estimates are uniform on closed compact manifolds. Near (Dirichlet or Neumann) boundary components, one can proceed by doubling (McKean and Singer \cite{McKean_Singer}).

Consider the spectrum of the positive operator $(-\Lap)$ : $0\leq\lambda_1\leq\cdots\leq\lambda_n\cdots$. Since $L^2(M)$ has a Hilbert basis of (smooth) eigenfunctions of $\Delta$, one gets:
$$\int_M p_t(x,x)dA(x)=\Tr(e^{t\Lap})=\sum_i e^{-t\lambda_i}$$
From here the $\zeta$-function of the Laplacian is defined as:
$$\zeta(s)=\sum_{\lambda_i\neq 0}\lambda_i^{-s}=\frac{1}{\Gamma(s)}\int_0^\infty(\Tr(e^{t\Lap})-h^0)t^{s-1}dt$$
where $h^0$ designates the dimension of $\Ker(\Lap)$ (this is zero if there is a Dirichlet boundary condition). The identity follows from $\lambda^{-s}=\frac 1{\Gamma(s)}\int_0^\infty t^{s-1}e^{-\lambda t}dt$.

The $\zeta$-function is absolutely convergent and analytic in $s$ for $\Re s>1$. Moreover, it has a meromorphic continuation to $\C$ (in the variable $s$). Indeed, $\int_1^\infty(\Tr(e^{t\Lap})-h^0)t^{s-1}dt$ is an entire function in $s$ (exponential decay), while the short time heat kernel asymptotics give (for a closed surface):
$$\Tr(e^{t\Lap})-1=\int_M p_t(x,x)dA(x)-1=\frac{A}{4\pi t}+(\frac{\chi(M)}6-1)+O(t)$$
where $\chi(M)=2-2g$ is the Euler characteristic of $M$ (by the Gauss-Bonnet theorem $\chi(M)=\frac 1{2\pi}\int_MKdA$). It follows that  $\Gamma(s)\zeta(s)=\frac{A}{4\pi(s-1)}+(\frac{\chi}6-1)\frac 1s+\cdots$, where the remainder $(\cdots)$ is an analytic function in $\Re s>{-1}$. Further terms in the heat kernel expansion yield meromorphic continuation to $\C$. Remark that $\Gamma(s)\sim s^{-1}$ at $s=0$, so $\zeta(0)=\frac\chi 6-1$, a topological invariant.

In the case where $M$ has a boundary (say with at least one Dirichlet boundary component, so that $h^0=0$), by \cite{McKean_Singer}:
$$\Tr(e^{t\Lap})=\frac{A}{4\pi t}+\frac{\ell_N-\ell_D}{8\sqrt{\pi t}}+(\dots)+O(\sqrt t)$$
where $\ell_D$ (resp. $\ell_N$) is the length of the Dirichlet (resp. Neumann) boundary components, $(\dots)$ are integrals of explicit local quantities (in the interior or on the boundary). As before, this yields a meromorphic continuation of $\zeta$ to $\{s:\Re s>-\frac 12\}$, with simple poles at $s=1,1/2$.

In the case where the boundary is piecewise smooth (with corners), the $\zeta$-function still has a meromorphic continuation; corners contribute to the constant term in the heat kernel expansion, hence to $\zeta(0)$ (\cite{Kac_drum}).

We can now define ${\det}_\zeta(-\Lap)\stackrel{def}{=}e^{-\zeta'(0)}$ if $h^0=0$ and ${\det}'_\zeta(-\Lap)\stackrel{def}{=}e^{-\zeta'(0)}$ otherwise (${\det}_\zeta$ can be thought of as a regularized product of eigenvalues, and ${\det}'_\zeta$ as a regularized product of non-zero eigenvalues).

Alternatively, again in the case $h^0=0$, one can consider the associated loop  measure $\nu$. It is then easy to see that:
\begin{equation}\label{eq:zetaloop}
\zeta(s)=\frac 1{\Gamma(s)}\int T(\gamma)^sd\nu(\gamma)
\end{equation}
where $T(\gamma)$ is the lifetime of the loop $\gamma$ (with generator $\Lap$\ , i.e. running at speed 2). While $\nu$ is defined solely from the complex structure, the functional $T$ depends on the Riemannian metric. Thus one can think (heuristically) of $\zeta'(0)$ as a normalized total mass for the loop measure.

\subsection{Conformal anomaly formulae}

A (closed) Riemann surface $M$ can be equipped with different compatible Riemannian metrics, each yielding a Laplacian and its determinant. The dependence of the determinant on the metric within a conformal class is given by the Polyakov-Ray-Singer conformal anomaly formula (\cite{Pol_bosonic}). Conformal anomalies will be instrumental in our approach to Virasoro representations (in non-zero central charge).

\begin{Thm}
Let $g,g'=e^{2\sigma}g$ be two conformally equivalent metrics on $M$. Then:
\begin{equation}\label{PRS}
\log{\det}'(-\Lap_{g'})-\log{\det}'(-\Lap_{g})=-\frac 1{6\pi}\left(\frac 12\int_M |\nabla\sigma|^2dA+\int_MK\sigma dA\right)+\log A'-\log A\end{equation}
where $dA,\nabla,K$ are the volume element, gradient, scalar curvature associated with $g$.
\end{Thm}

In the presence of a boundary, a similar formula was obtained by Alvarez (\cite{Alv_boundary}):

\begin{Thm}\label{Thm:Alvarez}
Let $g,g'=e^{2\sigma}g$ be two conformally equivalent metrics on $M$; we consider Dirichlet boundary conditions on $\partial M$. Then:
\begin{equation}\label{Alv}
\log{\det}(-\Lap_{g'})-\log{\det}(-\Lap_{g})=-\frac 1{6\pi}\left(\frac 12\int_M |\nabla\sigma|^2dA+\int_MK\sigma dA+\int_{\partial M}k\sigma ds\right)-\frac 1{4\pi}\int_{\partial M}\partial_n\sigma ds\end{equation}
where $dA,\nabla,K,ds,k$ are the volume element, gradient, scalar curvature, boundary length element, geodesic curvature associated with $g$; $\partial_n$ is the outer normal derivative.
\end{Thm}
\begin{proof}
We give an argument based on the loop representation. 
For simplicity, we restrict ourselves to the case where $\sigma=0$ on a neighborhood of $\partial M$ (which will be sufficient for our purposes later on). Consider the $\zeta$-function \eqref{eq:zetaloop}
$$\zeta_g(s)=\frac 1{\Gamma(s)}\int T_g(\gamma)^sd\nu(\gamma)$$
where $\nu$ is the (conformally invariant) loop measure (valid if $\Re s>1$). We want to compute the first order variation of $\zeta'_g(0)$ under $g\rightarrow g_\eps=e^{2\eps\sigma}g$. Then $T_{\eps}(\gamma)=\int_0^Te^{2\eps\sigma}(\gamma_t)dt$, so:
$$\partial_\eps\zeta_\eps(s)_{|\eps=0}=\frac 1{\Gamma(s)}\int T(\gamma)^s \left(s\int_0^T 2\sigma(\gamma_t)\frac{dt}T\right)d\nu(\gamma).$$
In terms of the measure on rooted loops $\nu_r$, one has the rerooting identity:
$$\partial_\eps\zeta_\eps(s)_{|\eps=0}=\frac {2s}{\Gamma(s)}\int T(\gamma)^s\sigma(\gamma_0)d\nu_r(\gamma)=\frac {2}{\Gamma(s)}\int_M\sigma(x)dA(x)\int_0^\infty st^{s-1}p_t(x,x)dt.$$
This identity is {\em a priori} satisfied for $\Re s>1$. The right-hand side is well-defined in a neighborhood of $s=0$. From the construction of the analytic continuation of $\zeta_\eps$ (by substracting a rational counterterm), it follows that the identity stays valid in a neighborhood of $s=0$. It is now a matter of heat kernel asymptotics. Uniformly in $x$ away from the boundary, we have:
$$\int_0^\infty st^{s-1}p_t(x,x)dt=\int_0^1st^{s-1}\left(\frac 1{4\pi t}+\frac{K(x)}{12\pi}\right)ds+O(s)=\frac{K(x)}{12\pi}+O(s).$$
Applying this to $g_\eps=e^{2\eps\sigma}g$, this proves:
$$\frac{d}{d\eps}\zeta'_{\eps}(0)=\frac 1{6\pi}\int_M \sigma K_\eps dA_\eps=\frac 1{6\pi}\int_M(-\sigma\Lap\sigma+K\sigma)dA$$
since $dA_\eps=e^{2\eps\sigma}dA$, $K_\eps=(K-\Lap\sigma)e^{-2\eps\sigma}$.
By integration, this gives the formula, in the case $\sigma=0$ near the boundary. 
\end{proof}

Note that the proof given here stays valid for arbitrary boundary conditions, as long as $\sigma=0$ in a neighborhood of the boundary.

A fundamental feature of these formulae is that the logarithmic variation of the determinant
is local, in the sense that it is an integral of local quantities in the interior and on the boundary. Let us give a direct argument for this locality property.

Consider two compact surfaces $(M_1,g_1)$, $(M_2,g_2)$ that are identified (together with their metrics) in a disk $D_0$ (also assuming for simplicity that $M_1,M_2$ have a Dirichlet boundary component). One changes $g_1$ to $g'_1$, $g_2$ to $g'_2$ conformally, with the conditions: $g_i=g'_i$ outside $D\subset\subset D_0$; $g_1=g_2$ and $g'_1=g'_2$ in $D_0$. Then:
$$\Gamma(s)((\zeta_{g'_1}-\zeta_{g_1})-(\zeta_{g'_2}-\zeta_{g_2}))(s)=\int (T_{g'_1}^s-T_{g_1}^s)d\nu_1-\int (T_{g'_2}^s-T_{g_2}^s)d\nu_2.$$
Obviously loops contained in $M_1\setminus D$, or $M_2\setminus D$, or in $D_0$ do not contribute to the RHS (for any $s$). So one can restrict the integrals on the RHS to loops crossing the annulus $D_0\setminus D$. The mass of those loops is finite (say by the large deviation estimate). So the RHS vanishes at $s=0$, meaning that the logarithmic variation of determinants is the same for the two surfaces. The argument also works for a variation of the metric in a semidisk neighborhood of a boundary point.

\section{Determinant bundle and Virasoro representations}

We have defined a representation of the Witt algebra on smooth functions on appropriate Teichm\"uller spaces (or rather projective limits of these, see \eqref{eq:lndef}). Here we discuss the extension of this formalism to the Virasoro algebra, essentially following \cite{KontSuh,Fri_CFTSLE} - this requires considering sections of the determinant line bundle rather than smooth functions. 

The (real) Virasoro algebra is a central extension of the Witt algebra, with basis $\{(L_n)_{n\in\Z},{\bf c}\}$ and bracket given by
\begin{align*}
[L_m,L_n]&=(m-n)L_{m+n}+\frac{m(m^2-1)}{12}\delta_{n,-m}{\bf c}\\
[L_n,{\bf c}]&=0
\end{align*}
for all $m,n\in\Z$.

We consider $(\Sigma_0,X_0,\dots)$ a smooth surface with a marked boundary point and a Teichm\"uller marking (and any number of additional boundary and bulk points, and jets); let ${\mc T}$ denote the Teichm\"uller space of bordered Riemann surfaces smoothly equivalent to it. 

Let us consider Riemannian metrics $g$ on $(\Sigma_0)$ which are flat near boundary components and s.t. the boundary components are geodesic, i.e. the metric is modelled on that of a flat cylinder near boundary components. This allows to omit the Alvarez corrections to the Polyakov anomaly formula (Theorem \ref{Thm:Alvarez}); given a bordered surface, it is always possible to construct such a metric (using e.g. standard partition of unity arguments).

We consider functions $f$ defined on this space of metrics s.t.
\begin{equation}\label{eq:confanom}
f(e^{2\sigma}g)=f(g)\exp\left(\frac 1{12\pi}\left(\frac 12\int_{\Sigma}|\nabla \sigma|^2dA+\int_\Sigma K\sigma dA\right)\right)
\end{equation}
Any such function may be written as $f(g)=\det_\zeta(-\Lap_g)^{-1/2}h([(\Sigma_0,g)])$ where $[(\Sigma_0,g)]$ is the point in ${\mc T}$ corresponding to the Riemannian surface $(\Sigma_0,g)$ and $h$ is a function on ${\mc T}$. Consequently, we may identify the functions $f$ as sections of a line bundle ${\mc L}$ over ${\mc T}$, a reference section being given by:
$$s_\zeta(g)={\det}_\zeta(-\Lap_g)^{-1/2}$$
which is nowhere vanishing by definition, so that ${\mc L}$ is a trivial line bundle. We declare $s_\zeta$ to be smooth, so that ${\mc L}$ is a smooth line bundle. More intrinsically, a smooth section of ${\mc L}$ can be identified with a smooth functional $F$ of the metric satisfying the required conformal anomaly formula \eqref{eq:confanom} (in the sense that if $(g_t)_t$ is a smooth family of metrics, $t\mapsto F(\Sigma_0,g_t)$ is smooth).

If $c\in\R$, we may define ${\mc L}^{\otimes c}$ as the line bundle with a smooth nonvanishing section $s_\zeta^c$. We want to define a natural representation of the Virasoro algebra acting on sections of the pullback ${\mc L}_\infty$ of ${\mc L}$ to ${\mc T}_{1^\infty}$. By definition, a smooth section of ${\mc L}_\infty^{\otimes c}$ over $U_\infty$ is of type $hs_\zeta^c$, with $h\in C^\infty(U_\infty)$; the space of these smooth sections is denoted by $C^\infty(U_\infty,{\mc L}_\infty^{\otimes c})$. Heuristically, we would like to define a connection $\nabla$ on ${\mc L}_\infty$ such that its curvature form gives the Virasoro cocycle,
ie:
$$``[\nabla_{\ell_m},\nabla_{\ell_n}]=\nabla_{[\ell_m,\ell_n]}+\frac c{12}m(m^2-1)\delta_{m,-n}"$$

\subsection{Virasoro generators}\label{ssec:Virgen}

As earlier (see \eqref{eq:lndef}) we consider a marked point $X$ on the boundary of the surface $\Sigma$ and a local coordinate $z$ at $X$, which identifies a neighborhood of $X$ in $\Sigma$ with a semidisk $D^+(0,r)\subset\H$; ${\mc T}_{1^k}$ denotes the Teichm\"uller space of surfaces with a marked $k$-jet at $X$, all other markings being fixed; and ${\mc L}_k$ denotes the pullback of ${\mc L}$ to ${\mc T}_{1^k}$.

Let us consider a local section of ${\mc L}_k$ over $U_k$, an open subset in ${\mc T}_{1^k}$; it may be identified with a function $f$ on the space of Riemannian metrics. Let $h_t$ be the flow in ${\H}$ generated by the vector field $-z^{n+1}\partial_z$. Then consider $(\Sigma_t,X,\dots)$ the surface obtained by twisting the gluing of $z^{-1}(D^+(0,r))$ and $\Sigma\setminus z^{-1}(D^+(0,r/2))$ by $h_t$. We also consider ${\mb H}_t$ the deformation of $({\mb H},0)$ by the same twist, i.e. changing the gluing of $D^+(0,r)$ and ${\mb H}\setminus D^+(0,r/2)$); as a Riemann surface this is of course equivalent to $({\mb H},0)$. Then the pairs of Riemannian manifolds $(\Sigma_t,\H_t)$ and $(\Sigma_0,\H_0)$ are naturally identified near the marked point; and the pairs $(\Sigma_t,\Sigma_0)$ and $(\H_t,\H_0)$ are identified away from the marked point.

Choose a metric $g_t$ on $\Sigma_t$ which is flat near the geodesic boundary and is constant in $t$ in $\Sigma\setminus z^{-1}(D(0,r))$. Similarly, we choose a metric $\tilde g_t$ on ${\mb H}_t$ with the same conditions. Finally we assume that $(g_t)$ and $(\tilde g_t)$ agree in the semidisk around $0$ where they are identified via the local coordinate. Because of the locality of the Liouville action \eqref{eq:confanom}, the ratio
$$\frac{f(\Sigma_t,g_t)s_\zeta(\H_0,\tilde g_0)}{f(\Sigma_0,g_0)s_\zeta(\H_t,\tilde g_t)}$$
is independent of choices of metrics (this is the key ``neutral collection" argument of \cite{KontSuh}). We set:
\begin{equation}\label{eq:Lndef}
(L_n f)(\Sigma_0,g_0)=\lim_{t\rightarrow 0}\frac{f(\Sigma_t,g_t)s_\zeta(\H_0,\tilde g_0)}{s_\zeta(\H_t,\tilde g_t)}
\end{equation}
If $h\in C^\infty(U_k)$, by \eqref{eq:lndef} we have trivially the Leibniz rule: 
$$L_n(hf)=h(L_nf)+(\ell_nh)f$$
whenever $L_nf$ exists.  We want to show that the limit \eqref{eq:Lndef} is well-defined; depends on the choice of local coordinate $z$ only through a $k'$-jet; and is smooth. By the previous remark, it is enough to show that $L_n s_\zeta/s_\zeta$ is well-defined and smooth on $U_{k'}$ for some $k'\geq k$. 

The evaluation of $L_n s_\zeta/s_\zeta$ is a very concrete problem on variation of complex structures (smoothness is then immediate). It is also unfortunately rather lengthy and involved, and thus postponed to Appendix \ref{Sec:varform}. Recall the Schwarzian connection $S=S_\Sigma$ from \eqref{Sconn}; it depends on the bordered surface $\Sigma$, the marked point $X$, and a $3$-jet of local coordinate at $X$. We may phrase:

\begin{Thm}\label{Thm:Virrep}
Let $L_n$ be defined by \eqref{eq:Lndef}. Then $L_n$ maps $C^\infty(U_k,{\mc L}^{\otimes c}_k)$ to $C^\infty(U_{k+n^{-}+1},{\mc L}^{\otimes c}_{k+n^-+1})$. We have the expressions:
\begin{align*}
L_n(fs)&=(\ell_n f)s&\hphantom{bbbbbbbbbb}&n\geq -1\\
L_n(fs)&=(\ell_n f)s+c\frac{(\ell_{-1})^{-n-2}S_\Sigma}{12(-n-2)!}fs&\hphantom{bbbbbbbbbb}&n\leq -2
\end{align*}
with $f\in C^\infty(U_\infty)$ and $s=(s_\zeta)^{\otimes c}$ the reference section. 
As operators on $C^\infty(U_\infty,{\mc L}_\infty^{\otimes c})$, the $L_n$'s satisfy the Virasoro commutation relations:
$$L_nL_m-L_mL_n=(n-m)L_{m+n}+\frac c{12}n(n^2-1)\delta_{m,-n}$$
\end{Thm}
We need a $(k+n^-)-$ jet to evaluate $\ell_nf$ and a $(3+(n+2)^-)$-jet to evaluate $\ell_{-1}^{-n-2}S_\Sigma$.
\begin{proof}
The expression for $L_n$, $n\leq -2$ is the content of Appendix \ref{Sec:varform}. 
We thus simply have to check the commutation relations. Saying that $S\in C^\infty(U_3)$ is a Schwarzian connection \eqref{eq:Sconndef} is equivalent to: $\ell_0S=2S$, $\ell_1S=0$, $\ell_2S=6$. This implies directly that the commutation relations are satisfied if $-2\leq m,n\leq 2$ and $-2\leq m+n\leq 2$. Moreover it is immediate to check that for $n\geq 0$
\begin{align*}
L_nL_1-L_1L_n&=(n-1)L_{n+1}\\
L_{-1}L_{-n}-L_{-n}L_{-1}&=(n-1)L_{-n-1}
\end{align*}
The subalgebra of $\End(C^\infty(U_\infty,{\mc L}^{\otimes c}_\infty))$ generated by the $(L_n)_{n\in\Z}$ is generated by $(L_n)_{-2\leq n\leq 2}$; the commutation relations on this generators are identical to the Virasoro algebra relations. The universal enveloping algebra ${\mc U}(\Vir)$ of the Virasoro algebra has a presentation (as an algebra over $\C$) with generators $(L_n)_{-2\leq n\leq 2},{\bf c}$ (${\bf c}$ a central element) and relations given by the commutation relations for $-2\leq m,n\leq 2$ and $-2\leq m+n\leq 2$. Consequently we have defined a representation of ${\mc U}(\Vir)$ in $\End(C^\infty(U_\infty,{\mc L}^{\otimes c}_\infty))$ (${\bf c}$ operates by multiplication by $c$), and 
the commutation relations are satisfied for all $m,n$.
\end{proof}
More generally, if $V$ is a vector bundle over $U$, $V_k$ its pullback to $U_k$, we can define: $$L_n: C^\infty(U_k,V_k\otimes {\mc L}^{\otimes c}_k)\rightarrow C^\infty(U_{k+n^{-}+1},V_{k+n^-+1}\otimes{\mc L}^{\otimes c}_{k+n^-+1})$$ by
$L_n(v\otimes s)=(\ell_n v)\otimes s+v\otimes (L_ns)$. The same commutation relations are obviously satisfied. 

In particular, if $h$ in $\R$, one may consider $V=|T^{-1}\Sigma|^{\otimes h}$, the (norm of the) cotangent bundle (relative to the marked point $X$) raised to the power $h$. 
Explicitly, sections of this bundle can be identified with functions $f$ of $(\Sigma,X,\dots)$ and a local coordinate $z$ at $X$ ($z$ maps neighborhood of $X$ to a neighborhhood of $0$ in $\C$ or $\H$ depending on whether $X$ is a bulk or boundary point) with the transformation property
$$f(\Sigma,X,z')=\left|\frac {dz'}{dz}\right|^{-h}f(\Sigma,X,z)$$
In the case where $X$ is a boundary point, this may be identified in turn with an element $f\in C^\infty(U_1)$ satisfying $\ell_0f=hf$; trivially $\ell_n f=0$ if $n>0$.

If $\phi\in C^\infty(U_1,{\mc L}_1^{\otimes c})$ with $\ell_0\phi=h\phi$ (equivalently, $\phi$ a section of $|T^{-1}\Sigma|^{\otimes h}\otimes{\mc L}^{\otimes c}$), one may consider the ${\mc U}(\Vir)$-module generated by $\phi$: ${\mc U}(\Vir)\phi$. This is the {\em highest-weight} Virasoro module generated by the highest-weight vector $\phi$; it has central charge $c$ and highest-weight $h$.

\subsection{Canonical differential equations and Virasoro singular vectors}

We now discuss Virasoro singular vectors and the (genuine, finite dimensional) differential operators associated to them, along the lines of \cite{FriKal,Fri_CFTSLE,Kont_arbeit,KontSuh}.

An element of ${\mc U}(\Vir)$ maps to an operator $C^\infty(U_k,V_k\otimes {\mc L}_k^{\otimes c})\rightarrow
C^\infty(U_{k'},V_{k'}\otimes {\mc L}_k^{\otimes c})
$ for $k'$ large enough. Given a choice of section $\sigma$ of the natural covering map $U_{k'}\rightarrow U_k$, one obtains a differential operator
on $C^\infty(U_k\otimes {\mc L}_k^{\otimes c})$. In general, this operator is non-canonical, as it depends on the choice of section $\sigma$. This dependence disappears (up to a multiplicative factor) for special elements of ${\mc U}(\Vir)$, the {\em singular vectors}, which appear in the study of degenerate Verma modules.

In ${\mc U}(\Vir)$, there is a natural grading given by $\deg(L_m)=m$, $\deg({\bf c})=0$. If $\phi$ is a highest-weight element with weight $h$, $\Delta\in{\mc U}(\Vir)$ an homogeneous element of degree $d$, then $L_0\Delta\phi=(h-d)\Delta\phi$. Let $\Vir^-$ (resp. $\Vir^+$) be the subalgebra of $\Vir$ spanned by negative (resp. positive) degree elements.

In ${\mc U}(\Vir)$, a {\em $(c,h)$-singular vector} is an element $\Delta\in{\mc U}(\Vir^-)$ such that for all $m>0$, $L_m\Delta$ is in the left ideal ${\mc U}(\Vir)(L_0-h,{\bf c}-c,L_1,\dots,L_n,\dots)$. Let $\Delta$ be a non-trivial homogeneous element of degree $-n$ ($n>0$) and $P$ any element of degree $n$; then $P\Delta$ is an element of degree $0$. It may be written uniquely as a polynomial in $L_0$ and ${\bf c}$ modulo the left ideal ${\mc U}(\Vir)\Vir^+$. If for any $P$, this polynomial vanishes upon evaluation at $L_0=h$, ${\bf c}=c$, then $\Delta$ is a $(c,h)$-singular vector. The {\em Kac determinant formula} states in particular that if $c=13-6(\tau+\tau^{-1})$, a $(c,h)$ singular vector exists iff 
$$h=h_{r,s}(\tau)=\frac{(r\tau-s)^2-(\tau-1)^2}{4\tau}=h_{s,r}(\tau^{-1})$$
for some $r,s\in\N$. In this case there is a singular vector $\Delta_{r,s}$ in degree $-rs$, which is uniquely defined up to multiplicative constant (see e.g. \cite{IohKog_Vir}). Of particular interest to us here will be
$$\Delta_{2,1}=L_{-1}^2-\tau L_{-2}$$
Let $\phi\in C^\infty(U_1,V_1\otimes {\mc L}_1^{\otimes c})$ with $L_0\phi=h\phi$ (or equivalently $\phi\in C^\infty(U,|T^{-1}\Sigma|^{\otimes h}\otimes{\mc L}^{\otimes c})$). We have $L_n\phi=0$ for $n>0$, i.e. $\phi$ a $(c,h)$ highest-weight vector. Then $\Delta_{r,s}\phi$ is defined in $C^\infty(U_k,V_k\otimes{\mc L}_k^{\otimes c})$, $k$ large enough. If $h=h_{r,s}$, for any $m>0$ we have $L_m\Delta_{r,s}\in {\mc U}(\Vir)(L_0-h,{\bf c}-c,L_1,\dots)$ (by a slight abuse of notation, we identify elements of ${\mc U}(\Vir)$ with their images as operators on $C^\infty(U_\infty,|T^{-1}\Sigma|^{\otimes h}\otimes{\mc L}^{\otimes c}_\infty)$), and consequently:
$$L_m\Delta_{r,s}\phi=0$$
which is saying that $\Delta_{r,s}\phi\in C^\infty(U_k,V_k\otimes{\mc L}_k^{\otimes c})$ depends solely on the 1-jet (through a multiplicative constant). Since this holds for a generic section $\phi$, this shows that the coefficients of the differential operator $\Delta_{r,s}$ depend on the choice of local coordinate only through a multiplicative constant.

Consequently we may consider:
\begin{equation}\label{eq:candiffop}
\Delta_{r,s}:C^\infty(U,|T^{-1}\Sigma|^{\otimes h_{r,s}}\otimes{\mc L}^{\otimes c})\rightarrow C^\infty(U,|T^{-1}\Sigma|^{\otimes (h_{r,s}+rs)}\otimes{\mc L}^{\otimes c})
\end{equation}
which is now a genuine differential operator on bundles over the Teichm\"uller space ${\mc T}$ and is well-defined up to multiplicative constant.

\subsection{Commuting representations}\label{ssec:commrep}

The construction of the Witt/Virasoro generators is local at a marked point. Consequently it may be carried out simultaneously at several marked points, leading to commuting representations of the Virasoro algebra. Let us formalize this observation.

We consider a bordered surface $\Sigma$ with two marked boundary points $X,Y$ at which the deformations will occur; several ``spectator" boundary and bulk points may also be marked. We mark a $k_1$-jet of local coordinate $z_1$ at $X$ and a $k_2$-jet of local coordinate $z_2$ at $Y$. The Teichm\"uller space of surfaces of the same type (with the same markings) is simply denoted by ${\mc T}_{1^{k_1},1^{k_2}}$. Let us fix $m_1,m_2\in\Z$. If $k_1,k_2$ are large enough, one may consider a two-parameter family of deformations of the surface $\Sigma_{s,t}$ obtained in the following way: excise a semidisk around $X$ and glue it back according to $g_s$, the flow generated by $-z_1^{m_1+1}\partial_{z_1}$; excise a semidisk around $Y$ and glue it back according to $h_t$, the flow generated by $-z_2^{m_2+1}\partial_{z_2}$. These semidisks are chosen to be disjoint. Then if $U$ is a neigborhood of (the class of) $\Sigma$ in ${\mc T}_{1^0,1^0}$ and $U_{k_1,k_2}$ its preimage in ${\mc T}_{1^{k_1},1^{k_2}}$, we have the Witt generators at $X,Y$ \eqref{eq:lndef} satisfying:
\begin{align*}
\ell_{m_1}^X f(\Sigma_{s,t})&=\frac{d}{ds}f(\Sigma_{s,t})\\
\ell_{m_2}^Y f(\Sigma_{s,t})&=\frac{d}{dt}f(\Sigma_{s,t})
\end{align*}
if $f\in C^\infty(U_{k_1,k_2})$. Then we have for all $m,n\in\Z$:
\begin{align*}
[\ell^X_{m},\ell^X_n]&=(m-n)\ell^{X}_{m+n}\\
[\ell^Y_{m},\ell^Y_n]&=(m-n)\ell^{Y}_{m+n}\\
[\ell^X_m,\ell^Y_n]&=0
\end{align*}
the last line resulting from the existence of an explicit commuting flow $(s,t)\mapsto \Sigma_{s,t}$.\\
Since the Virasoro generators are also defined in terms of local deformations, the same argument shows that
\begin{align*}
[L^X_{m},L^X_n]&=(m-n)L^{X}_{m+n}+\frac c{12}m(m^2-1)\delta_{m+n,0}\\
[L^Y_{m},L^Y_n]&=(m-n)L^{Y}_{m+n}+\frac c{12}m(m^2-1)\delta_{m+n,0}\\
[L^X_m,L^Y_n]&=0
\end{align*}
as operators on $C^\infty(U_\infty,{\mc L}^{\otimes c}_\infty)$, where a smooth section of ${\mc L}_\infty^{\otimes c}$ can be written locally as the pullback of a smooth section of ${\mc L}^{\otimes c}_{k_1,k_2}$ for some (finite) $k_1,k_2$.

Let us make an observation which may give some additional motivation to the definition of the Virasoro operators \eqref{eq:Lndef}. By trivializing ${\mc L}$, one may identify $L^X_{-2}$ (resp. $L^Y_{-2}$) with $\ell^{X}_{-2}+\frac{c}6S_\Sigma(X,z_1)$ (resp. $\ell^{Y}_{-2}+\frac{c}6S_\Sigma(Y,z_2)$) where $S_\Sigma(X,z)$ is the Schwarzian connection evaluated at the point $X$ in the local coordinate $z$ (a $3$-jet is enough for this). Comparing $[\ell^{X}_{-2},\ell^Y_{-2}]=0$ and $[L^X_{-2},L^Y_{-2}]=0$ leads to:
$$\ell^X_{-2}S_\Sigma(Y,z_2)=\ell^Y_{-2}S_\Sigma(X,z_1)$$
Conversely, one may recover $[L^X_m,L^Y_n]=0$ from $[\ell^X_m,\ell^Y_n]=0$ and this identity. The reader may convince himself that both sides of this identity are proportional to:
$$(\partial_{n_1}\partial_{n_2}G_\Sigma(X,Y))^2$$
If $\tilde S$ is another Schwarzian connection, it differs from $S$ by a quadratic form, i.e. $\tilde S=S+\omega$ where $\omega$ is an arbitrary section of the bundle $\Sigma\mapsto H^0(\Sigma,K_\Sigma^2)$. Setting $\tilde L^X_n=\ell^X_n$ if $n\geq -1$, $\tilde L^X_n=\ell^X_n+\frac c{6(-n-2)!}(\ell_{-1}^X)^{-n-2}\tilde S$ also gives operators satisfying the Virasoro commutation relations. However for a generic $\omega$ one cannot construct operators $\tilde L^Y_n$ also satisfying the Virasoro commutation relations and commuting with the $\tilde L^X_n$'s.

Manifestly, the argument above extends to any finite number of boundary deformation points $X_1,\dots,X_n$, yielding a representation of ${\mc U}(\Vir)^{\otimes n}$.

\subsection{Uniformization}

Here we illustrate concretely the concepts introduced earlier by explaining how to write the Virasoro representation described above in coordinates; in simple topologies, a choice of coordinates on the Teichm\"uller space is given by uniformization.

\subsubsection{Simply-connected domains}

Consider the space of simply-connected domains with $N+2$ marked points $X_0,\dots,X_{N+1}$ on the boundary; at $X_0$ we also give a $k$-jet of local coordinate; this gives the space ${\mc T}_{1^k}$. By uniformization, any such surface is equivalent to $(\H,0,z_1,\dots,z_N,\infty)$ and the $k$-jet at $0$ can be written as $w=z(1+a_2z+\cdots+a_kz^{k-1})$, where $z$ denotes the natural coordinate in $\H$. Then $(z_1,\dots,z_N,a_2,\dots,a_k)$ are smooth coordinates on ${\mc T}_{1^k}$ ($z_i\neq z_j$, $z_i\neq 0$). Alternatively, we could fix $z_N=1$ and take $(z_1,\dots,z_{N-1},a_1,\dots,a_k)$ as coordinates, but the resulting expressions are a bit more complicate.

For $n\in\Z$, let us consider the vector field 
$$w^{n+1}\partial_w=z^{n+1}(1+a_2z+\cdots+a_{k-1}z^{k-1})^{n+1}(1+2a_2z+\cdots+ka_kz^{k-1})^{-1}\partial_z$$
We may write:
\begin{equation}\label{eq:bjndef}
w^{n+1}\partial_w=\sum_{j=n}^\infty b_{j,n}z^{j+1}\partial_z
\end{equation}
where $b_{j,n}(a_1,\dots,a_{j-n+1})$ (and $b_{n,n}=1$). Here $\sum_{j=n}^{0}b_{j,n}z^{j+1}\partial_z$ is a vector field which is holomorphic away from $0$ and vanishes at infinity; $\sum_{j=1}^k b_{j,n}z^{j+1}\partial_z$ is holomorphic near $0$; and $\sum_{j=k+1}^\infty b_{j,n}z^{j+1}\partial_z$ maps to $0$ in $T{\mc T}_{1^k}$.

Let us first consider the vector fields $w^{n+1}\partial_w$, $n>0$, which operate on the $k$-jet but not on the moduli $z_1,\dots,z_N$. The vector field $-z^2\partial_z$ corresponds to replacing the coordinate $w$ with $w-\eps w^2$ and differentiating w.r.t. $\eps$ (see the end of Section \ref{Sec:Wittrep}), or in terms of coefficients:
$$-\ell_1=\partial_{a_2}+2a_2\partial_{a_3}+\cdots$$

For general $n$, $\ell_n:C^\infty({\mc T}_{k})\rightarrow C^\infty({\mc T}_{k'})$, $k'=k+n^{-}$ is given by:
$$-(\ell_n f)(z_1,\dots,z_N,a_2,\dots,a_{k'-1})=\sum_{i=1}^N \left(\sum_{j:n\leq j\leq 0} b_{j,n}z_i^{j+1}\right)\partial_{z_i}+\sum_{j=1}^{k}b_{j,n}\partial_{a_j}
$$
where the $b_{j,n}$'s are computable polynomials in the $a_i$ variables, specified by \eqref{eq:bjndef}.

For small $n$'s we have
\begin{align*}
w\partial_w&=(z-a_2z^2+(2a_2^2-2a_3)z^3+\cdots)\partial_z\\
\partial_w&=(1+2a_2z+3a_3z^2+4a_4z^2+\cdots)^{-1}\partial_z\\
&=(1-2a_2z+(4a_2^2-3a_3)z^2+(-8a_2^3+12a_2a_3-4a_4)z^3+\cdots)\partial_z\\
w^{-1}\partial_w&=\frac{(1+2a_2z+3a_3z^2+\cdots)^{-1}}{z(1+a_2z+a_3z^2+\cdots)}\partial_z\\
&=\left(\frac 1z-3a_2+(7a_2^2-4a_3)z+(-15a_2^3+19a_2a_3-5a_4)z^2+\cdots\right)\partial_z\\
w^{-2}\partial_w&=\frac{(1+2a_2z+3a_3z^2+\cdots)^{-1}}{z^2(1+a_2z+a_3z^2+\cdots)^2}\partial_z\\
&=\left(\frac 1{z^2}-\frac{4a_2}z+(11a_2^2-5a_3)+(-26a_2^3+28a_2a_3-6a_4)z+\cdots\right)\partial_z
\end{align*}
From this we read:
\begin{align*}
-\ell_1&=\partial_{a_2}+2a_2\partial_{a_3}+\cdots\\
-\ell_0&=\sum_{i=1}^Nz_i\partial_{z_i}-a_2(\partial_{a_2}+2a_2\partial_{a_3})+(2a_2^2-2a_3)\partial_{a_3}+\cdots\\
-\ell_{-1}&=
\sum_{i=1}^N(1-2a_2z_i)\partial_{z_i}+(4a_2^2-3a_3)(\partial_{a_2}+2a_2\partial_{a_3})+(-8a_2^3+12a_2a_3-4a_4)\partial_{a_3}+\cdots\\
-\ell_{-2}&=\sum_{i=1}^N\left(\frac{1}{z_i}-3a_2+(7a_2^2-4a_3)z_i\right)\partial_{z_i}+(-15a_2^3+19a_2a_3-5a_4)(\partial_{a_2}+2a_2\partial_{a_3})+\cdots\\
-\ell_{-3}&=\sum_{i=1}^N\left(\frac 1{z_i^2}-\frac{4a_2}{z_i}+(11a_2^2-5a_3)+(-26a_2^3+28a_2a_3-6a_4)z_i\right)\partial_{z_i}+\cdots
\end{align*}
From these expressions one may check directly some of the simplest commutation relations, e.g. $\ell_1\ell_{-1}-\ell_{-1}\ell_1=2\ell_0$ on $C^\infty({\mc T_2})$, $\ell_{-1}\ell_{-2}-\ell_{-2}\ell_{-1}=\ell_{-3}$ on $C^\infty({\mc T}_1)$. Moreover, if $\ell_0 f=hf$ (viz. $f$ is a homogeneous function of the $z$'s. with our choice of coordinates), then 
\begin{align*}
\ell_{-1}f&=-\sum_i\partial_{z_i}f-2a_2hf\\
\ell_{-2}f&=-\sum_i\frac{1}{z_i}\partial_{z_i}f+3a_2\sum_i\partial_{z_i}f+(7a_2^2-4a_3)hf\\
\ell_{-1}^2f&=\sum_{i,j}\partial_{z_iz_j}f+2a_2(h+1)\sum_i\partial_{z_i}f+2a_2h\left(\sum_i\partial_{z_i}f+2a_2hf\right)+2h(4a_2^2-3a_3)f\\
&=\sum_{i,j}\partial_{z_iz_j}f+2a_2(2h+1)\sum_i\partial_{z_i}f+2h(4a_2^2-3a_3+2a_2^2h)f
\end{align*}
where we used $\ell_0(\partial_{z_i}f)=(h+1)\partial_{z_i}f$.

In order to get to Virasoro generators, we evaluate
$$S_\H(w)=\{w;z\}=6(a_2^2-a_3)$$
since $S_\H(z)=0$. Then (omitting the reference section $s$)
\begin{align*}
L_{-1}&=\ell_{-1}\\
L_{-2}&=\ell_{-2}+\frac c2(a_2^2-a_3)\\
L_{-3}&=\ell_{-3}+\frac c2(-8a_2^3+12a_2a_3-4a_4)
\end{align*}
Now consider the operator $\Delta_{2,1}=L_{-1}^2-\tau L_{-2}$. If $L_0f=hf$, then
\begin{align*}
(L_{-1}^2-\tau L_{-2})f&=
\sum_{i,j}\partial_{z_iz_j}f+\tau\sum_i\frac{1}{z_i}\partial_{z_i}f+a_2(2(2h+1)-3\tau)\sum_i\partial_{z_i}f\\
&+\left(h(8a_2^2-6a_3+4a_2^2h-\tau(7a_2^2-4a_3))+\frac{\tau c}2(a_3-a_2^2)\right)f
\end{align*}
For the special values
\begin{align*}
\tau&=\frac 4\kappa\\
h&=h_{2,1}=\frac{6-\kappa}{2\kappa}=\frac 34\tau-\frac 12\\
c&=\frac{(6-\kappa)(3\kappa-8)}{2\kappa}=h(12\tau^{-1}-8)
\end{align*}
this reduces to
$$
(L_{-1}^2-\tau L_{-2})f=
\sum_{i,j}\partial_{z_iz_j}f+\tau\sum_i\frac{1}{z_i}\partial_{z_i}f=\left((-\sum_i\partial_{z_i})^2-\tau(-\sum_i{z_i}^{-1}\partial_{z_i})\right)f
$$
where the RHS has weight $h+2$ (and no longer depends on $a_2,a_3,\dots$). We recover the usual rule (\cite{BPZ}) that a singular vector is obtained by substituting $\ell^0_n=-\sum_iz_i^{n+1}\partial_{z_i}$ for $L_n$. The $\ell^0_n$'s represent the Witt algebra and operate on functions of the $z$'s (a fixed number of variables); the $L_n$'s represent the Virasoro algebra and operate on functions of the $z$'s and $a$'s. Let us expand on this comment.

We can write $L_{-n}=\ell^0_{-n}+D_n$, where $D_n$ is a differential operator (in the $z$'s and $a$'s) with coefficients vanishing at $a_2=a_3=\cdots=0$ ($n\geq 0$). It follows easily that
$$L_{-n_1}\dots L_{-n_k}=\ell^0_{-n_1}\cdots\ell^0_{-n_k}+D$$
where $D$ has the same property. Then $\Delta_{r,s}=\Delta_{r,s}^0+D_{r,s}$, where $\Delta_{r,s}^0$ is obtained by substituting $L_{-n}$ with $\ell^0_{-n}$ in the singular vector $\Delta_{r,s}$. If $\phi$ satisfies $\Delta_{r,s}(\phi s)=0$ ($\phi$ has weight $h_{r,s}$ at $0$ and $\infty$) and we set $f(z_1,\dots,z_N)=\phi(\H,0,z_1,\dots,z_N,\infty,z)$ (viz. the local coordinate at $0$ is the standard coordinate), we have $\Delta^0_{r,s}f=0$, a PDE in $N$ variables and of degree $rs$. In this fashion we recover the classical {\em Belavin-Polyakov-Zamolodchikov differential equations} \cite{BPZ}. 

This discussion can easily be extended to the case where the spectator points $z_1,\dots,z_n$ also carry weights $h_1,\dots,h_n$, in which case
$$\ell^0_{-n}=\sum_i\left(-z_i^{1-n}\partial_{z_i}-(1-n)h_iz_i^{-n}\right)$$

\subsubsection{Relation with the Bauer-Bernard framework}\label{ssec:BB}

At this point it is instructive to comment on the relation with the work initiated by Bauer and Bernard \cite{BauBer_martVir,BB_CFTSLE,BB_part}, and developed in particular in \cite{Kyt_Virmod}, which, in simply-connected domains, also produces a Virasoro representation in terms of differential operators. 

A simply-connected domain with $(n+1)$ marked points, one of them carrying a (formal) local coordinate, is equivalent to $(\H,x_1,\dots,x_n,\infty,w)$, where $w$ is the local coordinate at $\infty$, which we now parameterize. Consider a series 
$$f(u)=\sum_{j\leq 0}f_ju^{j+1}=u+0+\frac{f_{-2}}u+\cdots$$
with $f_0=1$, $f_{-1}=0$ (``hydrodynamic normalization"); this may be thought as a formal series in $\R[[u^{-1}]][u]$ or (for suitable $f_j$'s) as a germ of analytic function near $\infty$. Such series occur in particular when uniformizing subdomains of $\H$: $f=f_K:\H\setminus K\rightarrow\H$ a conformal equivalence. Consider its inverse
$$u=f^{-1}(z)=z+0-\frac{f_{-2}}z+\cdots$$
the coefficients of which are polynomials in the $f_j$'s and may be expressed by the Lagrange inversion formula.

If $z$ denotes the standard coordinate in $\H$, define the local coordinate $w$ by
\begin{equation}\label{eq:BBw}
w=\frac 1u=\frac{1}{f^{-1}(z)}=\left(z-\frac{f_{-2}}z+\cdots\right)^{-1}=z^{-1}+\frac{f_{-2}}{z^3}+\cdots
\end{equation}
This makes sense for germs of local coordinates at $\infty$ or for formal coordinates. 
Remark that the hydrodynamic normalization fixes the translation and scaling degrees of freedom. Here the space of simply-connected domains with $(n+1)$ marked points and a marked $k$-jet has coordinates $(x_1,\dots,x_n,f_{-2},\dots,f_{1-k})$. We now want to write the $\ell_n$'s and $L_n$'s in these coordinates.

We now consider the $\ell_n$'s corresponding to a deformation {\em at infinity} written in these coordinates. For $n\geq 2$, the action consists in replacing $w$ with $w-\eps w^{n+1}$ (or $u$ with $u+\eps u^{1-n}$) and differentiating w.r.t. $\eps$ at $0$:
\begin{align*}
\frac{1}{f^{-1}_\eps(z)}&=\frac{1}{f^{-1}(z)}-\frac{\eps}{f^{-1}(z)^{n+1}}+o(\eps)\\
f_\eps^{-1}(f(u))&=u+\eps u^{1-n}+o(\eps)\\
f(u)&=f_\eps(u)+\eps f'(u)u^{1-n}+o(\eps)
\end{align*}
This gives the simple expression (for $n\geq 2$)
\begin{equation}\label{eq:BB0}
\ell_n=-\sum_{j\leq 0}(j+1)f_j\partial_{f_{j-n}}=-\sum_{\ell\leq -n}(\ell+n+1)f_{\ell+n}\partial_{f_\ell}
\end{equation}
For $n\leq 1$, we focus for simplicity on the action on the $x_1,\dots,x_n$ variables, and start from the vector field
\begin{align*}
-w^{n+1}\partial_w&=u^{1-n}\partial_u=u^{1-n}\left(\frac{du}{dz}\right)^{-1}\partial_z\\
&=(P_n(z)+R_n(z))\partial_z
\end{align*}
where $P_n\in\C[z]$ and $R_n\in z^{-1}\C[[z^{-1}]]$. As for $n\geq 2$, $R_n\partial_z$ corresponds to an action on the $f_j$'s (leaving the $x_i$'s fixed). Classically, one can project a Laurent series onto its holomorphic at $0$ (resp. at infinity) component with the Cauchy integral formula. Here
$$P_n(x)=\frac{1}{2i\pi}\oint u^{1-n}(z)\left(\frac{du}{dz}\right)^{-1}\frac{dz}{z-x}$$
where $\oint$ is a contour integral on a large circle containing $x$ on which $u$ converges (if $u$ is merely a formal series, this can still be interpreted as an algebraic map on the $f_j$'s, as will be the case for all such integrals in this section). A change of variables gives
$$P_n(x)=\frac{1}{2i\pi}\oint \frac{u^{1-n}f'(u)^2}{f(u)-x}du$$
and then
$$\ell_n=\sum_i P_n(x_i)\partial_{x_i}+\cdots$$
For the action on the jet coordinates, we start by writing
$$R_n(x)=-\frac{1}{2i\pi}\oint  u^{1-n}(z)\left(\frac{du}{dz}\right)^{-1}\frac{dz}{z-x}=-\frac{1}{2i\pi}\oint \frac{u^{1-n}f'(u)^2}{f(u)-x}du$$
where $x$ lies now outside of the contour. Then
\begin{align*}
R_n(z)\partial_z&=-\frac{1}{2i\pi}\oint \frac{v^{1-n}f'(v)^2}{f(v)-z}dv\partial_z\\
&=-\frac{1}{2i\pi}\oint \frac{v^{1-n}f'(v)^2}{f(v)-f(u)}dv f'(u)^{-1}\partial_u
\end{align*}
By \eqref{eq:BB0} we know the effect of vector fields $-w^{1+m}\partial_w$, $m\geq 2$, on the jet coordinates. We simply expand $R_n(z)\partial_z$ in such vector fields, using again a Cauchy integral formula to extract coefficients:
$$R_n(z)\partial_z=-\sum_{m\geq 2}\left(\frac{1}{(2i\pi)^2}\oint_{C_o}\oint_{C_i}
\frac{v_1^{1-n}f'(v_1)^2}{f(v_1)-f(v_2)}dv_1 f'(v_2)^{-1}v_2^{m-2}dv_2\right) w^{m+1}\partial_w
$$
where $C_i$ is a contour in the interior of the larger counter $C_o$. From \eqref{eq:BB0} and
$$\sum_{m\geq 2}v^{m-2}\sum_{\ell\leq -m}(\ell+m+1)f_{\ell+m}\partial_{f_\ell}=
\sum_{\ell\leq -2}v^{-\ell-2}f'(v)\partial_{f_\ell}
$$
we get
\begin{align*}
\ell_n=&\sum_i\left(\frac{1}{2i\pi}\oint \frac{u^{1-n}f'(u)^2}{f(u)-x}du\right)\partial_{x_i}\\
&-\sum_{\ell\leq -2}\left(\frac{1}{(2i\pi)^2}\oint_{C_o}\oint_{C_i}
\frac{v_1^{1-n}f'(v_1)^2}{f(v_1)-f(v_2)}dv_1v_2^{-\ell-2}dv_2\right)\partial_{f_\ell}
\end{align*}
Finally we turn to the Schwarzian correction for $c\neq 0$. From \eqref{eq:BBw} (and the fact that $S_\H=0$ at $\infty$ in the standard coordinate $z^{-1}$), we have $S_\H=-6 f_{-2}$ (at infinity, w.r.t. the coordinate $w$). More generally, set $\iota:z\mapsto z^{-1}$ (a fractional linear transformation) and $w_0=z^{-1}$. Then by basic properties of the Schwarzian derivative
$$\{w_0;w\}=\{\iota\circ f\circ\iota;w\}=\{f\circ\iota;w\}=u^4(Sf)(u)
$$
By the geometric representation of $\ell_{-1}$, $S_\H$ evaluated at the point with $w=\eps$ w.r.t. the coordinate $w-\eps$ has expansion $\sum_{k\geq 0}\frac{\eps^k}{k!}\ell_{-1}^kS_\H$, which by the previous computation equals $\eps^{-4}(Sf)(\eps^{-1})$. Then
$$\frac{1}{k!}\ell_{-1}^kS_\H=\frac{1}{2i\pi}\oint \eps^{-k-4}(Sf)(\eps^{-1})\frac{d\eps}{\eps}=\frac{1}{2i\pi}\oint u^{k+4}(Sf)(u)\frac{du}{u}$$
so that (with $k=-n-2$)
\begin{align*}
L_n=&\frac{c}{12}\cdot\frac{1}{2i\pi}\oint u^{1-n}(Sf)(u)du\\
&+\sum_i\left(\frac{1}{2i\pi}\oint \frac{u^{1-n}f'(u)^2}{f(u)-x}du\right)\partial_{x_i}
-\sum_{\ell\leq -2}\left(\frac{1}{(2i\pi)^2}\oint_{C_o}\oint_{C_i}
\frac{v_1^{1-n}f'(v_1)^2}{f(v_1)-f(v_2)}dv_1v_2^{-\ell-2}dv_2\right)\partial_{f_\ell}
\end{align*}
which is precisely the differential representation of the Virasoro algebra constructed and studied in \cite{BauBer_martVir,BB_CFTSLE,BB_part,Kyt_Virmod}. Let us stress that the expressions on the RHS are polynomials in the $f_j$'s and consequently are still well-defined if $f$ is a formal series. We also refer to Section \ref{sssec:locmart} for further comments on that framework.

\subsubsection{Doubly-connected domains}

The Loewner equation and SLE in doubly-connected domains have been studied rather extensively, see e.g. \cite{Dub_ann,Zhan_doubly,Zhan_annprop,HBB_GFF,Law_multdef}. As is well-known, any doubly-connected domain is conformally equivalent to an annulus; we shall use the flat (with geodesic boundary) model given by
$${\bf S}_t=\{z:0<\Im z<t\}/\Z$$
with $\Z$ acting by (horizontal) translations. If we consider the space of ${\mc T}_{1^k}$ such surfaces with $n+1$ points $Z_0,\dots,Z_{N}$ marked on the same (for simplicity) boundary arc, with a $k$-jet marked at $Z_1$ (where the perturbation occurs), we get a set of local coordinates $t,z_1,\dots,z_N,a_1,\dots,a_k$ by mapping such a marked surface to
$$({\bf S}_t,0,z_1,\dots,z_N)$$
with the $k$-jet of local coordinate at $Z_1$, with a representative given by $w=\sum_{i\geq 1} a_i(z-z_1)^i$ ($z-z_1$ is the reference local coordinate given by the natural embedding - up to translations - ${\bf S}_t\hookrightarrow\C$). Here $z_1,\dots,z_N\in\R\mod \Z$, $t>0$, $a_1>0$, $a_2,\dots\in\R$. For the variety - and by contrast with the simply-connected case -  the fixed point 0 is chosen to be a spectator point.

Our next task is to write the $\ell_n$'s in these explicit coordinates. The first step, identical to the simply-connected case \eqref{eq:bjndef}, consists in writing
$$w^{n+1}\partial_w=\sum_{j\in\Z} b_{n,j}(z-z_1)^{j+1}\partial_z$$
an identity of Laurent vector fields near $z_1$. The vector field $(z-z_1)^{n+1}\partial_z$ (in a semi-annulus neighborhood of $z_1$) defines an element $H^1({\bf S}_t,(T^{-1}{\bf S}_t)\otimes{\mc O}(-D))$, where $D=Z_0+(k+1)Z_1+Z_2+\cdots+Z_k$ is the relevant divisor; we denote this element by $[(z-z_1)^{n+1}\partial_z]$. By Kodaira-Spencer \eqref{eq:KSisom}, this is identified with the tangent space (to the Teichm\"uller  space), which is spanned by the vectors $\partial_t,\partial_{z_1},\dots,\partial_{a_1},\dots$ and we simply want to find the coordinates of $[(z-z_1)^{n+1}\partial_z]$ in this basis. 

We can consider a covering of ${\bf S}_t$ adapted to the situation. Let $U_i=D^+(z_i,\eps)$ for $i=0,\dots,N$ with $z_0=0$ and $\eps>0$ small enough. Let $U_{N+1}=\{z\in{\bf S}_t:\Im z>2t/3\}$; and $U_{N+2}$ be an open set intersecting each of $U_0,\dots,U_N$ in a semi-annulus and $U_{N+1}$ in a strip; we choose $\eps$ so that $U_0,\dots,U_{N+1}$ are all disjoint and $U_0,\dots,U_{N+2}$ cover ${\bf S}_t$.

Relatively to this cover we have simple representations (as \v Cech cocycles) of our basis vectors. Namely $\partial_{z_i}$ corresponds to the cocycle equal to $\partial_z$ in $U_{i,N+2}$ and $0$ elsewhere; and $\partial_t$ corresponds to the cocycle equal to $i\partial_z$ in $U_{N+1,N+2}$ and $0$ elsewhere (with the usual notation $U_{i,j}=U_i\cap U_j$).

It is easy to check that there is a unique meromorphic vector field $x\mapsto V_0(x,y)\partial_x$ which is holomorphic in the interior ${\bf S}_t$, has continuous extension to $it+\R$ and has constant imaginary part there; has a continuous extension to $\R$ and is real there, with the exception of $y\in\R/\Z$ where it has a simple pole with residue 1; and finally $V_0(0,y)=0$. Explicitly, let $\theta$ be the odd theta function associated to $\C/(\Z+it\Z)$, and set
$$V(x,y)=\frac{\theta'}\theta(x-y)=\frac{1}{x-y}+(reg)%
$$
Then we have $\Im(V(x,y))=-2\pi$ if $x\in it+\R$, and we get explicitly
$$V_0(x,y)=V(x,y)-V(0,y)=\frac{1}{x-y}+c_0(x,y)$$
Then for $m>0$, we have
$$\frac{1}{m!}\partial_y^m V_0(.,y)=\frac{1}{(x-y)^{m+1}}+c_m(x,y)$$
with $c_m$ biholomorphic; and $\partial_y^m V_0(.,y)$ is real along both boundary components. 

We can now evaluate the coordinates of $[(z-z_1)^{n+1}\partial_z]$. If $n\geq 0$, it is a vertical vector (operating only on the formal coordinate) and corresponds to $\partial_{a_n}$ if $n\leq k$. If $n=-1$, we get $\partial_{z_1}$. If $n=-2$, we write 
$$(z-z_1)^{-1}\partial_z=\left((z-z_1)^{-1}-V_0(z,z_1)\right)\partial_z+V_0(z,z_1)\partial_z$$
where the first summand is regular near $z_1$. In turn, near $z_j$, we can write $V_0(z,z_1)=V_0(z_j,z_1)+(V_0(z,z_1)-V_0(z_j,z_1))$.  Hence $(z-z_1)^{-1}\partial_z$ corresponds to the tangent vector
$$2\pi\partial_t+\sum_{j>1}V_0(z_j,z_1)\partial_{z_j}-c_0(z_1,z_1)\partial_{z_1}-\partial_xc_0(z_1,z_1)\partial_{a_1}-\cdots$$
Note that $c_0(z,z)=V(z,0)$. Similarly, if $m\geq 1$, $(z-z_1)^{-m-1}\partial_z$ corresponds to
$$0\partial_t+\sum_{j>1}\frac{1}{m!}\partial_y^mV_0(z_j,z_1)\partial_{z_j}-c_m(z_1,z_1)\partial_{z_1}-\partial_xc_m(z_1,z_1)\partial_{a_1}-\cdots
$$
The Poisson kernel in ${\bf S}_t$ is given by
$$\pi P_{{\bf S}_t}(x,y)=-\Im(V(x,y))-\frac{2\pi}t y$$
for $x\in{\bf S}_t$, $y\in\R$; consequently the Poisson excursion kernel \eqref{eq:Poissonexcdef} is given by
$$\pi H_{{\bf S}_t}(x,y)=\partial_xV(x,y)-\frac{2\pi}t$$
for $x,y\in\R$, and the Schwarzian connection \eqref{eq:Sconndef} (w.r.t. the standard coordinate $z$ is
$$\frac 16S_{{\bf S}_t}(z)=-\frac{\theta'''}{3\theta'}(0)-\frac{2\pi}t$$
which also gives an explicit expression for the $L_n$'s (with $L_{-2}=\ell_{-2}+\frac c{12}S_{{\bf S}_t}$ etc).

\section{$\SLE$ measures}

In this section we describe the construction of SLE measures on open Riemann surfaces by localization in path space, which allows to reduce to the simply-connected case. These are positive measures; their total mass (or {\em partition function}), may be seen as a function (or more accurately, a section) on Teichm\"uller space. The main result of this section is that these partition functions are finite (if $c\leq 0$) and smooth and are annihilated by the canonical differential operator $\Delta_{2,1}$.

\subsection{Chordal SLE measures}

We start by gathering a few facts about chordal Schramm-Loewner Evolutions (SLEs); for general reference, see e.g. \cite{W1,Law_AMS}. In the upper half-plane $\H$, consider the flow of analytic maps $g_t:\H\setminus K_t\rightarrow \H$ given by $g_0(z)=z$,
$$\frac{d}{dt} g_t(z)=\frac{2}{g_t(z)-\sqrt\kappa B_t}$$
where throughout $B$ is a standard linear Brownian motion started at 0 and $\kappa>0$ is fixed. For a given $z\in\H$, this ODE is solvable up to an explosion time $\tau_z$. The compact {\em hull} $K_t\subset{\overline \H}$ is defined by ${\overline{\{z\in\H:\tau_z\leq t\}}}$. The {\em trace} $t\mapsto \gamma_t=\lim_{\eps\rightarrow 0}g_t^{-1}(\sqrt\kappa B_t+i\eps)$ is a continuous non-self-traversing path, and $\H\setminus K_t$ is a.s. the unbounded connected component of $\H\setminus\gamma_{[0,t]}$. If $\kappa\leq 4$, $\gamma$ is a.s. simple. 

The path space ${\mc P}(\H,0,\infty)$ is the space of all continuous non-self-traversing paths from $0$ to $\infty$ in $\overline\H$ up to continuous increasing time reparameterizations (a non-self traversing path can be represented as the limit of a sequence of simple paths). Specifically, one can metrize the compactification $\overline\H\cup\{\infty\}$ e.g. by $d(z,w)=\min(|z-w|,|z^{-1}-w^{-1}|)$; then consider $C^\infty([0,1],{\overline \H}\cup\{\infty\})$ restricted to non-self-traversing paths with $\gamma(0)=0$, $\gamma(1)=\infty$ with semidistance given by
$$d_{\mc P}(\gamma_1,\gamma_2)=\inf_{\phi:[0,1]\nearrow [0,1]}\sup_{t\in [0,1]}d(\gamma_1(t),\gamma_2(\phi(t)))$$
We say that $\gamma_1\sim\gamma_2$ if $\gamma_1=\gamma_2\circ\phi$ for some increasing reparameterization $\phi:[0,1]\rightarrow [0,1]$. Then 
$${\mc P}(\H,0,\infty)=(C^\infty([0,1],{\overline \H}\cup\{\infty\})/\sim,d_{\mc P})$$
is a Polish (complete, metric, separable) space.

Chordal SLE induces (via the trace $\gamma$) a probability measure $\mu^\sharp_{(\H,0,\infty)}$ on the path space ${\mc P}(\H,0,\infty)$. If $\kappa\leq 4$, this measure is supported on the Borel subset of simple paths which intersect the boundary only at their endpoints. The measure $\mu^\sharp_{(\H,0,\infty)}$ is invariant under scaling. Consequently, if $(D,x,y)$ is a bounded simply-connected domain with two marked boundary points $x,y$, and $\phi:(\H,0,\infty)\rightarrow (D,x,y)$ a conformal equivalence, then $\phi$ induces a quasi-isometry ${\mc P}(\H,0,\infty)\rightarrow {\mc P}(D,x,y)$ and one may define $\mu^\sharp_{(D,x,y)}\stackrel{def}{=}\phi_*\mu^\sharp_{(\H,0,\infty)}$ the image measure.

If $\kappa\leq 4$, we parameterize
\begin{align*}
\tau&=\frac 4\kappa=\frac {4(m+1)}{m}\\
c&=1-\frac 6{m(m+1)}=\frac{(6-\kappa)(3\kappa-8)}{2\kappa}\\
h&=h_{2,1}(\tau)=\frac{(1-m)^2-1}{4m(m+1)}=\frac{6-\kappa}{2\kappa}
\end{align*}
where $m>0$.

The fundamental {\em restriction property} \cite{LSW3} states that if $(D,x,y)$ is a simply-connected domain, $(D',x,y)$ is a simply-connected subdomain which agrees with $(D,x,y)$ in neighborhoods of $x,y$, then:
\begin{equation}\label{eq:SLErestr}
d\mu^\sharp_{(D',x,y)}(\gamma)= \ind_{\gamma\subset D'} \left(\frac{H_{D}(x,y)}{H_{D'}(x,y)}\right)^h\exp(\frac c2\nu_D(\gamma;D\setminus D'))d\mu^\sharp_{(D,x,y)}(\gamma)
\end{equation}
where both sides are considered as measures on ${\mc P}(D',x,y)$. 
Here $H_D(x,y)$ is the Poisson excursion kernel \eqref{eq:Poissonexcdef} (the kernel of the Dirichlet-to-Neumann map in $D$):
$$H_D(x,y)=\partial_{n_xn_y}G_D(x,y)$$
and $\nu$ is the Brownian loop measure in $D$, and we denote:
$$\nu_D(A;B)=\nu\{\delta\subset D: \delta\cap A\neq\varnothing,\delta\cap B\neq\varnothing\}$$

\subsection{Localization in path space}

\subsubsection{Construction}

Let us consider a compact bordered Riemann surface $(\Sigma,X,Y)$, where $X,Y$ are marked boundary points. We consider the path space ${\mc P}(\Sigma,X,Y)$ of continuous paths from $X$ to $Y$ given up to reparameterization. Let us fix a reference simple path $\gamma_0$ from $X$ to $Y$ which lies in the interior of $\Sigma$  except at its endpoints. Let ${\mc P}_0(\Sigma,X,Y)$ be the pathwise connected component of ${\mc P}(\Sigma,X,Y)$ consisting of paths homotopic to $\gamma_0$ (a homotopy induces a path in path space); and ${\mc P}_0^s(\Sigma,X,Y)$ be the Borel subset of ${\mc P}_0(\Sigma,X,Y)$ consisting of simple paths which intersect the boundary only at their endpoints. The goal is now to define a suitable SLE-type measure $\mu_0$ on ${\mc P}^s_0(\Sigma,X,Y)$.

Let us define a {\em tube neighborbood} in $(\Sigma,X,Y)$ as a simply-connected,  relatively open domain $D$ which contains a semidisk neighborhood of $X$ and of $Y$. Trivially, ${\mc P}_0^s(\Sigma,X,Y)$ is covered by relatively open sets of type ${\mc P}^s(D,X,Y)$, where  $D$ is a tube neighborhood (note that all paths in ${\mc P}^s(D,X,Y)$ are in the same homotopy class). Let $\phi_D: (\H,0,\infty)\rightarrow (D,X,Y)$ a conformal equivalence; then $\mu^\sharp_{(D,X,Y)}=(\phi_D)_*\mu^\sharp_{(\H,0,\infty)}$ (the image measure of $\mu^\sharp_{(\H,0,\infty)}$ under $\phi_D$) gives a reference measure on ${\mc P}^s(D,X,Y)$; for simplicity of notation we simply denote $\mu^\sharp_D=\mu^\sharp_{(D,X,Y)}$. Let us set:
\begin{equation}\label{eq:locdef}
\ind_{\gamma\subset D}d\mu_0(\gamma)=\phi^\Sigma_D(\gamma)d\mu^\sharp_D(\gamma)
\end{equation}
where the Radon-Nikodym derivative $\phi^\Sigma_D$ is to be determined. For the definition to be consistent, we need:
\begin{equation}\label{eq:consist}
\phi^\Sigma_{D'}\frac{d\mu^\sharp_{D'}}{d\mu_D^\sharp}=\phi^\Sigma_{D}
\end{equation}
(a.e.) on ${\mc P}^s(D',X,Y)$ whenever $D'\subset D$ are tube neighborhoods containing paths homotopic to $\gamma_0$; notice that $\frac{d\mu^\sharp_{D'}}{d\mu_D^\sharp}$ is an explicit functional given by the restriction property of SLE \eqref{eq:SLErestr}.

Let $D_1,D_2$ be two tube neighborhoods. If $\gamma$ is a path (from $X$ to $Y$, as are all paths we are considering here) in $D_1\cap D_2$, let $D_\gamma$ be the union of the ranges of paths in $D_1\cap D_2$ homotopic to $\gamma$. Plainly, $D_\gamma$ is a tube neighborhood contained in $D_1$ and $D_2$. We have ${\mc P}^s(D_1)\cap{\mc P}^s(D_2)=\cup_\gamma {\mc P}_0^s(D_\gamma)$, where the RHS is a union of disjoint path spaces (a priori two paths may be homotopic in $D_1$ and in $D_2$ but not in $D_1\cap D_2$). Hence provided \eqref{eq:consist} is satisfied for all pairs of tube neighborhoods, then the measure $\phi^\Sigma_{D_i}d\mu^\sharp_{D_i}$ on ${\mc P}^s(D_i)$, $i=1,2$, agree on the intersection ${\mc P}(D_1)\cap{\mc P}(D_2)$.

A grid approximation argument shows that ${\mc P}_0^s(\Sigma)=\cup_{i=1}^\infty {\mc P}^s(D_i)$ for a well-chosen sequence of tube neighborhoods $(D_n)_{n\geq 1}$. By inclusion-exclusion, under \eqref{eq:consist} the measures $\phi^\Sigma_{D_i}d\mu^\sharp_{D_i}$ extend consistently to $\cup_{i=1}^n{\mc P}^s(D_i)$. By monotone limit, they extend uniquely to ${\mc P}_0^s(\Sigma)$. Trivially (also under \eqref{eq:consist}), the resulting measure does not depend on the choice of $D_i$'s.

Our task is now to find densities $\phi^\Sigma_D$ satisfying the consistency condition
\begin{equation}\label{eq:consistcond}
\frac{\phi^\Sigma_{D'}}{\phi^\Sigma_D}(\gamma)=\left(\frac{d\mu^\sharp_{D'}}{d\mu^\sharp_{D}}\right)^{-1}(\gamma)
=\ind_{\gamma\subset D'} \left(\frac{H_{D}(X,Y)}{H_{D'}(X,Y)}\right)^{-h}\exp(-\frac c2\nu_D(\gamma;D\setminus D'))
\end{equation}
for $D'\subset D$ tube neighborhoods. Since $D$ is given as a subset of $\Sigma$ and not of the plane, $H_D(X,Y)$ is well-defined as a 1-form in $X$ and in $Y$ (the normal derivative at $X$ and $Y$ is defined in terms of local coordinates at $X$ and $Y$). Intrinsically, $H_D(X,Y)=\ast d_X\ast d_Y G_D(X,Y)$, i.e. we take the differential of the Green kernel (which depends only on the complex structure) w.r.t. each variable. If we extend the marking of $(\Sigma,X,Y)$ to include a 1-jet of local coordinate at $X$ and $Y$, then $H_D(X,Y)$ is a function of this data. Remark that the ratio $\frac{H_D(X,Y)}{H_{D'}(X,Y)}$ does not depend on these choices, provided they are the same for $D$ and $D'$. In the case $c=0$, we may simply set:
$$\phi^\Sigma_D(\gamma)=H_D(X,Y)^h$$
which does not depend on $\gamma$ (this reflects the restriction property of $\SLE_{8/3}$). In the general case, one may set:
\begin{equation}\label{eq:candens}
\phi^\Sigma_D(\gamma)=H_D(X,Y)^h\exp(-\frac c2\nu_\Sigma(\gamma;\Sigma\setminus D))
\end{equation}
By the restriction property of the loop measure (Proposition \ref{Prop:loopmeas}), this prescription satisfies \eqref{eq:consist} and consequently we have constructed a measure $\mu_0=\mu_0(\Sigma,X,Y)$ on the path space ${\mc P}_0^s(\Sigma,X,Y)$, which we refer to as the {\em canonical} $\SLE_\kappa$ measure on ${\mc P}^s_0(\Sigma,X,Y)$ ($\kappa\in (0,4]$). It is a positive measure (and not in general a probability measure) for given 1-jets at $X,Y$; more intrinsically, it is a measure-valued tensor (or tensor-valued measure).

We obtain immediately the following restriction property: if $\Sigma'\subset\Sigma$ agrees with $\Sigma$ in a neigborhood of $X,Y$ (and other markings), then
$$d\mu_{\Sigma'}(\gamma)=\ind_{\gamma\subset\Sigma'}%
\exp\left(\frac c2\nu_\Sigma(\gamma;\Sigma\setminus\Sigma')\right)d\mu_\Sigma(\gamma)$$

Remark however that \eqref{eq:consist} may have multiple solutions. Given another solution $\tilde\phi$ (notice that $\phi^\Sigma_D$ is positive on ${\mc P}_0^s(D,X,Y)$), we see that $\tilde\phi^\Sigma_D(\gamma)/\phi^\Sigma_D(\gamma)$ does not depend on the choice of tube neighborhood $D$ of $\gamma$ and consequently may be written as a function of the path: $\tilde\phi^\Sigma_D(\gamma)=h_\Sigma(\gamma)\phi^\Sigma_D(\gamma)$ and $\tilde\mu_0=h_\Sigma\mu_0$. This is interesting in particular when $h_\Sigma(\gamma)$ depends only on the conformal type of $\Sigma\setminus\gamma$ (forgetting the marked points $X,Y$). For instance, in the case where $\Sigma$ is an annulus and $X,Y$ belong to the same boundary component, $\Sigma\setminus\gamma$ consists in a simply-connected domain and an annulus; let $r(\gamma)$ be the modulus of this last annulus. Then we may choose $h_\Sigma(\gamma)=g(r(\gamma))$, $g:(0,\infty)\rightarrow (0,\infty)$ an arbitrary bounded measurable function.

\subsubsection{Comparison}\label{ss:comp}

The construction presented here (for the ``canonical" measures) is essentially identical to the one given by Lawler in \cite{Law_multdef}, with minor nuances. In \cite{Law_multdef}, domains are planar, all tensors are evaluated w.r.t. the ambient coordinate (given by the planar embeddings of the domains) and the measures are not decomposed by isotopy type. Note that ``non-canonical" measures (viz. obtained by patching but with another collection of densities satisfying \eqref{eq:consistcond}) may also be of interest, see e.g. \cite{HBB_GFF}. 

In \cite{KontSuh}, Kontsevich and Suhov also use localization for measures on loops and paths (``intervals"). For the reader's convenience, we give a brief account of (parts of) this work and verify that, although phrased in a somewhat different formalism, that construction is also equivalent to the one discussed here.

To a surface $\Sigma$ with boundary components, one associates a line (real, one-dimensional vector space) $\det_\Sigma$ as follows. A metric (compatible with the complex structure of $\Sigma$) is {\em well-behaving} if it is flat near each boundary circle, and these circles are geodesic and have fixed length $2\pi$. (We do not consider surfaces with punctures, nor with infinitely many boundary components - and with at least one boundary component for simplicity). If $g,g'=e^{2\sigma}g$ are such metrics, $\det_\Sigma=\R[g]=\R[g']$, where by definition
$$[g']=\exp\left(-\frac{1}{12\pi}\int_\Sigma(\frac 12|\nabla_g\sigma|^2+R_g\sigma)dA_g\right)[g]$$  
Note that $\det_\Sigma$ has a natural orientation; and that $\det_{\Sigma_1\sqcup\Sigma_2}\simeq\det_{\Sigma_1}\otimes\det_{\Sigma_2}$. Note also that if $V$ is an oriented line (bundle) and $c\in\R$, one can define $V^{\otimes c}$ an oriented line (bundle) in such a way that $V^{\otimes c}\otimes V^{\otimes c'}\simeq V^{\otimes (c+c')}$ canonically.

Alternatively, one can consider the space of functions on metrics satisfying the conformal anomaly relation \eqref{eq:confanom}:
$$f(e^{2\sigma}g)=\exp\left(\frac{1}{12\pi}\int_\Sigma(\frac 12|\nabla_g\sigma|^2+R_g\sigma)dA_g\right)f(g)$$
The identification is by setting
$$[g](e^{2\sigma}g)=\exp\left(\frac{1}{12\pi}\int_\Sigma(\frac 12|\nabla_g\sigma|^2+R_g\sigma)dA_g\right)$$
Consequently, $g\mapsto \det_{\zeta}(-\Lap_g)^{-1/2}$ defines an element of $\det_\Sigma$; we refer to it as the {\em canonical vector} $v_\Sigma$. Thus one can think of the formal vector $[g]$ above as a class of equivalence of well-behaving metrics in the same conformal class and with the same $\zeta$-determinant.

Now consider $(\Sigma,X,Y)$ a surface with two marked boundary points and ${\mc P}={\mc P}(\Sigma,X,Y)$ the associated path space (``intervals" in the terminology of \cite{KontSuh}). To $\gamma\in{\mc P}$ is associated the line
$${\det}_{\Sigma,\gamma}={\det}_D\otimes{\det}_{D\setminus\Sigma}^{-1}$$
where $D$ is a tube neighborhood of $\gamma$. To see that this does not depend on $D$, choose $D'\subset D$ another such tube neighborhood. Then there is a canonical isomorphism
$${\det}_D\otimes{\det}_{D\setminus\Sigma}^{-1}\simeq
{\det}_{D'}\otimes{\det}_{D'\setminus\Sigma}^{-1}$$
Indeed, consider $g,g_l,g_r$ well-behaved metrics on $D,(D\setminus\gamma)_l,(D\setminus\gamma)_r$ respectively (here $(D\setminus\gamma)_{l/r}$ is the left/right component of $\Sigma\setminus\gamma$), chosen to agree outside of a neighborhood of $\gamma$. Similarly, we choose metrics $g',g'_l,g'_r$ on $D',(D'\setminus\gamma)_{l/r}$ that agree away from $\gamma$. Finally we assume that $g,g'$ (resp. $g_l,g'_l$, resp. $g_r,g'_r$) agree near $\gamma$. Then we identify
$$[g]/[g_{lr}]\in{\det}_D\otimes{\det}_{D\setminus\Sigma}^{-1}\simeq
[g']/[g'_{lr}]\in{\det}_{D'}\otimes{\det}_{D'\setminus\Sigma}^{-1}$$
This is a canonical isomorphism since different choices lead to multiplying both sides by the same factor (by locality of the Liouville action). Consequently $\det_{\Sigma,\gamma}$ depends only on an infinitesimal neighborhood of $\gamma$ in $\Sigma$.

At this point we have defined a collection of lines $\det_{\Sigma,\gamma}$ indexed by ${\mc P}$, which can be seen as a line bundle $\Det_\Sigma$ over the path space. This line bundle may be trivialized in the following way. The previous argument does not depend on $D$ being simply-connected and consequently we have canonically:
$${\det}_{\Sigma,\gamma}\simeq{\det}_\Sigma\otimes{\det}_{\Sigma\setminus\gamma}^{-1}$$
In turn we can trivialize the RHS using canonical vectors (which we define here in terms of $\zeta$-regularization, a deviation from \cite{KontSuh}): $v_{\Sigma,\gamma}\stackrel{def}{=}v_\Sigma\otimes v_{\Sigma\setminus\gamma}^{-1}$ is a non-vanishing section which trivializes $\Det_\Sigma$.

If $\xi:\Sigma\hookrightarrow\Sigma'$ is an embedding (mapping semidisk neighborhoods of $X,Y$ to semidisk neighborhoods of $X',Y'$), there is an induced map on path space:
$$\xi_*:{\mc P}(\Sigma)\longrightarrow{\mc P}(\Sigma')$$
and since $\det_{\Sigma,\gamma}$ is defined in terms of an infinitesimal fattening of $\gamma$, we also have a canonical map
$$\xi_{\det}:(\xi_*)^*\Det_{\Sigma'}\longrightarrow\Det_{\Sigma}$$

The following problem is addressed in \cite{KontSuh}: to show existence and uniqueness (up to multiplicative factor) of an assignment of $\lambda_{(\Sigma,X,Y)}$, a measure on ${\mc P}(\Sigma,X,Y)$ taking values in the bundle $\Det_\Sigma^{\otimes c}\otimes (T_X\Sigma\otimes T_Y\Sigma)^{-h}$ in such a way that for any embedding $\xi:(\Sigma,X,Y)\hookrightarrow(\Sigma',X',Y')$ as above,
$$\xi^*\lambda_{(\Sigma',X',Y')}=\lambda_{(\Sigma,X,Y)}$$
Recall that $T_X\Sigma$ is an oriented line (vectors tangent to the boundary, oriented so that $\Sigma$ lies to their left), so that $(T_X\Sigma)^{-h}$ is still well-defined for $h\notin\Z$ . In \cite{KontSuh} it is shown that such an assignment exists for $c\leq 1$, $h=h_{2,1}(\tau)$ and conjectured that it is unique up to multiplicative constant for these values.

In order to compare with our earlier discussion, we trivialize:
$$d\lambda_{(\Sigma,X,Y)}(\gamma)=d\mu_{(\Sigma,X,Y)}(\gamma)(v_{\Sigma}\otimes v_{\Sigma\setminus\gamma}^{-1})^{\otimes c}$$
where now $\mu$ has simply a tensor dependence at $X,Y$ (viz. is a measure once 1-jets at $X,Y$ are fixed). In order to quantify the covariance property for the $\mu$'s implied by that of the $\lambda$'s, we set $\gamma\in{\mc P}(\Sigma,X,Y)$, $\gamma'=\xi\circ\gamma\in {\mc P}(\Sigma',X',Y')$. Then we have: a canonical isomorphism $\phi:\det_{\Sigma,\gamma}\rightarrow\det_{\Sigma',\gamma'}$; a vector $v_{\Sigma,\gamma}$ in $\det_{\Sigma,\gamma}$; and a vector $v_{\Sigma',\gamma'}$ in $\det_{\Sigma',\gamma'}$. Then
$$\frac{d\xi^*\mu_{\Sigma'}}{d\mu_{\Sigma}}(\gamma)=\left(\frac{\phi(v_{\Sigma,\gamma})}{v_{\Sigma',\gamma'}}\right)^{c}$$
Let $g_1,g_2,g_3,g_4$ be well-behaving metrics in $\Sigma$, $\Sigma\setminus\gamma$, $\Sigma'$, $\Sigma'\setminus\gamma'$ respectively. We assume that $\xi_*g_1$ and $g_3$ (resp. $\xi_*g_2$ and $g_4$) agree near $\gamma$; and that $g_1$ and $g_2$ (resp. $g_3$ and $g_4$) agree away from $\gamma$. (More precisely, $\Sigma=U\cup V$, $U$ a tube neighborhood of $\gamma$, $\Sigma'=U'\cup V'$, $U'=\xi(U)$, $g_3=\xi_*g_1$ on $U'$, $g_1=g_2$ on $V$, etc.). Then
$$\left(\frac{\phi(v_{\Sigma,\gamma})}{v_{\Sigma',\gamma'}}\right)^{-2}=\frac{{\det}_\zeta(-\Lap_{g_1}){\det}_\zeta(-\Lap_{g_4})}{{\det}_\zeta(-\Lap_{g_2}){\det}_\zeta(-\Lap_{g_3})}$$
where the RHS does not depend on (consistent) choices. If $\xi$ is defined on a neighborhood of $\bar U$, we can rewrite the RHS as
$$\exp(-\nu_\Sigma(\gamma;\Sigma\setminus U)+\nu_{\Sigma'}(\gamma';\Sigma'\setminus U'))$$
reasoning as in Proposition 2.1 of \cite{Dub_SLEGFF}. In particular if $\Sigma\subset\Sigma'$ (and $\xi$ is the inclusion), we have
$$\frac{d\mu_{\Sigma}}{d\mu_{\Sigma'}}(\gamma)=\ind_{\gamma\subset\Sigma}\exp(\frac c2\nu_{\Sigma'}(\gamma;\Sigma'\setminus\Sigma))$$
This shows that the ``canonical" measures considered here (summed over isotopy types) are identical (given this trivialization) to those constructed in \cite{KontSuh}.

\subsection{Partition function}\label{ssec:partfun}

We consider again $(\Sigma,X,Y)$ a bordered surface with marked boundary points $X,Y$ (possibly additional spectator points are marked on the boundary and in the bulk); $\gamma_0$ is a reference simple path from $X$ to $Y$ on $\Sigma$; ${\mc P}^s_0(\Sigma,X,Y)$ is the space of simple paths homotopic to $\gamma_0$. We have constructed a canonical positive $\SLE_\kappa$ measure on ${\mc P}^s_0(X,Y)$, $\kappa\leq 4$, specified by \eqref{eq:locdef} and \eqref{eq:candens}. We are interested in the {\em partition function} of this measure, i.e. its total mass:
$${\mc Z}_0(\Sigma,X,Y)=\|\mu_0(\Sigma,X,Y)\|=\mu_0({\mc P}^s_0(\Sigma,X,Y))\in (0,\infty]$$
Recall that ${\mc Z}_0$ has a weight $h$ tensor dependence in the local coordinates at $X,Y$; we want to show that ${\mc Z}_0$ defines a smooth function on the Teichm\"uller space ${\mc T}_{1,1}$ (of surfaces with marked points $X,Y$ and marked 1-jets at $X,Y$; other markings are kept implicit).

First we argue that, if $c\leq 0$, ${\mc Z}_0$ is locally bounded on ${\mc T}_{1,1}$. The Poisson kernels $H_D(X,Y)$ may be explicitly represented in terms e.g. of theta functions (see \eqref{eq:Poissonexcdef}) and are not problematic. The universal cover of the bordered surface $\Sigma$ has interior the upper half-plane $\H$ (as a bordered surface, the universal cover is contained in $\overline\H$; a boundary cycle in $\Sigma$ lifts to a countable union of disjoint open intervals of $\R$). Consider a lift $\tilde\gamma_0$ of $\gamma_0$ to the universal cover, with endpoints $\tilde X$, $\tilde Y$; and $\tilde\gamma$ the lift of a path $\gamma$ isotopic to $\gamma_0$ with the same endpoints.

We have the trivial bound
$$\nu_\Sigma(\gamma;\Sigma\setminus D)
\leq \nu_\Sigma\{\delta\subset \Sigma: \delta\cap\gamma\neq\varnothing,\delta\cap \Sigma\setminus D\neq\varnothing, \delta{\rm\ contractible}\}+\nu_\Sigma\{\delta: \delta{\rm\ non\ contractible}\}$$
The second term is locally bounded on ${\mc T}$. Indeed we may choose a smooth family of Riemannian metrics for surfaces in an open neighborhood $U$ of the (class of) $\Sigma$ in ${\mc T}$. Then the systole (length of the shortest noncontractible loop) is locally bounded away from $0$ on $U'\subset\subset U$; the volume is bounded on $U'$; and the probability of hitting the boundary before time 1 for Brownian motion is bounded away from $0$ on $U'$ (and all possible starting points). Together e.g. with a Donsker-Varadhan large deviation estimate for short time diffusions gives the claimed upper bound (a Brownian loop is exponentially unlikely to travel to macroscopic distance in time $1/N\ll 1$ or to stay away from the boundary up to time $N\gg 1$, with uniform exponential bounds on $U'$).
 
Consequently, we have the local (in Teichm\"uller space) estimate $\phi_D^\Sigma(\gamma)\leq C\hat\phi_D^\Sigma(\gamma)$ where
$$\hat\phi_D^\Sigma(\gamma)=H_D(X,Y)^h\exp(-\frac c2\nu_\Sigma \{\delta\subset \Sigma: \delta\cap\gamma\neq\varnothing,\delta\cap \Sigma\setminus D\neq\varnothing, \delta{\rm\ contractible}\})$$
Let $\hat\mu_0$ be the measure given by $\hat\phi^\Sigma_Dd\mu_D^\sharp$ on ${\mc P}(D,X,Y)$ (as noted earlier, this solves \eqref{eq:consist} and therefore patches to a positive measure on ${\mc P}_0^S(\Sigma,X,Y)$), so that ${\mc Z}_0(\Sigma,X,Y)\leq C\|\hat\mu_0\|$.

Let $\tilde D$ be the preimage of $D$ in $\H$; it is a tube neighborhood of $\tilde\gamma_0$. We observe that
$$\nu_\Sigma(\gamma;\Sigma\setminus D)^{contr}\leq \nu_\H(\tilde\gamma;\H\setminus\tilde D)
$$
where the LHS counts only contractible loops. The difference comes from loops in $\H$ that intersect multiple lifts of $\gamma$; a loop (in $\Sigma$) intersecting exactly $k$ lifts of $\gamma$ (including the ``distinguished" one $\tilde\gamma$) is counted $k$ times on the RHS and once on the LHS. Then, if $c\leq 0$, $\tilde\phi_D^\Sigma(\gamma)\leq \phi_{\tilde D}^\H(\tilde\gamma)$ and consequently the pullback $\hat\mu_0$ of $\tilde\mu_0$ to ${\mc P}(\tilde D,\tilde X,\tilde Y)$ is dominated by $\phi^\H_{\tilde D}d\mu^\sharp_{\tilde D}$. By the restriction property \eqref{eq:SLErestr}, this is simply the restriction of $H_{\H}(\tilde X,\tilde Y)^h\mu^\sharp_\H$ (the local coordinates at $\tilde X,\tilde Y$ are pullbacks of the marked local coordinates at $X,Y$) to ${\mc P}(\tilde D,\tilde X,\tilde Y)$.

By patching the measure restricted to the ${\mc P}(\tilde D,X,Y)$, we see that $\hat\mu_0$ is dominated by the restriction of $H_{\H}(\tilde X,\tilde Y)^h\mu^\sharp_\H$ to the preimage of ${\mc P}^s(\Sigma,X,Y)$ in ${\mc P}^s(\H,\tilde X,\tilde Y)$ (i.e. simple paths in $\H$ which stay simple when projected down on $\Sigma$). Consequently $|\hat\mu_0|=|\tilde\mu_0|\leq H_{\H}(\tilde X,\tilde Y)^h$.

It is easy to see that $H_\Sigma(X,Y)=\sum_{Y':\pi(Y')=Y}H_\H(\tilde X,\tilde Y)$ (starting e.g. from the Green function), where $\pi:\H\rightarrow\Sigma$ is the covering map, and the Poisson kernels are evaluated with consistent choices of local coordinates. Since $h>0$ and $H_\Sigma(X,Y)$ is smooth on ${\mc T}_{1,1}$, we conclude that, if $c\leq 0$, $(\Sigma,X,Y)\mapsto {\mc Z}_0(\Sigma,X,Y)$ is locally bounded on ${\mc T}_{1,1}$.

It will be convenient to have some basic {\em a priori} regularity estimate on ${\mc Z}_0$; let us show it is lower semicontinuous on ${\mc T}_{1,1}$. Let us choose a countable cover of ${\mc P}_0^s(\Sigma,X,Y)$ by ${\mc P}^s_0(D_i,X,Y)$ where the $D_i$'s are tube neighborhood which intersect $\partial\Sigma$ only in prescribed boundary arcs around $X,Y$.  We know that $\mu_0=\mu_0(\Sigma,X,Y)$ is a finite measure; so for $\eps>0$ there is $n$ large s.t. $\mu_0(\cup_{i=1}^n{\mc P}^s_0(D_i,X,Y))\geq {\mc Z}_0(\Sigma,X,Y)-\eps$. We can then find $Z\in\partial\Sigma$, $D^+$ a semidisk neighborhood of $Z$ s.t. $D_1,\dots,D_n$ are at positive distance of $D^+$. Then we can represent a neighborhood of the class of $(\Sigma,X,Y)$ in ${\mc T}$ by deformations of the gluing data of $D^+$ with $\Sigma\setminus\{Z\}$ (equivalently, by smooth variation of the metric in $D^+$). Then clearly for fixed $\gamma\in D_i$, $i\in\{1,\dots,n\}$, $\phi^{\Sigma'}_{D_i}(\gamma)$ is continuous in $\Sigma'$, $\Sigma'$ close to $\Sigma$ in ${\mc T}$; and $\phi^{\Sigma'}_{D_i}/\phi^\Sigma_{D_i}$ is locally bounded on ${\mc P}^s_0(D_i,X,Y)$. It follows that for $\Sigma'$ close enough to $\Sigma$, ${\mc Z}_0(\Sigma',X,Y)\geq {\mc Z}_0(\Sigma,X,Y)-2\eps$, as claimed.

As a by-product of this argument, let us observe that the family of measures $(\Sigma,X,Y)\mapsto\mu_0(\Sigma,X,Y)$ is regular in the following sense. Let us fix a smooth surface $\Sigma^s$ and a smooth parametric family of Riemannian metrics representing a neighborhood of the Teichm\"uller surface $\Sigma$ on ${\mc T}_{1,1}$. Then we may regard the $\mu_0$'s as a parametric family of measures on the same Lusin space ${\mc P}^s_0(\Sigma^s,X,Y)$. If $D_i$ is a tube neighborhood as above, $E_i$ a Borel subset of ${\mc P}^s(D_i,X,Y)$, we have that $(\Sigma,X,Y)\mapsto\mu_0^{(\Sigma,X,Y)}(E_i)$ is continuous. By monotone limit, if $E$ is any Borel set in ${\mc P}(\Sigma^s,X,Y)$, $(\Sigma,X,Y)\mapsto\mu_0^{(\Sigma,X,Y)}(E)$ is Borel measurable.

\subsection{Disintegration and null vector equation}

Having constructed, on the one hand, a Virasoro action by geometric arguments and, on the other hand, $\SLE$ partition functions, we finally relate these notions by showing that the said partition functions satisfy a null-vector equation (Theorem \ref{Thm:part}). 

A defining property of chordal $\SLE$ is its Domain Markov property. The measures we built from chordal $\SLE$ inherit a path decomposition property (as these are positive, rather than probability, measures, it is improper to speak of Markov property there).

We then show that a (non-canonical) $\SLE$-type diffusion on Teichm\"uller space is hypoelliptic (for general background on diffusions, we refer the reader to \cite{StroockVaradhan,RW2}). Recall that a differential operator ${\mc D}$ on a manifold of the form 
$${\mc D}=\frac 12 X^2+Y$$
where $X$ and $Y$ are vector fields (seen as derivations) is said to satisfy the H\"ormander bracket condition (e.g. \cite{Stroock_PDE}) if $X,Y$ and their iterated brackets $[X,Y],[X,[X,Y]]$, etc. span the tangent space at every point. If so, ${\mc D}$ is {\em hypoelliptic} in the sense that any weak solution of ${\mc D}h=0$ is smooth, i.e. if $h$ is a distribution s.t. ${\mc D}h=0$ (as distribution) on an open set, then $h$ is smooth there.

From the path decomposition identities and standard results on the Dirichlet problem for hypoelliptic operators \cite{Bony_max}, we obtain smoothness of the partition functions and the null-vector equation (Theorem \ref{Thm:part}). Finally, we discuss how that relates to results in \cite{BauBer_martVir,BB_CFTSLE,BB_part,Kyt_Virmod}.

\subsubsection{Disintegration}

The chordal $\SLE$ measure has a Markovian property, which is inherited by the canonical $\SLE$ measures ${\mu}_0(\Sigma,X,Y)$. We proceed with describing this property, starting in a somewhat restricted framework. We want to disintegrate the measure $\mu_0^\Sigma$ with respect to some initial slit (up to first exit of a semidisk neighborhood of $X$, say). On each tube neighborhood $D$, by the $\SLE$ Markov property, the reference $\SLE$ can be decomposed as the concatenation of a stopped chordal $\SLE$ trace $\gamma^\tau$ and a chordal $\SLE$ in the slit domain $D\setminus\gamma^\tau$. This Markov property has also a natural compatibility with the restriction property. Using the Markov property in each tube neighborhood, we will relate $\mu^\Sigma_0$ to the corresponding measures on $\Sigma\setminus\gamma^\tau$.

Consider a compact semidisk neighborhood $D^+$ of $X$ in $\Sigma$, not containing any other marked point, and identified via a local coordinate $z$ to a neighborhood of 0 in $\H$. For $\gamma\in{\mc P}^s_0(\Sigma,X,Y)$, we denote by $\gamma^\tau$ the initial segment of the path up to first exit of $\gamma$ and $\gamma_\tau\in\partial D^+$ its endpoint; one may think of $\gamma^\tau$ as an element of a stopped path space ${\mc P}_\tau^s(\Sigma,X,Y)$, with the topology of uniform convergence up to time reparameterization. The mapping $R:\gamma\mapsto\gamma^\tau$ is not everywhere continuous (because of paths bouncing back on $\partial D^+$, a null set), but is Borel measurable. We may consider the marked surface $\Sigma_\tau=(\Sigma\setminus\gamma^\tau,\gamma_\tau,Y)$; the mapping $\gamma^\tau\mapsto (\Sigma\setminus\gamma^\tau,\gamma_\tau,Y)$ is continuous ${\mc P}_\tau^s(\Sigma,X,Y)\rightarrow{\mc T}$.

There is a ``natural" local coordinate at $\gamma_\tau$ on $\Sigma_\tau$ defined from the reference local coordinate at $z$ at $X$ (which maps $D^+$ to a neighborhood of $0$ in $\H$): consider a conformal equivalence $g_\tau:\H\setminus z(\gamma^\tau)\rightarrow\H$ with $g_\tau(z)=z+O(1)$ as $z\rightarrow\infty$. Then set $z_\tau=g_\tau\circ z-g_\tau(z(\gamma_\tau))$, a local coordinate at $\gamma_\tau$. This allows to evaluate tensors such as $H_{\Sigma_\tau}(\gamma_\tau,Y)$.

Let us consider a tube neighborhood $D$; we do not require $D^+\subset D$. Let $D'\subset D$ be another tube neighborhood and set $D_\tau=D\setminus\gamma^\tau$ (and likewise for $D'_\tau$). The $\SLE$ restriction property, in its martingale formulation \cite{LSW3}, implies that:
$$\frac{dR_*\mu^\sharp_{D'}}{dR_*\mu^\sharp_{D}}(\gamma^\tau)=\ind_{\gamma\subset D'}\left(\frac{H_{D'_\tau}(\gamma_\tau,Y)H_{D}(X,Y)}{H_{D_\tau}(\gamma_\tau,Y)H_{D'}(X,Y)}\right)^h
\exp(\frac c2\nu_D(\gamma^\tau;D\setminus D'))
$$
The Domain Markov property of $\SLE$ states that, under $d\mu^\sharp_{D_\tau}(\gamma')dR_*\mu^\sharp_D(\gamma^\tau)$, the concatenation $\gamma^\tau\bullet\gamma'$ has law $\mu^\sharp_D$. Consequently, on ${\mc P}^s(D,X,Y)$, we have
\begin{align*}
\ind_{{\mc P}^s(D,X,Y)}d\mu_0^\Sigma(\gamma)
&=\phi^\Sigma_D(\gamma)d\mu^\sharp_D(\gamma)=H_D(X,Y)^h\exp(-\frac c2\nu_\Sigma(\gamma;\Sigma\setminus D))d\mu^\sharp_D(\gamma)\\
&=\left(\frac{H_D(X,Y)}{H_{D_\tau}(\gamma_\tau,Y)}\right)^h\exp(-\frac c2\nu_\Sigma(\gamma^\tau;\Sigma\setminus D))dR_*\mu^\sharp_D(\gamma^\tau)
d\mu_0^{\Sigma_\tau}(\gamma')
\end{align*}
We observe that the measures $\frac{\phi^\Sigma_D}{\phi^\Sigma_{D_\tau}}dR_*\mu^\sharp_D$ on ${\mc P}_\tau(D,X,Y)$ (the space of paths in the tube $D$ stopped upon exiting $D_+$) are consistent by the restriction property at time $\tau$, and as before (see \eqref{eq:locdef}) patch up to a measure on ${\mc P}_\tau(\Sigma,X,Y)$, which we denote by $\mu_\tau^\Sigma$. Then we have the disintegration:
$$\ind_{{\mc P}^s(D,X,Y)}d\mu_0^\Sigma(\gamma)=
\ind_{{\mc P}_\tau(D,X,Y)}(\gamma^\tau)d\mu_\tau^\Sigma(\gamma^\tau)\ind_{{\mc P}(D_\tau,\gamma_\tau,Y)}(\gamma')d\mu^{\Sigma_\tau}_0(\gamma')$$
where $\gamma=\gamma^\tau\bullet\gamma'$; and consequently
\begin{equation}\label{eq:disintegr}
d\mu_0^\Sigma(\gamma)=
d\mu_\tau^\Sigma(\gamma^\tau)d\mu^{\Sigma_\tau}_0(\gamma')
\end{equation}
on ${\mc P}_0^s(\Sigma,X,Y)$. This implies in particular that the collection of probability measures $\Sigma\mapsto\mu^\sharp_\Sigma={\mc Z}(\Sigma)^{-1}\mu_0^\Sigma$ have the same type of Markov property as chordal $\SLE$. By integrating out $\gamma'$, we obtain
$$dR_*\mu_0^\Sigma(\gamma^\tau)={\mc Z}_0(\Sigma_\tau,\gamma_\tau,Y)d\mu_\tau^\Sigma(\gamma^\tau)$$
and thus ${\mc Z}_0(\Sigma_0,X,Y)=\int {\mc Z}_0(\Sigma_\tau,\gamma_\tau,Y)d\mu_\tau^\Sigma(\gamma^\tau)$. The measure $\mu_\tau^\Sigma$ is absolutely continuous w.r.t. a stopped chordal $\SLE$. Let us write this in a local chart. Again let $z$ be a local coordinate at $X$ which maps $\tilde D^+$ to a neighborhood in $\H$ (with $D^+\subset\subset \tilde D^+$); $\mu^\sharp_\H$ is the standard chordal $\SLE$ measure, $R_*\mu^\sharp_\H$ its projection on paths stopped upon their first exit of $z(D^+)$, and $z^*R_*\mu^\sharp_\H$ its pullback to ${\mc P}_\tau(\Sigma_0)$. Then
$$d\mu^\Sigma_\tau(\gamma^\tau)=
\left(\frac{H_{\Sigma_\tau}(\gamma_\tau,Y)}{H_\Sigma(X,Y)}\right)^h\exp(-c(\nu_\Sigma(\gamma^\tau;\Sigma\setminus \tilde D^+)
-\nu_\H(z(\gamma^\tau);\H\setminus z(\tilde D^+))))
d\left(z^*R_*\mu^\sharp_\H\right)(\gamma^\tau)$$
where the Poisson kernels are evaluated w.r.t. $z,z_\tau$. If $z$ extends to a conformal equivalence $D\rightarrow\H$, this is by construction of $\mu^\Sigma_\tau$; otherwise it follows from change of coordinate rules for chordal $\SLE$.

It follows that under the standard chordal $\SLE$ measure,
\begin{equation}\label{eq:partmart}
{\mc Z}_0(\Sigma_0,X,Y)=\E\left(\exp(c(\nu_\Sigma(z^{-1}(\gamma^\tau);\Sigma\setminus \tilde D^+)
-\nu_\H(\gamma^\tau;\H\setminus z(\tilde D^+)))){\mc Z}_0(\Sigma_\tau,\gamma_\tau,Y)\right)
\end{equation}
Since this is valid for all surfaces $\Sigma$, the same argument may be applied between fixed (in the half-plane parameterization) times to see that:
$$t\mapsto \exp(c(\nu_\Sigma(z^{-1}(\gamma^{t\wedge\tau});\Sigma\setminus \tilde D^+)
-\nu_\H(\gamma^{t\wedge\tau};\H\setminus z(\tilde D^+)))){\mc Z}_0(\Sigma_{t\wedge\tau},\gamma_{t\wedge\tau},Y)$$
is a (bounded) martingale under the standard chordal $\SLE$ measure (provided that ${\mc Z}_0$ is locally bounded). Remark that by martingale representation, this is a.s. continuous in $t$.

\subsubsection{Hypoellipticity}

We wish to use this representation \eqref{eq:partmart} to show that ${\mc Z}_0$ is smooth on Teichm\"uller space. A slight issue is that if $\gamma$ is the trace of a standard $\SLE$ in $\H$, $t\mapsto(\Sigma\setminus z^{-1}(\gamma^t),z^{-1}(\gamma_t),Y)$ is not a diffusion in Teichm\"uller space. In order to invoke standard hypoellipticity regularity conditions (e.g. \cite{Stroock_PDE}), we need to change the reference measure to a mutually absolutely continuous diffusion (which is itself {\em a priori} non canonical and used only as a technical intermediate step).

Recall (see \eqref{eq:candiffop}) that we have defined a differential operator
$$\Delta_{2,1}:C^\infty(U,|T^{-1}\Sigma|^{\otimes h}\otimes{\mc L}^{\otimes c})\rightarrow C^\infty(U,|T^{-1}\Sigma|^{\otimes (h+2)}\otimes{\mc L}^{\otimes c})
$$
and we denote by $s=s_\zeta^c$ the reference section of ${\mc L}^{\otimes c}$.

\begin{Lem}\label{Lem:hypo}
Let $\omega s$ be a smooth positive local section of $|T^{-1}\Sigma|^{\otimes h}\otimes{\mc L}^{\otimes c}$ on $U\subset{\mc T}$. Assume that $\Delta_{2,1}(\omega s)=0$ on $U$. Then
\begin{enumerate}
\item Under the chordal $\SLE$ measure, 
$$t\mapsto M_t=\exp(c(\nu_\Sigma(z^{-1}(\gamma^{t\wedge\tilde\tau});\Sigma\setminus \tilde D^+)
-\nu_\H(\gamma^{t\wedge\tilde\tau};\H\setminus z(\tilde D^+))))\omega(\Sigma_{t\wedge\tilde\tau},\gamma_{t\wedge\tilde\tau},Y)$$
is a martingale, where $\tilde\tau=\tau\wedge\inf\{t\geq 0: \Sigma_t\notin U'\}$, $U'\subset\subset U$.
\item The martingale transform of the chordal $\SLE$ by $M$ is (up to exit of $U'$ and up to time change) a diffusion with generator
$${\mc G}=\frac\kappa 2(\omega s)^{-1}\Delta_{2,1}(\omega s)$$
\item ${\mc G}$ is hypoelliptic.
\end{enumerate}
\end{Lem}
As explained earlier, at time $t$ the tensor $\omega$ is evaluated w.r.t. the local coordinate $z_t$.
\begin{proof}
\begin{enumerate}
\item
Trivially $M$ is bounded (since we stop at $\tilde\tau$). From the Markov property of chordal $\SLE$ and the restriction property of the loop measures, it is enough to check the condition at a fixed time, i.e. $\E(M_t)=1$.
The null vector equation $\Delta_{2,1}(\omega s)=0$ gives $(\ell_{-2}+cS-\frac4\kappa\ell_{-1}^2)\omega=0$. In order to translate this in terms of the reference chart $\SLE$, it appears convenient to use a discrete time approximation.

Let $\delta>0$ be a small time step. We consider a piecewise continuous trace sampled as follows: move the marked point from $X$ to $X\pm\sqrt{\kappa\delta}$ with equal probability; grow a vertical slit of size $\sqrt{2\delta}$ at $X$; repeat. The horizontal and vertical displacements are defined in terms of the local coordinates $z_t$. These displacements are readily identified with a unit time flow along $\pm\sqrt{\kappa\delta}\ell_{-1}$, $-2\delta\ell_{-2}$ respectively (recall the end of Section \ref{Sec:Wittrep}).

It is then elementary to verify that, on the one hand, $\E^\delta(M_t)=1+o(1)$ (as $\delta\searrow 0$, by Taylor expansion); and on the other hand $\gamma^\delta$ converges weakly to chordal $\SLE$ (w.r.t. the Carath\'eodory topology on chains). This establishes the first point.

\item
For the second point, one may consider a test function $f$ (supported on $U'$, say). Here we need to be specific about multiplicative constants. We write $\Delta_{2,1}=L_{-1}^2-\frac 4\kappa L_{-2}$; choose a 1-jet $\tilde z$ of local coordinate at $X$ for each $(\Sigma,X)$; and set $\rho_s=\frac{d\tilde z_{(\Sigma_s,X_s)}}{dz_s}(\gamma_s)$. Extending the previous argument, one sees that
$$N_t=M_t\left(f(\Sigma_t)-\int_0^t\rho^{-2}_s{\mc G}f(\Sigma_s)ds\right)$$
is a martingale (omitting the stopping at $\tilde\tau$ for simplicity of notation). This gives 2. by the martingale problem characterization of diffusions (see e.g. \cite{StroockVaradhan}).

\item
Finally we need to check the H\"ormander bracket conditions for ${\mc G}$ (see e.g. \cite{Stroock_PDE}). We have chosen a section $\sigma$ of ${\mc T}_{1^k}\rightarrow {\mc T}$ ($k$ large enough). This allows to define $\ell_{-1}^\sigma$, $\ell_{-2}^\sigma$ as vector fields on ${\mc T}$. The generator ${\mc G}$ is of the form:
$${\mc G}=\frac\kappa 2(\ell^\sigma_{-1})^2-2\ell^\sigma_{-2}+b\ell_{-1}^\sigma$$
We have to show that the H\"ormander condition:
$$\ell_{-1}^\sigma,\ell_{-2}^\sigma,[\ell_{-1}^\sigma,\ell_{-2}^\sigma], 
[\ell_{-1}^\sigma,[\ell_{-1}^\sigma,\ell_{-2}^\sigma]],\dots
$$ 
span the tangent vector space to ${\mc T}$ at $(\Sigma_0,X_0)$. As usual we represent tangent vectors to ${\mc T}$ as Laurent vector fields in a semiannulus around the marked point. By definition $\ell_{-1}^\sigma$ is represented by $-\frac{\partial}{\partial_{z}}$, and $\ell_{-2}^\sigma$ is represented by $-z^{-1}\frac{\partial}{\partial_{z}}$ ($z$ is the local coordinate given by $\sigma$). 

Let us choose smooth local coordinates for ${\mc T}_{1^k}$ as follows: $u_1,\dots,u_d$ are coordinates on the base ${\mc T}$ and $a_0,\dots,a_k$ are coordinates on the fiber (e.g. the coefficients of the marked $k$-jet at $X$ w.r.t. to a reference local coordinate). For $f\in C^\infty(U_k)$ (i.e. $f$ depends smoothly on the marked surface and a $k$-jet at $X$), we have defined
$$(\ell_m f)(u_1,\dots,u_d,a_0,\dots,a_{k'})=\sum_{i=1}^dg_i(u_1,\dots,a_{k'})\partial_{u_i}f+\sum_{j=0}^kh_j(u_1,\dots,a_{k'})\partial_{a_j}f$$
($k'=k+m^-$) so that the Witt commutation relations $[\ell_m,\ell_n]=(m-n)\ell_{m+n}$ are satisfied. 

For $f\in C^\infty(U)$, we have by construction
$$(\ell_m^\sigma f)(u_1,\dots,u_d)=\sum_{i=1}^d
g_i(u_1,\dots,u_d,\sigma_0(u),\dots,\sigma_{k'}(u))\partial_{u_i}f%
$$
where the $\sigma_i$'s are the coordinates of the chosen section of ${\mc T}_{1^k}\rightarrow{\mc T}$ and $u=(u_1,\dots,u_d)$. From this expression of $\ell_m^\sigma$ in coordinates, it is clear that for $f\in C^\infty(U)$
\begin{equation}\label{eq:lemhypo0} 
[\ell_m^\sigma,\ell^\sigma_n]f=(m-n)\ell_{m+n}^\sigma f\mod \langle\partial_{a_0}\ell_mf,\partial_{a_0}\ell_nf,\dots,\partial_{a_{k'}}\ell_mf,\dots,\partial_{a_{k'}}\ell_nf\rangle
\end{equation}
(with a slight abuse of notation, as the RHS is evaluated at the jet $\sigma(u)$). By this we mean that there are smooth functions $v_0,\dots,w_{k'}$ (independent of $f$) s.t.
$$[\ell_m^\sigma,\ell^\sigma_n]f=(m-n)\ell_{m+n}^\sigma f+v_0\partial_{a_0}\ell_mf+\cdots+w_{k'}\partial_{a_{k'}}\ell_nf$$
If $m\in\Z$ and $n\geq 0$,
$$(m-n)\ell_{m+n}f=[\ell_m,\ell_n]f=-\ell_n\ell_mf\in\langle\partial_{a_0}\ell_mf,\dots,\partial_{a_{n}}\ell_mf\rangle$$
We observe that vertical vector fields which preserve $j$-jets are spanned (over $C^\infty(U)$) by $\partial_{a_j},\dots,\partial_{a_k}$ or alternatively by $\ell_j^\sigma,\dots,\ell_k^\sigma$. It follows that
$$\langle\partial_{a_0}\ell_mf,\dots,\partial_{a_{n}}\ell_mf\rangle=\langle \ell_m^\sigma f,\dots,\ell_{m+n}^\sigma f\rangle$$
Then from \eqref{eq:lemhypo0}, if $m\leq n\leq 0$,
$$[\ell_m^\sigma,\ell^\sigma_n]=(m-n)\ell_{m+n}^\sigma\mod \langle\ell_m^\sigma,\ell_{m+1}^\sigma,\dots,\ell_{-1}^\sigma\rangle$$
In particular, the span of 
$$\ell_{-1}^\sigma,\ell_{-2}^\sigma,[\ell_{-1}^\sigma,\ell_{-2}^\sigma], 
[\ell_{-1}^\sigma,[\ell_{-1}^\sigma,\ell_{-2}^\sigma]],\dots
$$
is the span of
$$\ell_{-1}^\sigma,\ell_{-2}^\sigma,\ell_{-3}^\sigma,\ell_{-4}^\sigma,\dots$$
Then by Virasoro uniformization we have that for $n$ large enough $\ell_{-1}^\sigma,\ell_{-2}^\sigma,\dots,\ell_{-n}^\sigma$ spans the tangent space $T_\Sigma{\mc T}$ (since the map \eqref{eq:Virunifcomp} is surjective). 

Roughly speaking, we used the section $\sigma$ to replace the infinite-dimensional vector fields $\ell_n$'s by the finite-dimensional $\ell_n^\sigma$'s (amenable to H\"ormander's condition); by doing so this we lose the Witt commutation relations $[\ell_m,\ell_n]=(m-n)\ell_{m+n}$ but retain them ``up to lower order error", which is enough to conclude.
\end{enumerate}
\end{proof}

\subsubsection{Smoothness}

Let ${\mc G}$ be as in the lemma. From \cite{Bony_max} it follows that each point of ${\mc T}$ has a basis of neighborhoods in which the Dirichlet problem with continuous boundary data is uniquely solvable for the operator ${\mc G}+V$, $V$ a smooth bounded potential; moreover the solutions are smooth and satisfy a Harnack inequality. (Note that $\ell_{-1}^\sigma$ is nowhere vanishing provided sufficiently many spectator points are marked). This still holds if $\omega s$ is simply a smooth positive section (without assuming the null vector equation).

Let us start from an arbitrary smooth positive section $\omega_0 s$. Then ${\mc G}_0=\frac\kappa 2(\omega_0 s)^{-1}\Delta_{2,1}(\omega_0 s)$ has a zeroth order term $V$. If $U$ is a small enough such neighborhood (depending on $V$), we can solve the Dirichlet problem with arbitrary continuous positive boundary data; this gives $f\in C^\infty(U)$ a smooth positive section so that $\Delta_{2,1} (f\omega_0 s)=0$ there. We consider the diffusion and generator ${\mc G}$ associated to $(f\omega_0s)$. Now we also have unique smooth solutions and a Harnack inequality for solutions of the Dirichlet problem for ${\mc G}$ in $U$ with continuous boundary data (for well-chosen neighborhoods forming a basis). Let $P$ denotes the Poisson operator for ${\mc G}$ on $U$, i.e. if $g$ is continuous on $\partial U$, $Pg$ is continuous on $\bar U$, $(Pg)_{\partial U}=g$ and ${\mc G}(Pg)=0$ in $U$; in particular $Pg$ is smooth in $U$.

Let $U$ be such a small enough neighborhood and $\tau$ be the first exit time of $U$. By a barrier function argument, one sees that $\tau$ is a.s. finite with exponential tails. As is well-known (e.g. \cite{StroockVaradhan}), it then follows from Dynkin's formula and optional stopping that we have a probabilistic representation for $P$: $(Pg)(\Sigma)=\E_{\mc G}^\Sigma(g(\Sigma_\tau))$, where the expectation refers to the diffusion measure with generator ${\mc G}$ started from the state $\Sigma=(\Sigma,X,\dots,)$.

Comparing \eqref{eq:partmart} with Lemma \ref{Lem:hypo}, we see that
$$t\mapsto \frac{{\mc Z}_0}{f\omega_0}(\Sigma_t,\gamma_t,Y)$$
is a martingale (at least up to exit of $U$) and consequently
$$ \frac{{\mc Z}_0}{f\omega_0}(\Sigma)=\E_{\mc G}^\Sigma\left( \frac{{\mc Z}_0}{f\omega_0}(\Sigma_\tau)\right)$$
(for brevity $\Sigma=(\Sigma,X,Y,\dots)$).

We know a priori that ${\mc Z}_0$ is lower semicontinuous (Section \ref{ssec:partfun}). We may thus represent it as the monotone increasing limit of continuous functions ${\mc Z}_n$ (e.g. by taking ${\mc Z}_n$ to be the largest $n$-Lipschitz minorant of ${\mc Z}$). Then by monotone limit we have 
$$ \frac{{\mc Z}_0}{f\omega_0}(\Sigma)=\lim_n\E_{\mc G}^\Sigma\left( \frac{{\mc Z}_n}{f\omega_0}(\Sigma_\tau)\right)=\lim_n P({\mc Z_n}/f\omega_0s)(\Sigma)$$
for $\Sigma\in U$. From the Harnack inequality \cite{Bony_max}, it follows that ${\mc Z}_0$ is, say, Lipschitz (since $P({\mc Z}_n/f\omega_0)$ is uniformly Lipschitz on compact subsets of $U$). Consequently, ${\mc Z}_0/f\omega_0=P(({\mc Z}_0/f\omega_0)_{|\partial U})$ and then ${\mc Z}_0$ is smooth and satisfies:
$${\mc G}\frac{{\mc Z}_0}{f\omega_0}=0$$
on ${\mc T}$ in the classical sense, i.e. $\Delta_{2,1}({\mc Z}_0s)=0$.

We can now state the main result on the partition function of $\SLE$ measures.

\begin{Thm}\label{Thm:part}
Let $(\Sigma,X,Y,\dots)$ be a bordered Riemann surface with at least two marked points on the boundary and $\gamma_0$ a simple path from $X$ to $Y$. Let $\mu_0^\Sigma$ be the $\SLE$ measure on paths homotopic to $\gamma_0$ on $\Sigma$ and ${\mc Z}_0(\Sigma)=\|\mu_0^\Sigma\|$ its partition function; and $s$ be the reference section of ${\mc L}^{\otimes c}$. If ${\mc Z}_0$ is locally bounded, then ${\mc Z}_0s$ is a smooth section of $|T^{-1}\Sigma|^{\otimes h}\otimes {\mc L}^{\otimes c}$ and satisfies the null vector equation
$$\Delta_{2,1}({\mc Z}_0s)=0$$
\end{Thm}

\subsubsection{Local martingales}\label{sssec:locmart}

Let ${\mc Z}$ be any smooth positive (local) solution of the null-vector equation $\Delta_{2,1}({\mc Z}s)=0$ (examples of such solutions are given in Theorem \ref{Thm:part}). By Lemma \ref{Lem:hypo}, there is a hypoelliptic diffusion $(\Sigma_t)_{t\geq 0}$ (defined at least locally on Teichm\"uller space) with generator
$${\mc G}f=\frac\kappa 2({\mc Z}s)^{-1}\Delta_{2,1}(f{\mc Z}s)$$
Let us fix such a ${\mc Z}$, defined in some open set in Teichm\"uller space. 
By Dynkin's formula, if $f$ is s.t. ${\mc G}f=0$, then $t\mapsto f(\Sigma_t)$ is a local martingale. The reader will have noticed that ${\mc G}f$ depends on a 1-jet at the seed (via a multiplicative constant); correspondingly, $(\Sigma_t)$ is defined only up to time change, unless one specifies a choice of section of ${\mc T}_{1^1}\rightarrow{\mc T}$ as in Lemma \ref{Lem:hypo} (such a choice is in general non-canonical). The class of local martingales is invariant under such time change. 

By construction $f\equiv 1$ gives a solution of ${\mc G}f=0$. We now observe that families of local martingales may be generated using a commuting Virasoro representation, as in  
Section \ref{ssec:commrep}. 

Specifically, let $X$ denote the position of the $\SLE$ seed and $Y$ be another marked point (either the target or just a spectator point). Then we denote $(L_n^X)_{n\in\Z}$ (resp. $(L_n^Y)_{n\in\Z}$) the images of the Virasoro generators operating by deformation at $X$ (resp. $Y$). Then $\Delta_{2,1}^X({\mc Z}s)=0$ by assumption and
$[L_n^X,L_m^Y]=0$ for all $m,n\in\Z$. Let $L^Y$ be a word in the $L_m^Y$'s (with $m\leq 0$) and $f$ s.t. ${\mc G}f=0$. Then trivially
$$\Delta_{2,1}^X(L^Y(f{\mc Z}_s))=L^Y(\Delta_{2,1}^X(f{\mc Z}s))=0$$
i.e. $L^Yf\stackrel{def}{=}({\mc Z}s)^{-1}(L^Y(f{\mc Z}s))$ also satisfies ${\mc G}(L^Yf)=0$. Remark that $f$ depends on a $k$-jet at $Y$, where $k$ is the degree of $L^Y$; however, since the evolution is at $X$, one can simply choose a jet at $t=0$ and keep it fixed as $t$ increases. Note also that even starting from the trivial solution $f\equiv 1$, one thus generates a hierarchy of non-trivial local martingales $L^Y1$ indexed by elements $L$ in ${\mc U}(\Vir^{-})$.

This generalizes the Virasoro action on $\SLE$ local martingales introduced and analyzed in \cite{BauBer_martVir,BB_CFTSLE,BB_part,Kyt_Virmod} (recall the discussion in Section \ref{ssec:BB}).

\subsection{Multiple $\SLE$s}

The same method may be implemented in order to construct measures on $n$-tuples of paths connecting pairs of boundary points (multiples $\SLE$s) or a boundary point with a bulk point (radial case). Let us start with multiple $\SLE$s; the simply-connected case and some elements of the multiply-connected case were analyzed in \cite{Dub_Comm}.

We consider again a bordered surface $(\Sigma,X_1,\dots,X_n,Y_1,\dots,Y_n)$ with $n$ pairs $(X_i,Y_i)_{1\leq i\leq n}$ of marked boundary points (one may also keep track of additional marked points on the boundary or in the bulk). We want to construct an $\SLE$ measure on $n$-tuples of simple paths connecting these pairs. The state space is thus $\prod_{1\leq i\leq n}{\mc P}^s(\Sigma,X_i,Y_i)$ (or the subset consisting of disjoint paths). Let us fix $(\gamma^0_1,\dots,\gamma^0_n)$ a reference $n$-tuple of such disjoint paths; ${\mc P}^s_0(\Sigma,X_1,\dots,Y_n)$ consists of disjoint simple paths in $\Sigma$ jointly homotopic to $(\gamma^0_1,\dots,\gamma^0_n)$.

The path space ${\mc P}^s_0(\Sigma,X_1,\dots,Y_n)$ is covered by relatively open sets of the form:
$$\prod_{1\leq i\leq n}{\mc P}(D_i,X_i,Y_i)$$
where $D_i$ is a simply-connected domain containing semidisk neighborhoods of $X_i$ and $Y_i$ and no other marked point in its closure; we also require that the $\overline{D_i}$'s are pairwise disjoint; and we set $D=\sqcup_i D_i$. In a straightforward extension of the chordal case, in order to define $\mu_0$ a measure on ${\mc P}^s_0(\Sigma)$, it is enough to exhibit densities $\phi_D^\Sigma$ s.t.
$$(\prod_i\ind_{\gamma_i\subset D_i})d\mu_0(\gamma_1,\dots,\gamma_n)=\phi_D^\Sigma(\gamma_1,\dots,\gamma_n)\prod_id\mu^\sharp_{D_i}(\gamma_i)$$
satisfying the consistency condition (compare with \eqref{eq:consist})
\begin{equation}\label{eq:consistmult}
\frac{\phi_{D}^\Sigma}{\phi_{D'}^\Sigma}(\gamma)
=
\prod_i\frac{d\mu^\sharp_{D'_i}}{d\mu^\sharp_{D_i}}(\gamma_i)
\end{equation}
where $D'=\sqcup_i D'_i\subset D$ and $\gamma=(\gamma_1,\dots,\gamma_n)\in\prod_i{\mc P}^s(D'_i)$. 

A ``natural" solution is given by the specification:
$$\phi_D^\Sigma(\gamma)=(\prod_{1\leq i\leq n}H_{D_i}(X_i,Y_i)^h)
\exp(-\frac c2\nu_\Sigma (\cup\gamma_i;\Sigma\setminus D))$$
(compare with \eqref{eq:candens}).

We denote by $\mu_{(\Sigma,X_1,\dots,Y_n)}$ the resulting measure, viz. the one that restricts to $\phi_D\prod_i\mu_{D_i}^\sharp$ on ${\mc P}_0^s(\sqcup_i D_i)$. Then $\mu$ is a positive measure on ${\mc P}^s_0(\Sigma)$; let us denote ${\mc Z}_0={\mc Z}_0(\Sigma,X_1,\dots,Y_n)$ its total mass (which has a $h$-tensor dependence in local coordinates at $X_1,\dots,Y_n$). If $c\leq 0$, we see immediately that ${\mc Z}_0$ is locally bounded.

Alternatively, one may also want to use $\mu_{(\Sigma,X_1,Y_1)}\otimes \cdots\otimes \mu_{(\Sigma,X_n,Y_n)}$ as a reference measure. Observe that
\begin{align*}
\sum_i\nu_\Sigma(\gamma_i,\Sigma\setminus D_i)-\nu_\Sigma(\cup\gamma_i;\Sigma\setminus D)
&=\int(\ind_{\gamma_1\cap\delta\neq\varnothing}+\cdots+\ind_{\gamma_n\cap\delta\neq\varnothing}-1)d\nu_\Sigma(\delta)\\
&=\sum_{j=2}^n\nu_\Sigma(\gamma_j;\cup_{i<j}\gamma_i)
\end{align*}
is independent of the choice of $D_i$'s (and the last expression is invariant under relabelling of the paths). Then
\begin{align*}
\mu_{(\Sigma,X_1,\dots,Y_n)}&=\ind_{{\mc P}_0^s}\exp\left(\frac c2\sum_{j=2}^n\nu_\Sigma(\gamma_j;\cup_{i<j}\gamma_i)\right)\prod_i\mu_{(\Sigma,X_i,Y_i)}
\end{align*}
which in the simply-connected case appears in \cite{Dub_Comm,Dub_Euler}. We can disintegrate w.r.t. $(\gamma_1,\dots,\gamma_{n-1})$:
\begin{align*}
d\mu_{(\Sigma,X_1,\dots,Y_n)}(\gamma_1,\dots,\gamma_n)&=\ind_{{\mc P}_0^s}\left(e^{\frac c2\sum_{j=2}^{n-1}\nu_\Sigma(\gamma_j;\cup_{i<j}\gamma_i)}\prod_{i<n}d\mu_{(\Sigma,X_i,Y_i)}(\gamma_i)\right)
\left(e^{\frac c2\nu_\Sigma(\gamma_j;\cup_{i<n}\gamma_n)}d\mu_{(\Sigma,X_n,Y_n)}(\gamma_n)\right)\\
&=\ind_{{\mc P}_0^s}\left(e^{\frac c2\sum_{j=2}^{n-1}\nu_\Sigma(\gamma_j;\cup_{i<j}\gamma_i)}\prod_{i<n}d\mu_{(\Sigma,X_i,Y_i)}(\gamma_i)\right)
d\mu_{(\Sigma\setminus\cup_{i<n}\gamma_i,X_n,Y_n)}(\gamma_n)%
\end{align*}
where the second line is by the restriction property; hence the disintegration of the measure on $n$-paths w.r.t. $n-1$ of them is proportional to a canonical measure on the last path on the random surface $\Sigma\setminus\cup_{i<n}\gamma_i$. By integrating separately over $\gamma_n$, we obtain in particular
$${\mc Z}_0(\Sigma,X_1,\dots,Y_n)=\int
\ind_{{\mc P}_0^s}\left(e^{\frac c2\sum_{j=2}^{n-1}\nu_\Sigma(\gamma_j;\cup_{i<j}\gamma_i)}\right)
{\mc Z}_0(\Sigma\setminus\cup_{i<n}\gamma_i,X_n,Y_n)
\prod_{i<n}d\mu_{(\Sigma,X_i,Y_i)}(\gamma_i)
$$

In turn we know how to disintegrate $d\mu_{(\Sigma\setminus\cup_{i<n}\gamma_i,X_n,Y_n)}(\gamma_n)$ w.r.t. $\gamma_n^\tau$, the path $\gamma_n$ up to first exit of a semidisk neighborhood of $X_n$ (see \eqref{eq:disintegr}). Let $z$ be a local coordinate at $X_n$; we find again that
$$t\mapsto \exp(c(\nu_\Sigma(z^{-1}(\gamma^{t\wedge\tau});\Sigma\setminus \tilde D^+)
-\nu_\H(\gamma^{t\wedge\tau};\H\setminus z(\tilde D^+)))){\mc Z}_0(\Sigma_{t\wedge\tau},X_1,\dots,X_{n-1},\gamma_{t\wedge\tau},Y_1,\dots,Y_n)$$
is a bounded martingale under the standard chordal SLE measure in the upper half-plane (we work under the assumption that ${\mc Z}_0$ is locally bounded). As in the $n=1$ case, this implies that
$$\Delta_{2,1}({\mc Z}_0(\Sigma,X_1,\dots,X_n,Y_1,\dots,Y_n)s)=0$$
where $\Delta_{2,1}$ is the canonical differential operator corresponding to a perturbation at $X_n$ (and as before $s$ is the reference section of ${\mc L}^{\otimes c}$).

\appendix

\section{Analytic surgery}\label{sec:surgery}

In this appendix, we discuss the analytic surgery results of Forman (\cite{Forman_surgery}) and Burghelea-Friedlander-Kappeler (\cite{BFK}). These pertain to the effect of cutting and pasting manifolds on the determinants of the Laplacians (or elliptic operators), and can be understood in terms of Dirichlet space decompositions. We give a new proof based on probabilistic arguments, which extends arguments presented in \cite{Dub_SLEGFF}. This will be needed to analyze the variation of determinants under deformations prescribed by Witt algebra elements \eqref{eq:lndef}. 

Let $(M,g)$ be a Riemannian manifold, possibly with boundary (and boundary conditions). Consider a codimension one manifold (a simple loop in dimension 2) $\delta$ in $M$, that splits $M$ in two manifolds $M_l$, $M_r$ with induced metrics. The Neumann jump operator corresponding to this situation is defined as follows:\\
given $\phi\in C^\infty(\delta)$, $\phi_l$ (resp. $\phi_r$) is the solution of the Dirichlet problem $(\phi_l)_{|\delta}=\phi$ on $M_l$ (resp. $M_r$). Then:
\begin{equation}\label{eq:neujump}
N\phi=\partial_{n_l}\phi_l+\partial_{n_r}\phi_r
\end{equation}
where $\partial_{n_l}$, $\partial_{n_r}$ are outward pointing normal derivatives in $M_l$, $M_r$. The jump operator is a positive, first-order pseudodifferential operator on  $C^\infty(\delta)$. In the case where the Laplacian on $M$ has a non-trivial kernel, so has $N$, and these kernels are canonically identified. 

It turns out that the $\zeta$-function associated with $N$ has a meromorphic continuation to a neighborhood of $0$, which defines ${\det}_\zeta(N)$.
We can now phrase:

\begin{Thm}\label{Thm:surgery}
The following relation holds:
$${\det}_\zeta(-\Lap)=C_\delta{\det}_\zeta(-\Lap)_l{\det}_\zeta(-\Lap)_r{\det}_\zeta(N)$$
where the constant $C_\delta$ depends only on the metric in a neighborhood of $\delta$ in $M$ (if $\Lap$ has a trivial kernel).
\end{Thm}
\begin{proof}
The idea is that $\log({\det}_\zeta(-\Lap)/{\det}_\zeta(-\Lap)_l{\det}_\zeta(-\Lap)_r)$ and $\log({\det}_\zeta(N))$ correspond to two ways of counting loops on $M$ that cross $\delta$: by duration and by local time at $\delta$.

Consider the loop measure $\nu$ relative to the Laplacian on $M$. It restricts to the loop measures on $M_l$, $M_r$. Expressing $\zeta$-functions in terms of $\nu$ (see \eqref{eq:zetaloop}),
$$\Gamma(s)(\zeta-\zeta_l-\zeta_r)(s)=\int (1-\ind_{\gamma\subset M_l}-\ind_{\gamma\subset M_r})
T^s(\gamma)d\nu(\gamma)=\int \ind_{\gamma\cap\delta\neq \varnothing}T^s(\gamma)d\nu(\gamma).$$
Let us identify smoothly a neighborhood of $\delta$ in $M$ with $\delta\times(-\eps,\eps)$. A Brownian path $\gamma_t$ in the neighborhood of $\delta$ projects on $(-\eps,\eps)$ as a semimartingale, that has a local time at 0 denoted by $\ell_t$. 
One can modify the identification $\delta\times(-\eps,\eps)$ so that the quadratic variation of this process is that of linear Brownian motion running at speed 2 (at $\eps=0$).
Let $u\mapsto\tau_u$ be the right continuous inverse of $t\mapsto\ell_t$. Clearly, $u\mapsto\gamma_u=\gamma_{\tau_u}$ is a c\`adl\`ag Markov process on $\delta$. Let us identify the generator of this process. If $\phi,\phi_l,\phi_r$ are as above \eqref{eq:neujump}, and $\phi_{lr}$ is the function on $M$ that restricts to $\phi_l,\phi_r$, then, by the It\^o-Tanaka formula
:
$$t\rightarrow \phi_{lr}(\gamma_t)+\int_0^t (N\phi)(\gamma_t)d\ell_t$$ 
is a local martingale. %
After a time change, it means that the generator of $(\gamma_u)$ is $(-N)$ (Dynkin's formula). %

Hence the $\zeta$-function of $N$: $\zeta_N(r)=\frac{1}{\Gamma(r)}\int_0^\infty \Tr(e^{-uN})u^{r-1}du$ can also be understood in terms of loops. More precisely:
$$\zeta_N(r)=\frac 1{\Gamma(r)}\int \ell(\gamma)^r \ind_{\gamma\cap\delta\neq\varnothing}d\nu(\gamma)$$
where $\ell(\gamma)$ is the local time at $\delta$ of the loop $\gamma$. This corresponds to rerooting loops that cross $\delta$ uniformly in local time.

Now we can split loops that cross $\delta$ into those that stay in a neighborhood of $\delta$, say $\delta^\eps\simeq\delta\times(-\eps,\eps)$, and the macroscopic loops that exit $\delta^\eps$. The mass of such macroscopic loops in $\nu$ is finite (using here the assumption that $\Lap$ has a trivial kernel). Hence:
$$\int \ell^r\ind_{\gamma\cap\delta\neq\varnothing,\gamma\subsetneq \delta^\eps} d\nu(\gamma)$$
and 
$$\int T^s\ind_{\gamma\cap\delta\neq\varnothing,\gamma\subsetneq \delta^\eps} d\nu(\gamma)$$
are convergent at $r=s=0$, where they take the same value: $\nu\{\gamma:\gamma\cap\delta\neq\varnothing,\gamma\subsetneq \delta^\eps\}$. The loops contained in $\delta^\eps$ contribute to the local constant $C_\delta$. This concludes the proof.
\end{proof}

This can be easily modified if $\Lap$ (and hence $N$) has a nontrivial kernel, by using ${\det}'(-\Lap)$, ${\det}'(N)$ etc\dots

Let us illustrate this in the simple situation of the real line splitting  $\C$ in two half-planes $\H$, $-\H$ (this is not compact, but gives the correct local behaviour). The real and imaginary part of the Brownian motion $X$ in $\C$ are independent linear Brownian motions. Taking the trace of the complex Brownian motion on the real line, one gets the real process $Y_u=\Re X_{\tau_u}$ where $\tau$ is the right-continuous inverse of the local time at 0 of $\Im X$. Hence $Y$ is obtained as the subordination of a linear Brownian motion by an independent $1/2$-stable subordinator, so that $Y$ is a symmetric 1-stable L\'evy process, i.e. a Cauchy process. The infinitesimal generator of $Y$ is known to be the first-order pseudo-differential operator $-(-\Lap)^{1/2}$, where $\Lap$ is the Laplacian on $\R$. The transition densities of $Y$ are:
$$q_t(x,y)dy=\frac{\pi^{-1}tdy}{t^2+(x-y)^2}$$
giving a nice short-time expansion on the diagonal, which replaces \eqref{eq:pleijel}.

\section{A variation formula}\label{Sec:varform}

In this appendix, we evaluate $(L_ns_\zeta)/s_\zeta$, which is crucial to the construction of Virasoro representations (with general $c$, Theorem \ref{Thm:Virrep}). Notation is as in the beginning of  \ref{ssec:Virgen}.

This is a rather concrete problem: we have to evaluate the variation of an explicit functional of the metric (the $\zeta$-regularized determinant) under an explicit deformation of the metric (encoded by $\ell_n$); since this deforms the complex structure and not the actual metric, we compensate by the counterterm $s_\zeta(\H_t,\tilde g_t)$. The argument proceeds in a few steps.
 
{\bf  Surgery.}
 
First we observe that $s_\zeta$ does not depend on the choice of local coordinates, or for that matter the position of marked points. Consequently $L_ns_\zeta=0$ if $n\geq -1$. For $n\leq -2$, we need to evaluate
$$\lim_{t\rightarrow 0}\frac{s_\zeta(\Sigma_t,g_t)s_\zeta(\H_0,\tilde g_0)}{s_\zeta(\Sigma_0,g_0)s_\zeta(\H_t,\tilde g_t)}$$
Recall that this ratio depends only on the complex structures and not the choices of metrics. We use analytic surgery to work on a fixed contour. With notations as above, let $\gamma^+=z^{-1}(C(0,r))$, so that $\Sigma_t,\H_t$ are identified inside $\gamma^+$ and $\Sigma_t$, ${\H}_t$ are constant outside of $\gamma^+$. Let $N_t$, $\tilde N_t$ be the Neumann jump operators along $\gamma^+$ on $\Sigma_t$, ${\H}_t$ respectively. By applying analytic surgery (Theorem \ref{Thm:surgery}) to all four Laplacian determinants and noticing pairwise cancellations, we have:
 \begin{equation}\label{eq:Neuratio}
 \left(
 \frac{s_\zeta(\Sigma_t,g_t)s_\zeta(\H_0,\tilde g_0)}{s_\zeta(\Sigma_0,g_0)s_\zeta(\H_t,\tilde g_t)}\right)^{-2}
 =\frac{\det_\zeta(N_t)\det_\zeta(\tilde N_0)}{\det_\zeta(N_0)\det_\zeta(\tilde N_t)}
 \end{equation}
The Neumann operators are pseudodifferential operators built from Poisson kernels; we then study the variation of these kernels. 

{\bf Variation of the Dirichlet-to-Neumann operator.}

We start from a semidisk neighborhood of $X$ in $\Sigma_0$, to which are associate a Poisson and a Dirichlet-to-Neumann operator on the semicircle. Then we deform the complex structure inside the disk (according to $\ell_n$); the goal here is to quantify the resulting variation of these operators.

Let us identify a neighborhood of $X$ in $\Sigma_0$ with a semidisk via the local coordinate $z$. For notational simplicity we assume $r=1$. The Poisson operator extending real functions on $\gamma^+$ to harmonic functions on $D^+(0,1)$ vanishing on $(-1,1)$ is:
\begin{equation}\label{eq:Poissonop}
(Pf)(w)=\oint_\gamma f(z)\Re\left(\frac{z+w}{z-w}\right)\frac{dz}{2i\pi z}
\end{equation}
where $f$ is extended to the lower unit semicircle by $f(\bar z)=-f(z)$. We can represent the deformation by a variation of the complex structure of the unit disk (this is the same deformation for $\Sigma_t$ and $\H_t$). Let us choose $\phi_t$ mapping conformally the deformed disk to the standard unit disk so that $\phi_t(\bar z)=\overline{\phi_t(z)}$ (this leaves one degree of freedom). Notice that by Schwarz reflection, $\phi_t$ extends analytically across the unit circle. Then by a change of variable:
\begin{equation}\label{eq:Poissondefo}
(P_tf)(w)=\oint_\gamma f(\phi_t^{-1}(z))\Re\left(\frac{z+\phi_t(w)}{z-\phi_t(w)}\right)\frac{dz}{2i\pi z}
=\oint_\gamma f(z)\Re\left(\frac{\phi_t(z)+\phi_t(w)}{\phi_t(z)-\phi_t(w)}\right)\frac{d\phi_t(z)}{2i\pi\phi_t(z)}
\end{equation}
where $P_t$ denotes the Poisson operator (for $w$ close to $\gamma$) relative to the deformed complex structure. 

Now by definition of the deformation one can find a function $\psi_t$ which is smooth in $t$ and analytic from the deformed disk to a neighborhood of the unit disk, with expansion: $\psi_t(z)=z+t z^{n+1}+o(t)$. By the argument principle, it is injective on the disk for small $t$. Then $\phi_t =h_t\circ\psi_t$ for some analytic function $h_t$ depending smoothly on $t$. Since $\phi_t$ maps the unit circle to itself, we get $h_t(z)=z-tz^{-n+1}+o(t)$ (recall that $n\leq -2$) and 
$$\phi_t(z)=z(1+t(z^{n}-z^{-n}))+o(t)$$
near the unit circle. Together with \eqref{eq:Poissondefo}, this gives an expression for the derivative ${\frac{d}{dt}}_{|t=0}P_t$:
$${\frac{d}{dt}}_{|t=0}
(P_tf)(w)=\oint f(z)\Re(Q(z,w))\frac{dz}{2i\pi z}$$
where
\begin{align*}
Q(z,w)&=\frac{z+w}{z-w}\left(\frac{z^{n+1}-z^{1-n}+w^{n+1}-w^{1-n}}{z+w}-\frac{z^{n+1}-z^{1-n}-w^{n+1}+w^{1-n}}{z-w}+n(z^n+z^{-n})\right)\\
&=\frac{1}{z-w}\left(z^{n+1}-z^{1-n}+w^{n+1}-w^{1-n}-(z+w)((z^n+\cdots+w^n)+\frac{z^{2-n}+\cdots+w^{2-n}}{zw})
+n(z^n+z^{-n})(z+w)\right)
\end{align*}
We check that the singularity at $z=w$ is removable and that $Q(z,w)\in\C[z,z^{-1},w,w^{-1}]$; consequently, $P_t:L^2(\gamma^+)\rightarrow L^2_{loc}(D^+\setminus\{0\})$ and $\partial_nP_t:L^2(\gamma^+)\rightarrow L^2(\gamma^+)$ are smooth kernel operators ($\partial_n$ is the inward pointing normal derivative). We have for $w\in\gamma^+$:
$$(\partial_nP_tf)(w)
=\oint_\gamma f(z)\Re\left(-w\phi'_t(w)\frac{2\phi_t(z)}{(\phi_t(z)-\phi_t(w))^2}\right)\frac{d\phi_t(z)}{2i\pi\phi_t(z)}
=\oint_\gamma f(z)\Re\left(-zw\frac{\phi'_t(z)\phi_t'(w)}{(\phi_t(z)-\phi_t(w))^2}\right)\frac{dz}{i\pi z}
$$
since $\phi'(z)z/\phi_t(z)$ is real on the circle; $\partial_nP_t$ is symmetric w.r.t. the Lebesgue measure on the circle. Then the variation of the Dirichlet-to-Neumann operator $\partial_nP_t$ is given by
$${\frac{d}{dt}}_{|t=0}(\partial_nP_tf)(w)=\oint f(z)\Re\left(R(z,w)\right)\frac{dz}{i\pi z}$$
where
$$
R(z,w)=\frac{-zw}{(z-w)^2}\left((n+1)(z^n+w^n)+(n-1)(z^{-n}+w^{-n})-2\frac{z^{n+1}-w^{n+1}}{z-w}+2\frac{z^{1-n}-w^{1-n}}{z-w}\right)$$
We have (recall that $n\leq -2$)
\begin{align*}
2\frac{z^{1-n}-w^{1-n}}{z-w}-(1-n)(z^{-n}+w^{-n})\\
=(z^{-n}+w^{-n})+(z^{-n-1}w+w^{-n-1}z)+\cdots+(w^{-n}+z^{-n})-(1-n)(z^{-n}+w^{-n})\\
=(w-z)(z^{-n-1}-w^{-n-1})+(w^2-z^2)(z^{-n-2}-w^{-n-2})+\cdots+(w^{-n-1}-z^{-n-1})(z-w)
\end{align*}
and
we check that if $|z|=|w|=1$, 
$$\overline{\frac{-zw}{(z-w)^2}\left((n+1)(z^n+w^n)-2\frac{z^{n+1}-w^{n+1}}{z-w}\right)}=
\frac{-zw}{(z-w)^2}\left((n-1)(z^{-n}+w^{-n})+2\frac{z^{1-n}-w^{1-n}}{z-w}\right)$$

Finally we obtain for $z,w\in\gamma$:
\begin{align*}
\Re(R(z,w))&=\Re\left(2zw\left(\sum_{k=0}^{-n}\frac{z^k-w^k}{z-w}\frac{z^{-n-k}-w^{-n-k}}{z-w}
\right)\right)\\
&=\Re\left(2zw\sum_{i+j+k+l=-n-2}z^iw^jz^kw^l\right)=\Re\left(2zw\sum_{i+j=-n-2}(i+1)(j+1)z^iw^j\right)
\end{align*}
viz. an explicit expression for the kernel of ${\frac{d}{dt}}_{|t=0}\partial_nP_t$.

{\bf A trace evaluation.}

We have studied $K={\frac{d}{dt}}_{|t=0}\partial_nP_t:L^2(\gamma^+)\rightarrow L^2(\gamma^+)$, a trace class integral operator with smooth kernel (also denoted by $K$) given by
$$K(w,z)=\frac 1\pi\Re(R(z,w)+R(\bar z,w))$$
w.r.t. length on $\gamma^+$ (here we need to write the kernel on the semicircle $\gamma^+$ and not on the circle $\gamma$). Let $T:L^2(\gamma^+)\rightarrow L^2(\gamma^+)$ be another integral operator with bicontinuous kernel also denoted by $T$. We wish to evaluate
$$\Tr(TK)=\int_{\gamma^+}(TK)(z,z)\frac{dz}{iz}=\int_{\gamma^+}\int_{\gamma^+}T(z,w)K(w,z)\frac{dz}{iz}\frac{dw}{iw}$$
Remark that for $m\geq 0,k\in\Z$, by differentiating the Poisson kernel \eqref{eq:Poissonop} we see:
\begin{align*}
\partial_x^{m+1}(Pf)(0)&=(m+1)!\oint f(z)\Re(z^{-m-1})\frac{dz}{i\pi z}=0\\
\partial_x^m\partial_y(Pf)(0)&=-(m+1)!\oint f(z)\Im(z^{-m-1})\frac{dz}{i\pi z}
\end{align*}
where $w=x+iy$ (since $\partial_x^i\partial_y^{j+2}Pf=-\partial_x^{i+2}\partial^j_yPf$, this evaluates all partial derivatives of $Pf$ at $0$). Consequently, if $(z,w)\mapsto f(z,w)$ is biharmonic on $D(0,1)^2$ (with $f(\bar z,w)=f(z,\bar w)=-f(z,w)$), $i,j\geq 0$, we have:
\begin{align*}
\oint\oint f(z,w)\Re(z^{i+1}w^{j+1})\frac{dw}{i\pi w}\frac{dz}{i\pi z}
&=\oint\oint f(z,w)\Re(z^{i+1})\Re(w^{j+1})\frac{dw}{i\pi w}\frac{dz}{i\pi z}-\oint\oint f(z,w)\Im(z^{i+1})\Im(w^{j+1})\frac{dw}{i\pi w}\frac{dz}{i\pi z}\\
&=-\frac{1}{(i+1)!(j+1)!}\partial_{x_1}^i\partial_{y_1}\partial_{x_1}^{j}\partial_{y_2}f(0,0)
\end{align*}
where $z=x_1+iy_2$, $w=x_2+iy_2$. If $h(u)=g(u,u)$, $(u,v)\mapsto g(u,v)$ regular enough, we have trivially:
$$\frac 1{k!}h^{(k)}(u)=\sum_{i+j=k}\frac{\partial_u^i\partial_v^j g}{i!j!}(u,u)$$
and consequently if we set $h(x)=\partial_{y_1}\partial_{y_2}f(x,x)$, we obtain:
$$
\oint\oint f(z,w)\Re(R(z,w))\frac{dw}{i\pi w}\frac{dz}{i\pi z}
=-\frac{2}{(-n-2)!}\partial_x^{-n-2}h(0)
$$
In conclusion, if $T$ has biharmonic extension $f$ to $D^2$ with $f(\bar z,w)=f(z,\bar w)=-f(z,w)$, we have
\begin{equation}\label{eq:tracevar}
\Tr(TK)=-\frac{\pi}{(-n-2)!}\partial_x^{-n-2}h(0)
\end{equation}
Notice the factor 2 coming from $\int_{\gamma^+}\oint=\frac 12\oint\oint$, given the symmetries.

{\bf Conclusion.}

Now we are in position to evaluate the derivative of \eqref{eq:Neuratio}. We have $K=\frac{d}{dt}N_t=\frac{d}{dt}\partial_nP_t$ (as the outside of the disk is unchanged) and the RHS is an integral operator with smooth kernel. More generally, $P_t-P_0$ has kernel (w.r.t. to length on the unit circle):
$$(z,w)\mapsto \frac 1\pi\Re\left(\frac{\phi_t(z)+\phi_t(w)}{\phi_t(z)-\phi_t(w)}\frac{z\phi_t'(z)}{\phi_t(z)}-\frac{z+w}{z-w}\right)$$
which is smooth as the singularity at $z=w$ is removable. Hence $N_t-N_0$ has a smooth kernel and $N_t^{-1}(z,w)=-G_{\Sigma_t}(z,w)$ for $z,w$ on the unit semicircle. Consequently $N_0^{-1}(N_t-N_0)$ is trace class on $L^2(\gamma^+)$ and
$${\det}_\zeta(N_t)={\det}_\zeta(N_0){\det}_F(\Id+N_0^{-1}(N_t-N_0))$$
where $\det_F$ is a Fredholm determinant. It follows that
$${\frac{d}{dt}}_{|t=0}\log{\det}_\zeta(N_t)=\Tr\left({N_0^{-1}\frac{d}{dt}}_{|t=0}N_t\right)$$
and the same statement holds for $\tilde N$ (with the natural definition for the trace of an operator with continuous kernel). This can be seen directly from the definition of ${\det}_\zeta$ and the Duhamel expansion of the semigroup generated by $N_t$ for small $t$. Consequently, 
$${\frac{d}{dt}}_{|t=0}\frac{\det_\zeta(N_t)\det_\zeta(\tilde N_0)}{\det_\zeta(N_0)\det_\zeta(\tilde N_t)}=
\Tr\left((N_0^{-1}-\tilde N_0^{-1}){\frac{d}{dt}}_{|t=0}N_t\right)$$
Set $f(z,w)=(N_0^{-1}-\tilde N_0^{-1})(z,w)=-G_{\Sigma_0}(z,w)+G_{\H_0}(z,w)$, which extends biharmonically to $D^+(0,1)^2$, and $T:L^2(\gamma^+)\rightarrow L^2(\gamma^+)$ the corresponding integral operator. Then by surgery (Theorem \ref{Thm:surgery})
$$\frac{L_ns_\zeta}{s_\zeta}=-\frac 12\Tr(TK)$$
Now we can use the evaluation \eqref{eq:tracevar}. For $z,w\in (-1,1)$, we have
$$\partial_{n_z}\partial_{n_w}f(z,w)=-H_{\Sigma_0}(z,w)+H_{\H_0}(z,w)$$
where $H$ is the Poisson excursion kernel \eqref{eq:Poissonexcdef} and $\pi\partial_{n_z}\partial_{n_w}f(z,z)=-\frac 16S_{\Sigma_0}(z)$ where $S$ is the Schwarzian connection \eqref{Sconn}. We obtain ($n\leq -2$):
$$\frac{L_ns_\zeta}{s_\zeta}=\frac 1{(-n-2)!}\partial_x^{-n-2}\frac{S_{\Sigma_0}(x)}{12}$$
Remark that $S_\Sigma(X)$ is defined in terms of the local coordinate at $X$; as a Schwarzian connection, it depends on this local coordinate through its 3-jet. From the expression of the Bergman connection in terms of theta functions, it is clear that $(\Sigma,X,z)\mapsto S_\Sigma(X)$, where $z$ is a $3$-jet of local coordinate at $X$, is smooth on ${\mc T}_3$.

\bibliographystyle{abbrv}
\bibliography{biblio}

-----------------------

\noindent Columbia University\\
Department of Mathematics\\
2990 Broadway\\
New York, NY 10027

\end{document}